\documentclass[a4paper,10pt,reqno]{amsart}

\usepackage{amsmath,mathtools}
\usepackage[authoryear,round]{natbib}
\usepackage{tikz}
\usetikzlibrary{arrows.meta}
\usepackage{petkov-preprint}

\numberwithin{equation}{section}

\newcommand{\Prob}[1]{\mathbb{P}\!\left(#1\right)}
\newcommand{\Exp}[1]{\mathbb{E}\!\left[#1\right]}

\newcommand{\displayheading}[1]{\par\addvspace{0.68\baselineskip}\noindent{\scshape #1.}\par\nobreak\smallskip\noindent}

\theoremstyle{plain}
\newtheorem*{maintheorem}{Main theorem}
\newtheorem{theorem}{Theorem}[section]
\newtheorem{proposition}[theorem]{Proposition}
\newtheorem{lemma}[theorem]{Lemma}
\newtheorem{corollary}[theorem]{Corollary}

\theoremstyle{remark}
\newtheorem{remark}[theorem]{Remark}

\BeforeBeginEnvironment{maintheorem}{\Needspace{6\baselineskip}}
\BeforeBeginEnvironment{theorem}{\Needspace{4\baselineskip}}
\BeforeBeginEnvironment{lemma}{\Needspace{4\baselineskip}}
\BeforeBeginEnvironment{proposition}{\Needspace{4\baselineskip}}
\BeforeBeginEnvironment{corollary}{\Needspace{4\baselineskip}}
\BeforeBeginEnvironment{remark}{\Needspace{3\baselineskip}}
\BeforeBeginEnvironment{proof}{\Needspace{3\baselineskip}}

\hypersetup{
  pdftitle={A Full-Sequence Quantitative Gap Between the Chromatic and Cochromatic Numbers of a Random Graph},
  pdfauthor={Samuil Petkov},
  pdfsubject={Random graphs; chromatic and cochromatic numbers},
  pdfkeywords={random graph, chromatic number, cochromatic number, second moment method, configuration model},
  pdfcreator={LaTeX}}

\title[Chromatic and cochromatic numbers]{A Full-Sequence Quantitative Gap
Between the Chromatic and Cochromatic Numbers of a Random Graph}
\author[Samuil Petkov]{Samuil Petkov}
\address{\'Ecole normale sup\'erieure, Universit\'e PSL, Paris, France;
\texttt{samuil.petkov@ens.psl.eu}}
\date{30 August 2026}
\subjclass[2020]{Primary 05C80; Secondary 05C15, 60C05}
\keywords{random graph, chromatic number, cochromatic number, second moment method, configuration model}

\begin{document}

\begin{abstract}
The cochromatic number \(\zeta(G)\) of a graph \(G\) is the least number of parts in a partition of \(V(G)\) into cliques and independent sets, so \(\zeta(G)\le\chi(G)\). Erd\H{o}s and Gimbel asked whether \(\chi(G_n)-\zeta(G_n)\to\infty\) with high probability for \(G_n\sim G(n,1/2)\). Heckel answered this affirmatively for about \(95\%\) of all integers \(n\) and conjectured that the natural order of the difference along the full sequence is \(n/(\ln n)^3\). We prove that, as \(n\to\infty\) through all integers, \(\chi(G_n)-\zeta(G_n)\ge0.063\,n/(\ln n)^3\) with high probability. The failure probabilities are summable, so under any coupling of the graphs \(G_n\) the bound holds almost surely for all sufficiently large \(n\).

The proof counts partitions of the vertex set into classes of five consecutive sizes, each class declared independent or complete. The two choices per class outweigh the entropy lost by restricting the class sizes, so the expected number of such partitions is large already at a number of classes that lies order \(n/(\ln n)^3\) below the first-moment lower bound for the chromatic number, uniformly in the rounding phase. A second-moment argument, which controls all overlaps between two such partitions, including whole shared classes and large intersections of unequal classes, bounds the probability that such a partition exists from below. Vertex-exposure concentration, together with a coloring of the leftover vertices, then raises this to high probability at a cost of \(o(n/(\ln n)^3)\) extra classes.
\end{abstract}

\maketitle

\section*{Introduction}
\label{sec:introduction}

A \emph{cocoloring} of a graph $G$ is a partition of $V(G)$ into nonempty
classes, each inducing either an edgeless graph or a complete graph. The least
number of classes in such a partition is the \emph{cochromatic number}
$\zeta(G)$. Since an ordinary proper coloring is a cocoloring using only
independent-set classes, one always has $\zeta(G)\le\chi(G)$.

Write $[n]:=\{1,\ldots,n\}$, and let $G_n\sim G(n,1/2)$ be the labeled random
graph on $[n]$. All logarithms are natural unless a base is displayed.
The leading asymptotics of $\chi(G_n)$ and $\zeta(G_n)$ agree, so they do
not determine whether the difference between the two numbers grows. Erd\H{o}s
and Gimbel asked whether
\[
  \chi(G_n)-\zeta(G_n)\longrightarrow\infty
\]
with high probability \citep[p.~263]{erdos-gimbel-1993}. The question was
restated by Gimbel \citep[Section~7.4]{gimbel-2016} and is cataloged as
Erd\H{o}s Problem~625 \citep{bloom-erdos625}. We prove a quantitative bound
along the full sequence of integers.

\Needspace{5\baselineskip}
The expected number of independent sets of size $k$ is
\[
  \mu_k=\binom nk2^{-\binom k2}.
\]
Solving $\mu_k=1$ asymptotically gives a real-valued center
$\alpha_0(n)$ \citep[p.~420]{bollobas-erdos-1976}. The phase
$\delta_n$ is its fractional part, which records where this center lies
between consecutive integer class sizes. Both quantities are defined
precisely in the main theorem below.

\begin{maintheorem}[Uniform lower bound and phase refinement]
For every integer \(n\ge2\), let \(G_n\sim G(n,1/2)\) be the random simple
graph on \([n]\), with edges present independently with probability \(1/2\).
All logarithms without a displayed base are natural, and all limits are
as \(n\to\infty\) through the integers. Set
\[
  \alpha_0(n):=
  2\log_2 n-2\log_2\log_2 n+2\log_2(\mathrm e/2)+1,
  \qquad
  \delta_n:=\{\alpha_0(n)\}
  =\alpha_0(n)-\lfloor\alpha_0(n)\rfloor.
\]
Thus $\delta_n\in[0,1)$ is deterministic.  For $0\le\delta\le1$, define
\[
  S_+:=\{i\in\mathbb Z:i\ge-1\},\qquad
  S_5:=\{2,3,4,5,6\},\qquad
  T_\delta:=1+\frac{2}{\log 2}-\delta.
\]
Then $T_\delta\in(2,5)$.  For $S\in\{S_+,S_5\}$ and $T\in(2,5)$, put
\[
  \mathcal F_S(T):=
  \inf_{\lambda\in\mathbb R}
  \left\{
    \log\!\sum_{i\in S}
      \exp\!\left(\lambda i-\frac{\log 2}{2}i^2\right)
    -\lambda T
  \right\},
  \qquad
  A_5(\delta):=
  \log 2-\mathcal F_{S_+}(T_\delta)+\mathcal F_{S_5}(T_\delta).
\]
The sum inside the logarithm defining $\mathcal F_{S_+}(T)$ converges for every
$\lambda\in\mathbb R$; for $T\in(2,5)$, each displayed infimum is a finite real
and is uniquely attained.

Then $A_5$ is continuous, positive and nonconstant on $[0,1]$, so
\[
  c_5:=\frac{(\log 2)^2}{4}\min_{0\le\delta\le1}A_5(\delta)>0.
\]
For every constant \(c<c_5\),
\[
  \mathbb P\!\left(
    \chi(G_n)-\zeta(G_n)
    \ge
    c\,\frac{n}{(\log n)^3}
  \right)
  \longrightarrow 1.
\]
Moreover, there is a deterministic sequence of nonnegative reals
$(\varepsilon_n)_{n\ge2}$ with $\varepsilon_n\to0$ such that
\[
  \mathbb P\!\left(
    \chi(G_n)-\zeta(G_n)
    \ge
    \left[
      \frac{(\log 2)^2}{4}A_5(\delta_n)-\varepsilon_n
    \right]
    \frac{n}{(\log n)^3}
  \right)
  \longrightarrow1.
\]
\end{maintheorem}

The first bound holds for every fixed coefficient below \(c_5\). The
refinement keeps the coefficient \(A_5(\delta_n)\), which depends on where
the independent-set center lies between neighboring integers. Since
\(A_5(\delta_n)\) is at least the minimum of \(A_5\) on \([0,1]\) and
\(\varepsilon_n\to0\), the refined bound implies the uniform bound for every
\(c<c_5\). The constant \(c_5\) is bounded from below in
Section~\ref{sec:final-assembly} by Lemma~\ref{lem:phase-gain-sharp}
(Section~\ref{the-five-size-signed-first-moment-advantage}), which shows that
\(A_5>\log(169/100)\) on \([0,1]\), so
\[
  c_5>\frac{(\log 2)^2}{4}\log\!\left(\frac{169}{100}\right)
  =0.0630268508\ldots.
\]
Since \(n/(\log n)^3\to\infty\), both forms answer the question of
Erd\H{o}s and Gimbel along the full sequence of integers.
Remark~\ref{rem:phase-remainder-rate} shows that one may take
\(\varepsilon_n=O(\log\log n/\log n)\).
Corollary~\ref{cor:phase-failure-probability} bounds the failure probability
by a summable sequence, and Corollary~\ref{cor:almost-sure-gap} deduces the
almost-sure statement for each fixed coupling.

\paragraph{Origin of the coefficient.}
Write a class size as $\alpha-i$ with $\alpha=\lfloor\alpha_0(n)\rfloor$; we
call $i$ the deficit. Up to a term affine in $i$, which the totals of classes
and vertices fix,
$-\log\bigl((\alpha-i)!\,2^{\binom{\alpha-i}2}\bigr)
=-\frac{\log 2}2i^2+O(1/\log n)$
for fixed $i$ (Section~\ref{continuous-profile-roots}). Thus
$\mathcal F_S(T)$ is the optimal entropy per class, less this quadratic cost,
over deficit distributions on $S$ with mean $T$, and $T_\delta$ is the mean
deficit at the first-moment root. Declaring each class independent or
complete adds $\log 2$ per class, so $A_5(\delta)$ is the net gain per class
of marked five-size partitions over colorings. Dividing the total gain, about
$A_5(\delta_n)\frac{\log 2}2\frac n{\log n}$, by the slope
$\frac2{\log 2}(\log n)^2$ of the counting exponent in the number of classes
gives the coefficient in the theorem
(Section~\ref{the-five-size-signed-first-moment-advantage}).

When
$\alpha_0(n)$ crosses an integer, the size cutoff changes, and with it the
optimization problem used to choose the profile and the first-moment root
of that problem, the point where the optimized counting exponent vanishes.
An argument that is valid for most integers can therefore miss some phases.
We keep the phase explicit and make every estimate uniform up to both
endpoints of the phase interval, including for integers on either side of
a change of cutoff.

\subsection*{Relation to previous work}

The first-order asymptotic for the chromatic number of a dense random graph was
proved by Bollob\'as \citep{bollobas-1988}, following the early work of
Grimmett and McDiarmid \citep{grimmett-mcdiarmid-1975}; later refinements
include \citet{mcdiarmid-1990}, \citet{panagiotou-steger-2009}, and
\citet{heckel-2018}. The cochromatic number belongs to the theory of generalized
chromatic numbers associated with hereditary graph properties
\citep{scheinerman-1992,bollobas-thomason-1995}.

For the difference $\chi(G_n)-\zeta(G_n)$, Heckel and, independently,
Steiner showed how a high-probability upper bound on the gap would imply
concentration of the chromatic number
\citep{heckel-2024-question,steiner-2024}. Heckel subsequently proved a
near-linear lower bound for a phase-dependent set containing approximately
$95\%$ of the integers and conjectured that the natural full-sequence scale is
$n/(\log n)^3$ \citep{heckel-2025-difference}.  We prove a lower bound of
this order along the full sequence $n\to\infty$, with an explicit
phase-resolved coefficient.
We use two ingredients from Heckel's argument: the additional choices
obtained by declaring each class independent or complete, and the strategy
for raising a small success probability to high probability. Beyond these,
we allow five consecutive class sizes, which makes the profile optimization
uniform in the phase, and our second-moment estimate counts compatible class
declarations exactly and controls all overlaps between two partitions,
including whole shared classes.

Appendix~\ref{app:lean-reproducibility} records the companion Lean theorem and its reproducibility materials.

\Needspace{8\baselineskip}
\subsection*{Proof strategy}

We construct deterministic integers $k_\chi^-$ and $k_{\mathrm{co}}$, and a
deterministic $a_n=o(n/(\log n)^3)$, for which
\begin{enumerate}[label=(\Alph*)]
\item
\[
  k_\chi^- - k_{\mathrm{co}}-a_n
  =\left[\frac{(\log 2)^2}{4}A_5(\delta_n)+o(1)\right]
    \frac{n}{(\log n)^3};
\]
\item
\[
  \mathbb P\bigl(\chi(G_n)>k_\chi^-\bigr)\longrightarrow1;
\]
\item
\[
  \mathbb P\bigl(\zeta(G_n)\le k_{\mathrm{co}}+a_n\bigr)
  \longrightarrow1.
\]
\end{enumerate}
The intersection of the events in (B) and (C), together with (A), gives the
phase-resolved assertion of the main theorem; no independence between the two
events is required.

Here $r_+$ and $r_5^{\mathrm{co}}$ are first-moment roots: the numbers of
classes at which the optimized logarithmic counting exponent vanishes for
ordinary colorings and for marked five-size partitions, respectively
(Section~\ref{continuous-profile-roots}). Figure~\ref{fig:final-root-geometry}
shows the order of all five locations. Sections~1--3 collect standard
estimates, the phase expansion and the roots, and Section~4 proves (B).
Section~5 shows that $r_+-r_5^{\mathrm{co}}$ has the order in (A) and fixes
the profile and $k_{\mathrm{co}}$. Sections~6--9 bound the second moment of
the marked cocoloring count, and Section~10 uses that bound to construct
\(a_n\) and prove (C). The final section completes (A) and combines the
bounds. Long routine proofs are given in the appendices.

\begin{figure}[htbp]
\centering
\begin{tikzpicture}[
  x=1cm,
  y=1cm,
  >=Latex,
  every node/.style={font=\scriptsize}
]
  \draw[-{Latex[length=2mm]}] (-0.35,0) -- (12.0,0);
  \draw[blue!55!black,line width=1.6pt] (2.0,0) -- (10.6,0);

  \foreach \x in {0,1.0,2.0,10.6,11.5}
    \draw[line width=0.55pt] (\x,-0.10) -- (\x,0.10);

  \node[below=3pt] at (0,0) {$r_5^{\mathrm{co}}$};
  \node[above=3pt] at (1.0,0) {$k_{\mathrm{co}}$};
  \node[below=3pt] at (2.0,0) {$k_{\mathrm{co}}+a_n$};
  \node[above=3pt] at (10.6,0) {$k_\chi^-$};
  \node[below=3pt] at (11.5,0) {$r_+$};
  \node[text=blue!55!black,fill=white,inner sep=1.5pt] at (6.3,0.42)
    {usable separation};

  \draw[<->] (0,1.35) -- (11.5,1.35);
  \node[above=2pt,font=\scriptsize] at (5.75,1.35) {$\displaystyle
    r_+-r_5^{\mathrm{co}}
    =\left[\frac{(\log 2)^2}{4}A_5(\delta_n)+o(1)\right]
      \frac{n}{(\log n)^3}$};

  \draw[<->] (1.0,-1.30) -- (11.5,-1.30);
  \node[below=2pt,font=\scriptsize] at (6.25,-1.30) {$\displaystyle
    r_+-k_{\mathrm{co}}
    =\left[\frac{(\log 2)^2}{4}A_5(\delta_n)+o(1)\right]
      \frac{n}{(\log n)^3}$};
\end{tikzpicture}
\caption{First-moment roots and the class counts used after rounding and
amplification. For all sufficiently large $n$,
$r_5^{\mathrm{co}}<k_{\mathrm{co}}<k_{\mathrm{co}}+a_n<k_\chi^-<r_+$.
The thick segment is the remaining gap between the class counts.
Positions are schematic and not to scale.}
\label{fig:final-root-geometry}
\end{figure}

\subsection*{The two counting difficulties}

Two kinds of overlap enter the second moment. The first is whole shared
classes. Section~7 handles them by exact recurrences when few or almost
all classes are shared, and by a negative exponential rate in the
intermediate range. These estimates are uniform over the phase, including
the range omitted in the earlier argument. The second is large
intersections between classes that are not equal.

To handle these intersections, we form a table with one row for each class
of the first partition and one column for each class of the second. Each
entry is the size of the corresponding intersection. Entries larger than
half the maximum class size cannot share a row or column; viewed as edges
between row and column classes, they form a matching. Section~8 bounds
their contribution. Section~9 treats the other entries and the compatibility
of the class declarations. We retain the chosen vertex pairs throughout,
so each labeled configuration is counted once. If \(Z\) counts the marked cocolorings of the chosen
profile, the resulting second-moment bound is
\[
 \frac{\mathbb E Z^2}{(\mathbb E Z)^2}
 \le \exp\!\left\{o\!\left(\frac{n}{(\log n)^4}\right)\right\}.
\]

\section*{Notation and counting conventions}
\label{sec:conventions}

\paragraph{Asymptotic and counting conventions.}
An event holds \emph{with high probability} if its
probability tends to one as $n\to\infty$. Whenever a phase parameter is
present, each occurrence of $o(1)$ denotes a deterministic error sequence that
is uniform over the full phase interval, including sequences approaching
either endpoint. Since $\delta_n$ is a function of $n$, uniformity in the
phase is a property of the proofs: no constant, threshold or error sequence
depends on a lower bound for $\delta_n$ or for $1-\delta_n$.
For integers $x,r\ge0$, we use the falling factorial
\[
  (x)_r=x(x-1)\cdots(x-r+1),
  \qquad
  (x)_0=1,
\]
and set $(x)_r=0$ for $r>x$. A quotient involving falling factorials is used
only after the denominator has been shown to be nonzero.

\paragraph{Class marks and profiles.}
For the counting argument, give each partition class a mark $I$ or $K$,
meaning that the class must be independent or complete, respectively.
All classes in the selected profile have size at least two for sufficiently
large $n$. The two requirements are therefore disjoint, and a realized
cocoloring determines its marks uniquely. Forgetting the marks introduces
no multiplicity.

A \emph{profile} is a finitely supported sequence $(k_s)_{s\ge1}$ of
nonnegative integers, where $k_s$ is the number of classes of size $s$
\citep[Section~2.2]{heckel-2025-difference}. It is
feasible on $n$ vertices with $k$ parts when
\[
  \sum_s k_s=k,
  \qquad
  \sum_s s k_s=n.
\]
From Section~\ref{continuous-profile-roots} on, profiles are indexed by the
deficit $i=\alpha-s$, where $\alpha=\lfloor\alpha_0(n)\rfloor$: there $k_i$ is
the number of classes of size $\alpha-i$, so $k_2$ counts classes of size
$\alpha-2$, not $2$.
The ordinary first moment is the expected number of profile partitions
whose classes are all independent. The marked first moment is the expected
number of profile partitions with one $I/K$ declaration per class, each
satisfying its independence or completeness requirement.

\paragraph{Overlap tables.}
For two ordered profile partitions, the overlap table records the sizes of
all pairwise intersections of their classes, and its support graph joins two
classes whose intersection has size at least two; both are defined in
Section~\ref{exact-signed-second-moment-representation}. Large cells and
endpoint tables are introduced in Sections~8--9.

\section{Independent-set counts and probability inequalities}\label{notation-and-elementary-facts}

This section records the expected independent-set counts and the standard
probabilistic inequalities used throughout the paper; none of them is new.

For integers \(v,s\ge0\), with \(\binom vs=0\) when \(s>v\), let

\[
 \mu_s(v)=\binom vs2^{-\binom s2},\qquad \mu_s=\mu_s(n). \tag{1.1}
\]

We use the following standard inequalities, stated with the hypotheses
and constants needed in the proof.

\begin{enumerate}
\def\labelenumi{\arabic{enumi}.}
\tightlist
\item
  For integers \(m\ge 1\), \[
   \log(m!)=m\log m-m+\frac12\log(2\pi m)+O(1/m).          \tag{1.2}
  \] Consequently, with \(0\log 0=0\), one has
  \(\log(m!)=m\log m-m+O(\log(m+1))\) uniformly for
  \(m\in\mathbb N_0\), with an absolute implied constant.
\item
  If \(r\ge1\), \(t\ge0\), and a random variable \(Y\) is a function of \(r\)
  independent blocks and changing one block changes \(Y\) by at
  most one, then \[
   \Prob{\lvert Y-\Exp{Y}\rvert\ge t}
   \le2\exp(-2t^2/r).                                    \tag{1.3}
  \] We will use the corresponding one-sided bounds as well.
\item
  If \(Z\ge 0\), \(0<\Exp{Z}<\infty\), and \(\Exp{Z^2}<\infty\), then \[
   \Prob{Z>0}\ge\frac{\Exp{Z}^2}{\Exp{Z^2}}.             \tag{1.4}
  \]
\item
  If \(X\sim\operatorname{Bin}(m,1/2)\), then \[
   \Prob{X\le m/4}\le e^{-m/16}.                         \tag{1.5}
  \] Indeed, exponential Markov with \(t=\log 3\) bounds the
  probability by
  \([3^{1/4}(2/3)]^m\le e^{-m/16}\).
\item
  For every nonnegative random variable \(X\) and \(a>0\),
  \(\Prob{X\ge a}\le \Exp{X}/a\).
\end{enumerate}

The first is Stirling's estimate, the second is McDiarmid's
bounded-differences inequality, and the third is the zero-threshold case of
Paley--Zygmund. The fourth is the only binomial-tail estimate used
below; the fifth is Markov's inequality. The bounded-differences method is due to
\citet{mcdiarmid-1989}; the one-sided formulation used here is
\citet[Theorem~3.1]{mcdiarmid-1998}.

\section{Uniform first-moment estimates across the rounding phase}
\label{the-complete-independence-number-phase}

This section fixes the parameters that describe the rounding phase and
records, uniformly in the phase, how $\mu_s$ behaves for $s$ near
$\alpha_0$. Two consequences are used later: independent sets of size
$\alpha+2$ are absent with high probability by (2.9), which caps the class
sizes in Section~\ref{a-valid-unrestricted-lower-location-for-chi}; and
$\mu_{\alpha-2}$ is large by (2.4), which controls the common classes in
Section~\ref{exact-partial-diagonals}.

Put
\[
  q:=\log 2,
  \qquad L:=\log n,
  \qquad \ell:=\log\log n,
  \qquad C:=1+\log q-q.
\]
(In Sections~6--7 and in the appendix that proves
Lemma~\ref{lemma-7.1-all-common-subprofiles}, $\ell$ instead counts common
classes.)
We shall use the elementary bounds
\begin{equation}
 \frac23<q<\frac7{10}.
 \tag{2.0}
\end{equation}
We also use $q>56/81$: writing
$q=2\int_0^{1/3}(1-t^2)^{-1}\,dt$ and retaining $1+t^2$ in the integrand
gives this strict bound.
The location parameter $\alpha_0$ defined in the introduction has the
following equivalent form in natural logarithms
\citep[Section~2.1]{heckel-2025-difference}:
\begin{equation*}
  \alpha_0
  =2\log_2 n-2\log_2\log_2 n+2\log_2(\mathrm e/2)+1
  =\frac{2(L-\ell+C)}q+1.
  \tag{2.1}
\end{equation*}
Let
\[
  \alpha:=\lfloor\alpha_0\rfloor,
  \qquad
  \delta:=\alpha_0-\alpha,
  \qquad
  b:=1-\delta.
\]
Thus $0\le\delta<1$, $0<b\le1$, and
\[
  \alpha=\frac{2(L-\ell+C)}q+b.
\]
When emphasizing dependence on $n$, we write $\delta_n:=\delta$.
We prove uniform estimates on the closed range $0\le\delta\le1$.

\begin{lemma}[Uniform phase expansion and adjacent-size control]
\label{lem:phase-expansion}
There exist a bounded continuous function $K\colon[0,1]\to\mathbb R$,
absolute constants $C_{\mathrm{ph}},C_2,c>0$, and a real sequence
$(E_n^{\mathrm{ph}})_{n\ge2}$ such that, with
\[
  \varepsilon_n^{\mathrm{ph}}
  :=C_{\mathrm{ph}}\frac{1+\ell^2}{L},
\]
one has $\varepsilon_n^{\mathrm{ph}}\to0$ and, for all sufficiently large
$n$, the following expansion holds at the induced phase
$\delta_n=\alpha_0-\lfloor\alpha_0\rfloor$:
\begin{equation*}
  \log\mu_\alpha
  =
  \delta_n L
  +\left(\frac2q-\frac12-\delta_n\right)\ell
  +K(\delta_n)
  +E_n^{\mathrm{ph}},
  \qquad
  |E_n^{\mathrm{ph}}|\le\varepsilon_n^{\mathrm{ph}}.
  \tag{2.2}
\end{equation*}
The error bound is uniform over all induced phase values, and therefore also
along sequences approaching either endpoint of $[0,1]$. Moreover, uniformly
over those phase values,
\begin{equation*}
\begin{split}
  \log\mu_{\alpha+2}
  &=(\delta_n-2)L
    +\left(\frac2q+\frac32-\delta_n\right)\ell+O(1),\\
  \mu_{\alpha+2}
  &\le \exp\{-L+C_2\ell\}=o(1)
\end{split}
  \tag{2.3}
\end{equation*}
while
\begin{equation*}
  \mu_{\alpha-2}
  \ge c\,n^2(\log n)^{2/q-5/2}
  \tag{2.4}
\end{equation*}
for all sufficiently large $n$.
\end{lemma}

The proof starts with Stirling's formula and bounds its remainder
uniformly in the phase. Ratios between adjacent class sizes then give the
upper size cap and the lower first-moment bound.
Section~\ref{app:phase-expansion-details} gives both calculations.

Let $\alpha(G)$ denote the independence number of a graph $G$, and let
$X_{\alpha+2}$ denote the number of independent sets of size
$\alpha+2$. Then
$\mathbb E X_{\alpha+2}=\mu_{\alpha+2}$, and the event
$\alpha(G_n)>\alpha+1$ implies $X_{\alpha+2}\ge1$. Consequently Markov's
inequality and (2.3) give, uniformly along the full sequence of integers, the
deterministic cap error
\begin{equation*}
  \mathbb P\bigl(\alpha(G_n)>\alpha+1\bigr)
  \le\mu_{\alpha+2}
  =:\varepsilon_n^{\mathrm{cap}}
  \longrightarrow0
  \tag{2.9}
\end{equation*}

\section{Real profile optimization and first-moment roots}
\label{continuous-profile-roots}

We initially allow real class counts and maximize the logarithmic counting
exponent subject to fixed totals for classes and vertices. A zero of this
optimized exponent is a first-moment root. The ordinary-coloring root is used to obtain
a lower bound for the chromatic number; the five-size marked-partition root
determines the number of parts in our cocoloring. Section~5 compares these
roots before the class counts are rounded to integers.
Proposition~\ref{prop:selected-profile-contract} collects the properties of
the selected integer profile.
Throughout this section, $n$ is sufficiently large that $\alpha\ge10$.
For an integer class size $u\ge1$, put
\[
  d_u:=2^{\binom u2}u!.
\]
Write $u=\alpha-i$. The index $i$ records the deficit of a class size from
the integer center $\alpha$. We call the set of allowed deficits the
support and use
\begin{equation*}
  S_+:=\{-1,0,1,2,\ldots\},
  \qquad
  S_5:=\{2,3,4,5,6\}.
  \tag{3.1}
\end{equation*}
At finite $n$, the unrestricted support is
$S_+^{(n)}:=\{-1,0,\ldots,\alpha-1\}$: the lower endpoint is the
deterministic size cap $u\le\alpha+1$, and the upper endpoint comes from
positivity of class sizes.
For $S\in\{S_+,S_5\}$, set
\[
  S^{(n)}:=
  \begin{cases}
    S_+^{(n)},&S=S_+,\\
    S_5,&S=S_5.
  \end{cases}
\]
All profiles, maximizations and sums at finite $n$ use $S^{(n)}$.
Only the limiting quantities use the full support $S$.

We separate the class-size contribution into an affine function of the
deficit and a remainder. The conservation constraints fix the affine
contribution, leaving only the remainder to optimize.
For $i\in S_+^{(n)}$, define that remainder by
\begin{equation*}
  h_n(i):=
  \begin{cases}
    -\dfrac{\log 2}{2}i^2
      +\displaystyle\sum_{r=0}^{i-1}\log\!\left(1-\dfrac r\alpha\right),
      & i\ge0,\\[1ex]
    -\dfrac{\log 2}{2}+\log\!\left(\dfrac{\alpha}{\alpha+1}\right),
      & i=-1.
  \end{cases}
  \tag{3.1a}
\end{equation*}

The affine part is fixed by the mean constraint:
\begin{equation*}
  -\log d_{\alpha-i}=A_n+B_ni+h_n(i),
  \tag{3.1b}\label{eq:root-affine}
\end{equation*}
where
\begin{equation*}
  A_n:=-\log d_\alpha,
  \qquad
  B_n:=q\alpha-\frac q2+\log\alpha,
  \tag{3.1c}\label{eq:root-affine-coefficients}
\end{equation*}
Direct subtraction gives this identity also at $i=-1$.

For $S\in\{S_+,S_5\}$ and a nonnegative real profile
$\mathbf k=(k_i)_{i\in S^{(n)}}$, put
$k:=\sum_{i\in S^{(n)}}k_i$. When $k>0$, define
\begin{equation*}
  L_{\mathbf k}
  :=n\log n-n+k-\sum_{i\in S^{(n)}} k_i\log(k_i d_{\alpha-i}),
  \tag{3.2}
\end{equation*}
with the convention $0\log0=0$, and let $L_S(n,k)$ be its maximum over
real $k_i\ge0$ satisfying
\begin{equation*}
  \sum_{i\in S^{(n)}}k_i=k,
  \qquad
  \sum_{i\in S^{(n)}}(\alpha-i)k_i=n.
  \tag{3.3}
\end{equation*}
We call $k>0$ \emph{feasible for $S$} if (3.3) has a nonnegative real
solution, that is, if $\min S^{(n)}\le\alpha-n/k\le\max S^{(n)}$.
For $k$ feasible for the support $S$ under consideration, set
\begin{equation*}
  s:=\frac nk,
  \qquad
  T:=\alpha-s=\alpha-\frac nk.
  \tag{3.4}
\end{equation*}
With $p_i:=k_i/k$, the constraints give
$T=\sum_{i\in S^{(n)}}ip_i$; thus $T$ is the mean deficit of the normalized
profile.

\Needspace{8\baselineskip}
For $S\in\{S_+,S_5\}$, assign positive weights to the allowed deficits.
Normalizing by their sum, the partition function, gives a probability
distribution. A real parameter, the tilt, changes its mean deficit; the mean
map gives this mean as a function of the tilt. The finite and limiting
versions are
\[
  Z_{n,S}(\lambda)
  :=\sum_{i\in S^{(n)}}\exp\{\lambda i+h_n(i)\},
  \qquad
  M_{n,S}(\lambda):=\frac{d}{d\lambda}\log Z_{n,S}(\lambda),
\]
\[
  Z_S(\lambda)
  :=\sum_{i\in S}\exp\!\left(\lambda i-\frac q2i^2\right),
  \qquad
  M_S(\lambda):=\frac{d}{d\lambda}\log Z_S(\lambda).
\]
To obtain a prescribed mean deficit, we solve $M_{n,S}(\lambda)=T$ or
$M_S(\lambda)=T$. When the solution is unique, write it as
$\lambda_{n,S}(T)$ or $\lambda_S(T)$, respectively.

We need the first-moment roots only near the typical average class size.
The next lemma gives a narrow interval, called the root corridor, in which
the roots are unique and their slopes are controlled. It also relates the
finite profile optimization to its limiting one-parameter formula.

In detail, the lemma collects every property of the optimized exponents that
is used later: (i)~a unique root in a narrow corridor, located by (3.6),
which defines $r_+$ in
Section~\ref{a-valid-unrestricted-lower-location-for-chi} and
$r_5^{\mathrm{co}}$ in
Section~\ref{the-five-size-signed-first-moment-advantage}; (ii)~the slope
(3.7), which converts differences of the exponent into differences of the
number of classes in (4.3c), (5.11) and
Appendix~\ref{app:profile-rounding-proof}; (iii)~the exact support
comparison (3.8), used in (5.9a); (iv)~uniform convergence of the finite
duals and tilts, which bounds $\omega_n^{\mathrm{root}}$ in (5.9c); and
(v)~convergence and a positive lower bound for the five-size optimizer, used
in (5.15).

\Needspace{6\baselineskip}
\begin{lemma}[Uniform root, slope, and finite-dual estimates]
\label{lem:uniform-root-slope-dual}
Let
\begin{equation*}
  s_0:=\alpha_0-1-\frac2{\log 2},
  \qquad
  T_0:=\alpha-s_0
      =1+\frac2{\log 2}-\delta.
  \tag{3.5}
\end{equation*}
There exist a fixed $A_*>0$ and an integer $n_{\mathrm{root}}\ge2$ such that,
for every $n\ge n_{\mathrm{root}}$, every $S\in\{S_+,S_5\}$, and every
$0\le c\le\log 2$, the equation
\[
  L_S(n,k)+ck=0
\]
has a unique feasible zero $r_{S,c}$ in the
  corridor
\[
  \left|\frac n{r_{S,c}}-s_0\right|
  \le A_*\frac{\log\log n}{\log n}.
\]
No uniqueness outside this corridor is asserted.
Uniformly in $S$, $c$, and the phase,
\begin{equation*}
  \frac n{r_{S,c}}
  =s_0+O\!\left(\frac{\log\log n}{\log n}\right),
  \qquad
  \alpha-\frac n{r_{S,c}}
  =T_0+O\!\left(\frac{\log\log n}{\log n}\right).
  \tag{3.6}
\end{equation*}

For every fixed $A>0$, uniformly over $S\in\{S_+,S_5\}$,
$0\le c\le\log 2$, and $k$ feasible for $S$ with
$|n/k-s_0|\le A\log\log n/\log n$,
\begin{equation*}
  \frac{\partial}{\partial k}\{L_S(n,k)+ck\}
  =\frac2{\log 2}(\log n)^2
   +O_A((\log n)(\log\log n)).
  \tag{3.7}
\end{equation*}
For the normalized slope error, take the maximum over
$S\in\{S_+,S_5\}$ and $0\le c\le\log 2$. For each pair, take the supremum
over feasible $k$ satisfying $|n/k-s_0|\le A\log\log n/\log n$. Thus the
deterministic normalized slope error
\[
  \varepsilon_{n,A}^{\mathrm{slope}}
  :=
  \max_{S,c}
  \sup_k
  \left|
    \frac{1}{(\log n)^2}
    \frac{\partial}{\partial k}\{L_S(n,k)+ck\}
    -\frac2{\log 2}
  \right|
\]
satisfies $\varepsilon_{n,A}^{\mathrm{slope}}\to0$.

For two supports $R,S\in\{S_+,S_5\}$ at a $k$ feasible for both finite
supports $R^{(n)}$ and $S^{(n)}$,
\begin{equation*}
  \frac{L_R(n,k)-L_S(n,k)}{k}
  =\mathcal F_{n,R}(T)-\mathcal F_{n,S}(T),
  \tag{3.8}
\end{equation*}
where the maximum below is over nonnegative weights on $S^{(n)}$ with
total mass one and mean $T$:
\begin{equation*}
  \mathcal F_{n,S}(T)
  :=
  \max
  \left[
    -\sum_{i\in S^{(n)}} p_i\log p_i
    +\sum_{i\in S^{(n)}} p_i h_n(i)
  \right].
  \tag{3.8a}
\end{equation*}
Fix the compact target interval
\begin{equation*}
  K_*:=
  \left[
    \frac2{\log 2}-\frac1{10},
    1+\frac2{\log 2}+\frac1{10}
  \right]
  \subset(2,5).
  \tag{3.9}
\end{equation*}
Then, uniformly for $T\in K_*$,
\begin{equation*}
  \mathcal F_{n,S}(T)
  \longrightarrow
  \mathcal F_S(T),
\end{equation*}
where $\mathcal F_S$ is the function defined in the main theorem.
Equivalently, maximize over nonnegative weights on $S$ with total mass one,
mean $T$, and finite second moment, $\sum_{i\in S}i^2p_i<\infty$:
\begin{equation*}
  \mathcal F_S(T)
  =
  \max
  \left[
    -\sum_{i\in S} p_i\log p_i
    -\frac{\log 2}{2}\sum_{i\in S} i^2p_i
  \right].
  \tag{3.9a}
\end{equation*}
For every $S\in\{S_+,S_5\}$ and $T\in K_*$, the limiting mean equation
$M_S(\lambda)=T$ has the unique solution $\lambda_S(T)$. There exist a fixed
$\Lambda>0$ and an integer $n_{\mathrm{dual}}$ such that, for every
$n\ge n_{\mathrm{dual}}$, every such $S$, and every $T\in K_*$, the finite
equation $M_{n,S}(\lambda)=T$ has the unique solution $\lambda_{n,S}(T)$ and
\[
  |\lambda_{n,S}(T)|+|\lambda_S(T)|\le2\Lambda.
\]
For $n\ge n_{\mathrm{dual}}$, define the deterministic dual error
\[
\begin{split}
  \varepsilon_n^{\mathrm{dual}}
  :=\max_{S\in\{S_+,S_5\}}
  \sup_{T\in K_*}
  \bigl(&|\mathcal F_{n,S}(T)-\mathcal F_S(T)|\\
        &+|\lambda_{n,S}(T)-\lambda_S(T)|\bigr)
\end{split}
\]
and set it arbitrarily, say equal to zero, for $n<n_{\mathrm{dual}}$. Then
\[
  \varepsilon_n^{\mathrm{dual}}
  =O(1/\alpha)=O(1/\log n),
\]
with a phase-independent implied constant. In particular,
$\varepsilon_n^{\mathrm{dual}}\to0$.

For $S_5$, both optimizers are unique; denote them by $p_{n,i}(T)$ and
$p_i(T)$. Then
\begin{equation*}
  \sup_{T\in K_*}\max_{2\le i\le6}
  |p_{n,i}(T)-p_i(T)|=O(1/\alpha).
  \tag{3.9b}
\end{equation*}
Moreover, there are $\eta_*>0$ and $n_{\mathrm{int}}\in\mathbb N$ such that
\[
  p_{n,i}(T)\ge\eta_*
  \qquad
  (n\ge n_{\mathrm{int}},\ T\in K_*,\ 2\le i\le6).
\]
All displayed bounds and thresholds are uniform over the induced
phases $\delta_n$, including sequences approaching either endpoint.
\end{lemma}

Identity (3.8) is exact for a simple reason. By \eqref{eq:root-affine} and
(3.3), $\sum_ik_i(A_n+B_ni)=kA_n+B_n(\alpha k-n)$ for every admissible
profile, and writing $k_i=kp_i$ in (3.2) gives, for $k$ feasible for $S$,
\begin{equation*}
  L_S(n,k)=n\log n-n+k-k\log k+kA_n+B_n(\alpha k-n)
           +k\,\mathcal F_{n,S}\Bigl(\alpha-\frac nk\Bigr).
  \tag{3.9c}
\end{equation*}
Only the last term depends on the support;
\eqref{eq:root-counting-exponent} is (3.9c) divided by $k$.

Section~\ref{app:uniform-root-proof} proves the lemma. It first bounds
the finite tilts uniformly on \(K_*\). Stirling's formula then locates the
roots and controls their slopes.

\section{\texorpdfstring{A uniform lower bound for the chromatic number $\chi$}{A uniform lower bound for the chromatic number}}
\label{a-valid-unrestricted-lower-location-for-chi}

This section proves estimate (B) of the introduction. A first-moment bound
excludes colorings with $k_\chi^-+1$ classes of size at most $\alpha+1$:
the integer $k_\chi^-$ lies about $\log n$ below the root $r_+$, and the
slope (3.7) makes the counting exponent fall by order $(\log n)^3$ over
that distance, which outweighs the $\exp\{O((\log n)^2)\}$ profiles. The
size cap is then removed at the cost of one class, because independent sets
of size $\alpha+2$ are rare by (2.9).

Let $r_+(n)=r_{S_+,0}$ be the unrestricted-support corridor zero from
Lemma~\ref{lem:uniform-root-slope-dual}, write
$L_+(n,k):=L_{S_+}(n,k)$, and put
\begin{equation*}
  k_\chi^-:=\lfloor r_+(n)\rfloor-\lceil\log n\rceil.
  \tag{4.1}
\end{equation*}
We choose this deterministic integer a logarithmic distance below the
ordinary first-moment root. Over that distance, the slope estimate makes
the counting exponent decrease by more than the error incurred when we
sum over profiles. Since
$r_+(n)=\Theta(n/\log n)$ uniformly in the phase, one has
$1\le k_\chi^-<n$ for all sufficiently large $n$.

For an integer profile
$\mathbf k=(k_i)_{-1\le i\le\alpha-1}$ feasible on $n$ vertices, let
$X_{\mathbf k}$ be the number of unordered proper colorings having exactly
$k_i$ classes of size $\alpha-i$. Direct enumeration, as in
\citet[Section~2.3]{heckel-2025-difference}, gives
\begin{equation*}
  \Exp{X_{\mathbf k}}
  =\frac{n!}{\prod_i((\alpha-i)!)^{k_i}k_i!}
   2^{-\sum_i k_i\binom{\alpha-i}{2}}.
  \tag{4.2}
\end{equation*}
For an integer $k$, let $\mathcal K_{n,k,\alpha+1}$ be the set of these
profiles satisfying
\[
  \sum_i k_i=k,
  \qquad
  \sum_i(\alpha-i)k_i=n,
\]
and define
\[
  E_{n,k,\alpha+1}
  :=\sum_{\mathbf k\in\mathcal K_{n,k,\alpha+1}}
    \Exp{X_{\mathbf k}}.
\]
Thus $E_{n,k,\alpha+1}$ is the expected number of unordered proper colorings
with exactly $k$ nonempty parts, each of size at most $\alpha+1$.
When the feasible profile set is empty, we use the extended-real convention
$\log E_{n,k,\alpha+1}=-\infty$.

\subsection*{Summing the capped coloring profiles}

There are $\alpha+1=O(\log n)$ profile coordinates, and each coordinate lies
between zero and $n$. Hence
\begin{equation*}
  |\mathcal K_{n,k,\alpha+1}|
  \le(n+1)^{\alpha+1}
  =\exp\{O((\log n)^2)\}
  \tag{4.3}
\end{equation*}
uniformly in $k$ and in the phase. Keep the class-size factorials inside
$d_{\alpha-i}$ exact, and put
\[
  R(m):=\log(m!)-(m\log m-m),\qquad R(0):=0.
\]
For every feasible profile, exact cancellation in (4.2) gives
\[
  \log\Exp{X_{\mathbf k}}-L_{\mathbf k}
  =R(n)-\sum_iR(k_i).
\]
The uniform estimate (1.2) therefore gives
\[
  \log\Exp{X_{\mathbf k}}
  \le L_{\mathbf k}+O((\log n)^2).
\]
The error is uniform even when some $k_i=0$: there are $O(\!\log n)$ profile
coordinates and $|R(k_i)|=O(\!\log(n+1))$. For every $k$ feasible for
$S_+^{(n)}$, maximizing over the feasible profile and then summing (4.3) yields
\begin{equation*}
  \log E_{n,k,\alpha+1}
  \le L_+(n,k)+O((\log n)^2).
  \tag{4.3a}
\end{equation*}

\subsection*{Displacement below the root}

Set
\[
  \Delta_n:=r_+(n)-k_\chi^-.
\]
The floor and ceiling in (4.1) give the deterministic bounds
\begin{equation*}
  \log n\le\Delta_n<\log n+2.
  \tag{4.3b}
\end{equation*}
We next verify that the entire interval
$[k_\chi^-,r_+(n)]$ lies in the derivative corridor of
Lemma~\ref{lem:uniform-root-slope-dual}. Write
$s(k)=n/k$. Uniformly on this interval,
$k=\Theta(n/\log n)$ and $s(k)=\Theta(\log n)$. Therefore
\[
  |s(k_\chi^-)-s(r_+)|
  \le
  \sup_{k\in[k_\chi^-,r_+]}
  \frac{n}{k^2}\,\Delta_n
  =O\!\left(\frac{(\log n)^3}{n}\right)
  =o\!\left(\frac{\log\log n}{\log n}\right).
\]
Thus, after one phase-independent threshold, every point of the interval
satisfies
\[
  \left|\frac nk-s_0\right|
  \le (A_*+1)\frac{\log\log n}{\log n}.
\]

The same estimates place $T(k)=\alpha-n/k$ in
$K_*\subset(-1,\alpha-1)$ throughout this interval, so every point is
feasible for $S_+$. Equation (3.7), with $A=A_*+1$, $S=S_+$, and $c=0$,
consequently gives an
absolute $c_{\mathrm{sl}}>0$ such that
\begin{equation*}
  \frac{d}{dk}L_+(n,k)\ge c_{\mathrm{sl}}(\log n)^2
  \qquad
  (k_\chi^-\le k\le r_+)
  \tag{4.3c}
\end{equation*}
for all sufficiently large $n$, uniformly in the phase. The size-cap
argument below uses this one class above $k_\chi^-$. Put $k':=k_\chi^-+1$.
By (4.3b), $k'\in[k_\chi^-,r_+]$ and $r_+-k'\ge\log n-1$, so
$L_+(n,r_+)=0$, the mean-value theorem and (4.3c) give
$L_+(n,k')\le-c_{\mathrm{sl}}(\log n)^2(\log n-1)$. With (4.3a), there is an
absolute $c_1>0$ such that
\begin{equation*}
  E_{n,k',\alpha+1}\le e^{-c_1(\log n)^3}
  \tag{4.4}
\end{equation*}
(the left side is zero if its feasible profile set is empty).

\subsection*{Removing the size cap with one additional class}

For an integer \(s\), let \(X_s\) count independent \(s\)-subsets of \(G_n\),
and put \(\mu_s=\mathbb E X_s=\binom ns2^{-\binom s2}\).
For the size-cap argument, it is enough to rule out the presence of two
distinct independent sets of the exceptional size. The next bound therefore
estimates their pair count, allowing a single exceptional set.

\begin{lemma}[An independent-set pair bound]\label{lem:independent-pair-tail}
For integers \(1\le s<n\), put
\[
 a=\frac{s^2\,2^{s/4}}{n-s},
 \qquad b=\sqrt2\,sn\,2^{-3s/4}.
\]
Then
\[
 \mathbb E[X_s(X_s-1)]\le \mu_s^2 e^a+\mu_s(e^b-1).
\]
\end{lemma}
Section~\ref{app:pair-bound} proves the lemma. It splits the ordered pairs
by the size of their intersection: a union bound handles intersections of
size at most $s/2$, and a conditional first moment handles larger ones.

For any graph on \(n\) vertices, if \(2\le s<n\) and \(X_s\le1\), no
independent set has more than \(s\) vertices. Otherwise an independent
\((s+1)\)-set would contain at least two independent \(s\)-sets.
Moreover, a proper coloring has at most one class of size \(s\).
If such a class is present, split off one vertex as a singleton. Thus, whenever
\(1\le k<n\), a proper coloring with at most \(k\) classes yields one with
at most \(k+1\) classes, all of size at most \(s-1\). Since the target class
count does not exceed the number of vertices, further splitting gives exactly
\(k+1\) nonempty classes.

Apply this observation with \(s=\alpha+2\) and \(k=k_\chi^-\), using the
same \(n\) and cutoff throughout. If \(X_s\le1\) and
\(\chi(G_n)\le k_\chi^-\), the observation gives a proper coloring with
exactly \(k_\chi^-+1\) nonempty classes, all of size at most \(\alpha+1\).
Its profile lies in \(\mathcal K_{n,k_\chi^-+1,\alpha+1}\), so
\(\sum_{\mathbf k\in\mathcal K_{n,k_\chi^-+1,\alpha+1}}X_{\mathbf k}\ge1\),
and Markov's inequality bounds the probability of this event by
\(E_{n,k_\chi^-+1,\alpha+1}\). Since \(X_s(X_s-1)\ge2\) on
\(\{X_s\ge2\}\), Markov's inequality also gives
\(\mathbb P(X_s\ge2)\le\frac12\mathbb E[X_s(X_s-1)]\), which
Lemma~\ref{lem:independent-pair-tail} bounds. Hence
\[
 \mathbb P(\chi(G_n)\le k_\chi^-)
 \le\frac{\mu_s^2e^a+\mu_s(e^b-1)}2
       +E_{n,k_\chi^-+1,\alpha+1}.
\]
By (4.4), we have proved
\begin{equation*}
 \mathbb P(\chi(G_n)\le k_\chi^-)
 \le\frac{\mu_s^2e^a+\mu_s(e^b-1)}2+e^{-c_1(\log n)^3},
 \qquad s=\alpha+2.
 \tag{4.5}
\end{equation*}
The phase formula gives \(2^s=\Theta(n^2/(\log n)^2)\) uniformly, hence
\(a=O((\log n)^{3/2}/\sqrt n)\) and
\(b=O((\log n)^{5/2}/\sqrt n)\).
Together with \(\mu_{\alpha+2}\to0\) from (2.9), this proves
\begin{equation*}
 \mathbb P(\chi(G_n)>k_\chi^-)\longrightarrow1.
 \tag{4.6}
\end{equation*}
The sharper form of (4.5) will also yield a summable failure bound below.
Finally, (4.3b) records the unchanged location precision:
\begin{equation*}
 |k_\chi^--r_+(n)|
 <\log n+2=o\!\left(\frac n{(\log n)^3}\right).
 \tag{4.7}
\end{equation*}

\section{Five-size entropy gain and the selected integer profile}\label{the-five-size-signed-first-moment-advantage}

Restricting the allowed deficits to \(S_5=\{2,3,4,5,6\}\) reduces the
number of available profiles, but makes the later overlap estimates uniform.
We compare the resulting entropy loss per part with \(\log 2\), the gain from
the \(2^k\) class declarations. The loss is smaller, and
Lemma~\ref{lem:entropy-gap} gives a uniform margin for the resulting
positive net gain. The slope estimate
in Lemma~\ref{lem:uniform-root-slope-dual} then converts this gain in the
counting exponent into a separation between the roots.

\Needspace{8\baselineskip}
\subsection{The phase coefficient and entropy bound}
For $S\in\{S_+,S_5\}$ write
$w_i(\lambda):=\exp\bigl(\lambda i-\frac{\log 2}{2}i^2\bigr)$, so that
$Z_S(\lambda)=\sum_{i\in S}w_i(\lambda)$ and
$M_S(\lambda)=\sum_{i\in S}i\,w_i(\lambda)/Z_S(\lambda)=(\log Z_S)'(\lambda)$
as in Section~\ref{continuous-profile-roots}.

The Gaussian factor makes both sums converge. Moreover,
\(M_S'(\lambda)=\operatorname{Var}_{S,\lambda}(i)>0\), and the endpoint
limits of \(M_S\) are the endpoints of the convex hull of \(S\). Thus, for
every \(T\in(2,5)\), there is a unique tilt
\(\lambda_S(T)\) with \(M_S(\lambda_S(T))=T\). The Gibbs identity \eqref{eq:root-gibbs-identity}
proves that the optimizer in (3.9a) is

\[
 p_i=\frac{w_i(\lambda_S(T))}{Z_S(\lambda_S(T))}
 =\frac{e^{\lambda_S(T)i-{\log 2}i^2/2}}
        {\sum_{j\in S}e^{\lambda_S(T)j-{\log 2}j^2/2}}.       \tag{5.1}
\]
This also gives the dual representation

\[
 \mathcal F_S(T)
 =\inf_{\lambda\in\mathbb R}
   \{\log Z_S(\lambda)-\lambda T\}
 =\log Z_S(\lambda_S(T))-\lambda_S(T)T.
\]

These facts give the side assertions of the main theorem. By (2.0),
$T_\delta=1+2/q-\delta\in[2/q,1+2/q]\subset(14/5,4)\subset(2,5)$ for
$0\le\delta\le1$. For $T\in(2,5)$, the map
$\lambda\mapsto\log Z_S(\lambda)-\lambda T$ is strictly convex, with
derivative $M_S(\lambda)-T$, so its infimum $\mathcal F_S(T)$ is finite and
attained only at $\lambda_S(T)$. Continuity and nonconstancy of $A_5$ are
proved in Lemma~\ref{lem:phase-gain-concavity}, and positivity in
Lemma~\ref{lem:entropy-gap}.

Since $T_0=1+2/{\log 2}-\delta=T_\delta$, introduce the entropy loss
$D_5$ and the uniform bound $\gamma_5$ by

\[
  D_5(\delta):=\mathcal F_{S_+}(T_\delta)
               -\mathcal F_{S_5}(T_\delta),\qquad
  A_5(\delta)=\log 2-D_5(\delta),\qquad
  \gamma_5:=\log\frac{8750}{6681}.                            \tag{5.2}
  \label{eq:phase-function-definition}
\]
To find the smallest phase gain, compare the two optimizing distributions
at the same mean deficit. The five-size distribution has smaller variance
than the unrestricted one. The following lemma translates this comparison
into strict concavity of the phase gain.

\begin{lemma}[Curvature and minimum of the phase gain]
\label{lem:phase-gain-concavity}
For $S\in\{S_+,S_5\}$ and $T\in(2,5)$, let
\[
 V_S(T):=\sum_{i\in S}(i-T)^2
 \frac{w_i(\lambda_S(T))}{Z_S(\lambda_S(T))}.
\]
Then $0<V_{S_5}(T)<V_{S_+}(T)$. The function $A_5$ extends to a twice
continuously differentiable function on an open interval containing $[0,1]$,
and
\[
 A_5'(\delta)=\lambda_{S_5}(T_\delta)-\lambda_{S_+}(T_\delta),
 \qquad
 A_5''(\delta)=\frac1{V_{S_+}(T_\delta)}
               -\frac1{V_{S_5}(T_\delta)}<0.
\]
In particular, $A_5$ is strictly concave and nonconstant on $[0,1]$, and
\[
 \min_{0\le\delta\le1}A_5(\delta)
 =\min\{A_5(0),A_5(1)\}.
\]
\end{lemma}

Section~\ref{app:phase-gain-concavity} proves the lemma by differentiating
the dual values twice in $T$ and comparing the variances of the two
optimizing distributions.

For comparison only, extend the same tilted definitions to
$S_4:=\{2,3,4,5\}$, and put
\[
 A_4(\delta):=q-\mathcal F_{S_+}(T_\delta)+\mathcal F_{S_4}(T_\delta),
 \qquad D_4(\delta):=q-A_4(\delta).
\]
This four-size calculation gives an explicit reference bound for the
entropy loss. We use it to obtain the stronger five-size bound.
\begin{lemma}[Four-size comparison certificate]
\label{lem:four-size-entropy-certificate}
For $0\le\delta\le1$,
\[
 0\le D_4(\delta)<\log(153/100),\qquad A_4(\delta)>\log(200/153).
\]
\end{lemma}
Its proof in Section~\ref{app:entropy-certificate-details} uses only the
four-size optimizer; none of the profiles selected below uses that support.

\begin{lemma}[Uniform entropy bound and strict five-size gain]
\label{lem:entropy-gap}
For every $0\le\delta\le1$,
\[
 A_5(\delta)-A_4(\delta)>\log(525/524),\qquad
 0\le D_5(\delta)<\log(6681/4375),\qquad A_5(\delta)>\gamma_5.
 \tag{5.3}
\]
\end{lemma}
Section~\ref{app:entropy-gap-proof} proves the lemma. It brackets the
five-size tilt at $T_\delta$ and bounds from below the optimal mass at
deficit six, which is what $S_5$ gains over $S_4$.

Lemma~\ref{lem:entropy-gap} suffices for every positivity argument below.
The lower bound for the uniform constant in the main theorem uses the
following sharper bound. By Lemma~\ref{lem:phase-gain-concavity}, the
minimum of $A_5$ on $[0,1]$ is attained at an endpoint.
Section~\ref{app:phase-gain-numerics} records numerical values of $A_5(0)$
and $A_5(1)$; they are not used in any proof.

Lemma~\ref{lem:phase-gain-sharp} also implies $A_5>\gamma_5$. We keep
Lemmas~\ref{lem:four-size-entropy-certificate} and~\ref{lem:entropy-gap}
because their proofs do not use the numerical endpoint table of
Section~\ref{app:phase-gain-sharp-endpoints}, and because the bound
$A_4>\log(200/153)$ of Lemma~\ref{lem:four-size-entropy-certificate} gives
the coefficient $((\log 2)^2/4)\log(200/153)$ of the released Lean theorem
(Appendix~\ref{app:lean-reproducibility}).

\begin{lemma}[A sharp uniform bound for the phase gain]
\label{lem:phase-gain-sharp}
For every $0\le\delta\le1$,
\[
 A_5(\delta)>\log\frac{169}{100}.
\]
\end{lemma}
Section~\ref{app:phase-gain-sharp-proof} proves the lemma. It bounds the
five-size dual from below by convexity and the unrestricted dual from above
at an explicit tilt, on two cells that cover the phase interval; the
resulting inequalities reduce to rational evaluations at the cell endpoints
(Table~\ref{tab:phase-gain-sharp-endpoints}).

\subsection{The gain from class marks and root separation}

The factor \(2^k\) comes from the one-partition identity in
\citet[Proposition~6]{heckel-2025-difference}. We recall its short proof
before using it to compare the roots.  For a fixed
partition into \(k\) classes, assigning each class the declaration
``independent'' or ``complete'' multiplies its first moment by \(2^k\):
every declaration prescribes all internal edge bits, and all \(2^k\)
declarations have probability \(2^{-\sum \binom{u}{2}}\). The declarations
are disjoint because every allowed class size is at least two for all
sufficiently large $n$.
Put
\begin{equation*}
  \Phi_n(k):=L_{S_5}(n,k)+(\log 2)k,
  \tag{5.9}
\end{equation*}
and let $r_5^{\mathrm{co}}:=r_{S_5,\log 2}$ be its unique corridor zero from
Lemma~\ref{lem:uniform-root-slope-dual}. To compare the two roots, evaluate
$\Phi_n$ at the unrestricted root. Let
\[
  T_+(n):=\alpha-\frac{n}{r_+(n)}
\]
and introduce the finite-$n$ support loss
\[
  D_{5,n}(T):=
  \mathcal F_{n,S_+}(T)-\mathcal F_{n,S_5}(T).
\]
By (3.6) and the fixed margin in (3.9), one has
$T_+(n)\in K_*\subset(2,5)$ for all sufficiently large $n$, uniformly in the
phase.  Thus $r_+$ is feasible for both supports before (3.8) is applied.
Since $L_{S_+}(n,r_+)=0$, the exact support-comparison identity (3.8) gives
\begin{equation*}
\begin{aligned}
  \Phi_n(r_+)
  &=L_{S_5}(n,r_+)-L_{S_+}(n,r_+)+(\log 2)r_+\\
  &=r_+\{\log 2-D_{5,n}(T_+(n))\}.
\end{aligned}
\tag{5.9a}
\end{equation*}
For finite $n$, this identity expresses the marked counting exponent at
the unrestricted root in terms of the loss from restricting the support.
We next replace that finite loss by its phase-dependent limit.

There are two errors in this replacement: the finite dual values differ
from their limits, and the root target $T_+(n)$ differs from the phase-center
target $T_0$. For the latter, define
\[
  \varepsilon_n^{\mathrm{target}}
  :=|T_+(n)-T_0|.
\]
By (3.6), uniformly in the phase,
\begin{equation*}
  \varepsilon_n^{\mathrm{target}}
  =O\!\left(\frac{\log\log n}{\log n}\right).
  \tag{5.9b}
\end{equation*}

The limiting dual value is differentiable, with
\[
  \mathcal F_S'(T)=-\lambda_S(T),
\]
because the envelope theorem cancels the derivative of the optimizing tilt.
Section~3 bounds the limiting tilts by one constant on $K_*$, so the two
limiting dual values have a common Lipschitz constant $C_{\mathrm{Lip}}$.
The deterministic finite-dual error from Section~3 satisfies
\[
  \sup_{T\in K_*}
  |\mathcal F_{n,S}(T)-\mathcal F_S(T)|
  \le\varepsilon_n^{\mathrm{dual}}
  \qquad(S=S_+,S_5).
\]
Consequently the single root error used in the final theorem,
\begin{equation*}
  \omega_n^{\mathrm{root}}
  :=2\varepsilon_n^{\mathrm{dual}}
    +2C_{\mathrm{Lip}}\varepsilon_n^{\mathrm{target}}
  \longrightarrow0
  \tag{5.9c}
\end{equation*}
is deterministic and phase-uniform. Indeed,
\[
\begin{aligned}
  |D_{5,n}(T_+(n))-D_5(\delta)|
  \le{}&
  |\mathcal F_{n,S_+}(T_+)-\mathcal F_{S_+}(T_+)|\\
  &+|\mathcal F_{n,S_5}(T_+)-\mathcal F_{S_5}(T_+)|\\
  &+|\mathcal F_{S_+}(T_+)-\mathcal F_{S_+}(T_0)|\\
  &+|\mathcal F_{S_5}(T_+)-\mathcal F_{S_5}(T_0)|\\
  \le{}&\omega_n^{\mathrm{root}}.
\end{aligned}
\]
The quantitative finite-dual estimate in Section~3 and (5.9b) give
\[
  \omega_n^{\mathrm{root}}
  =O\!\left(\frac{\log\log n}{\log n}\right),
\]
with a phase-independent implied constant.

Substitution into (5.9a) yields the estimate used in the final theorem,
\begin{equation*}
  L_{S_5}(n,r_+)+(\log 2)r_+
  =r_+\{\log 2-D_5(\delta)+O(\omega_n^{\mathrm{root}})\}
  =r_+\{\log 2-D_5(\delta)+o(1)\}.
  \tag{5.10}
\end{equation*}

Finally, Lemma~\ref{lem:entropy-gap} gives
$\log 2-D_5(\delta)\ge\gamma_5>0$. Hence (5.10) implies
$\Phi_n(r_+)>0$ for all sufficiently large $n$, uniformly in the phase. The
normalized slope error in (3.7) tends to zero, so the derivative of $\Phi_n$
is positive throughout the common root corridor after one phase-independent
threshold. Since $\Phi_n(r_5^{\mathrm{co}})=0$, this proves
\[
  r_5^{\mathrm{co}}<r_+.
\]
The exponent gain becomes a root separation by division by the slope.
Both roots lie in (3.6), so the slope estimate holds throughout the interval
between them. The mean-value theorem gives a point
\(\xi_n\in(r_5^{\mathrm{co}},r_+)\) such that
\[
 \Phi_n(r_+)-\Phi_n(r_5^{\mathrm{co}})
 =\Phi_n'(\xi_n)(r_+-r_5^{\mathrm{co}}).
\]
The corridor estimate (3.6) gives
\[
  r_+
  =\frac{\log 2}{2}\frac n{\log n}
    \left\{1+O\!\left(\frac{\log\log n}{\log n}\right)\right\},
\]
Indeed, (2.1) and (3.5) give
$s_0=\alpha_0-1-2/q=(2/q)(L-\ell+C-1)=(2/q)L\{1+O(\ell/L)\}$, where
$L=\log n$ and $\ell=\log\log n$, and (3.6) gives $n/r_+=s_0+O(\ell/L)$;
dividing $n$ by $n/r_+$ proves the display. Equation~(3.7), uniformly at
\(\xi_n\), gives
\[
 \Phi_n'(\xi_n)=\frac2{\log 2}(\log n)^2
                 +O((\log n)(\log\log n)).
\]
Since $\Phi_n(r_5^{\mathrm{co}})=0$, (5.10) and the two displays above give
\[
 r_+-r_5^{\mathrm{co}}=\frac{\Phi_n(r_+)}{\Phi_n'(\xi_n)}
 =\frac{\frac{\log 2}2\frac n{\log n}
        \{A_5(\delta_n)+O(\omega_n^{\mathrm{root}})\}}
        {\frac2{\log 2}(\log n)^2}
  \Bigl\{1+O\Bigl(\frac{\log\log n}{\log n}\Bigr)\Bigr\}.
\]
As $A_5$ is bounded and $\omega_n^{\mathrm{root}}=O(\log\log n/\log n)$,
this is
\[
 r_+-r_5^{\mathrm{co}}
 =\frac{(\log 2)^2}{4}A_5(\delta_n)\frac n{(\log n)^3}
   +O\!\left(\frac{n\log\log n}{(\log n)^4}\right).
 \tag{5.11}
\]
The implied constant is independent of the phase: $A_5$ is bounded on
$[0,1]$, and all three relative errors above are uniform. This remainder
concerns the analytic roots. Rounding and amplification introduce further
errors in the probabilistic bound.

\Needspace{6\baselineskip}
For the fixed-offset selection below, we use the following coarse consequence
of (5.11): for all sufficiently large \(n\),

\[
 r_+-r_5^{\mathrm{co}}\ge\frac{(\log 2)^2\gamma_5}{8}\frac n{(\log n)^3}.     \tag{5.12}
\]

\subsection{An exact integer profile above the first-moment root for marked partitions}

Choose an integer a fixed distance above the first-moment root for marked partitions:

\[
 k_{\mathrm{co}}=\left\lceil r_5^{\mathrm{co}}\right\rceil+16. \tag{5.13}
\]

Thus
\[
  16\le k_{\mathrm{co}}-r_5^{\mathrm{co}}<17.
\]
By (5.12), $k_{\mathrm{co}}<r_+$ for all sufficiently large $n$.
Consequently the common corridor in
Lemma~\ref{lem:uniform-root-slope-dual} contains the entire interval
from $r_5^{\mathrm{co}}$ to $k_{\mathrm{co}}$.  Every real point of this
interval therefore has target deficit mean in the fixed compact interval
$K_*\subset(2,5)$ and is feasible for the real $S_5$ optimization.  In
particular this holds at the selected integer before the type counts are
rounded.  The convenient choice $16$ leaves a fixed explicit margin in (5.19).

We now round the optimizing real class counts to obtain the
\emph{selected integer profile}. The rounding must preserve both
finite-\(n\) conservation constraints: the total number of classes and
the total number of vertices.

At $k=k_{\mathrm{co}}$, let \(p_i^{(n)}\) be the normalized proportions
of the finite-\(n\) maximizing profile for \(L_{S_5}\). Let
\(\lambda_{n,\mathrm{def}}\) denote the Lagrange multiplier for the
deficit-mean constraint, with the sign convention that its contribution
is \(e^{\lambda_{n,\mathrm{def}}i}\).  The Lagrange equations and the
decomposition \eqref{eq:root-affine} then give

\[
 p_i^{(n)}\ \propto\ d_{\alpha-i}^{-1}e^{\lambda_{n,\mathrm{def}}i}
 \ \propto\ e^{\widehat\lambda_ni+h_n(i)},\qquad
 \widehat\lambda_n:=B_n+\lambda_{n,\mathrm{def}}.          \tag{5.14}
\]

The finite-support tilt \(\widehat\lambda_n\) includes the affine
contribution \(B_ni\).
The target deficit mean is $\alpha-n/k_{\mathrm{co}}$. By uniqueness in
Lemma~\ref{lem:uniform-root-slope-dual}, $\widehat\lambda_n$ is the finite-support
tilt $\lambda_{n,S_5}(\alpha-n/k_{\mathrm{co}})$ introduced there. The corridor estimate
and $k_{\mathrm{co}}-r_5^{\mathrm{co}}=O(1)$ give,
uniformly in the phase,
\[
 \left|\alpha-\frac n{k_{\mathrm{co}}}-T_0\right|
 =O\!\left(\frac{\log\log n}{\log n}\right)
  +O\!\left(\frac{(\log n)^2}{n}\right)
 \longrightarrow0.
\]
Lemma~\ref{lem:uniform-root-slope-dual} therefore makes these proportions
converge uniformly to (5.1), with
\(\widehat\lambda_n-\lambda_{S_5}(T_0)\to0\) uniformly in the phase.
Equivalently, the computation leading to (3.9c) shows that, for a real
profile with $k_{\mathrm{co}}$ parts and $n$ vertices, $L_{\mathbf k}$ equals
a function of $k_{\mathrm{co}}$ and $n$ alone plus $k_{\mathrm{co}}$ times
the bracket in (3.8a) at $p_i=k_i/k_{\mathrm{co}}$, with $S=S_5$ and mean
$T=\alpha-n/k_{\mathrm{co}}$. Hence $p_i^{(n)}=p_{n,i}(T)$, the finite
optimizer of Lemma~\ref{lem:uniform-root-slope-dual}, and (5.14) is its
Gibbs form.
Since $\alpha-n/k_{\mathrm{co}}\in K_*$, the last assertion of
Lemma~\ref{lem:uniform-root-slope-dual} gives (5.15) with $c_p:=\eta_*$:

\[
 \min_{2\le i\le6}p_i^{(n)}\ge c_p>0                    \tag{5.15}
\]
uniformly in the phase.

\Needspace{8\baselineskip}
For each sufficiently large $n$, fix a deterministic choice of preliminary integers
$\widetilde k_i$ with
$|\widetilde k_i-k_{\mathrm{co}}p_i^{(n)}|\le1$.
The estimates below hold for every such choice. Define the signed constraint errors

\[
 e_0=\sum_i\widetilde k_i-k_{\mathrm{co}},\qquad
 e_1=\sum_i i\widetilde k_i-(\alpha k_{\mathrm{co}}-n),
\]
so \(e_0,e_1=O(1)\), and add

\[
 \Delta k_2=e_1-3e_0,\qquad
 \Delta k_3=2e_0-e_1.                                   \tag{5.16}
\]

Set $\Delta k_4=\Delta k_5=\Delta k_6=0$, $k_i:=\widetilde k_i+\Delta k_i$ for
$2\le i\le6$, and $\mathbf k:=(k_2,k_3,k_4,k_5,k_6)$.

The correction removes both constraint errors by changing only \(O(1)\)
counts. Appendix~\ref{app:profile-rounding-proof} checks positivity and
shows that the loss in the counting exponent is \(O(1/k_{\mathrm{co}})\).

The resulting exact integer vector satisfies

\[
 u_i=\alpha-i,\qquad k_i=\Theta(n/{\log n})\quad(2\le i\le6).   \tag{5.17}
\]

Let $Z_{\mathbf k}^{\mathrm{sgn}}$ be the number of pairs consisting of a
partition with exactly $k_i$ classes of size $u_i$ for $2\le i\le6$ and a
declaration of each class as independent or complete, such that every
declaration is realized by $G_n$. Direct counting gives the exact identity

\[
 \Exp{Z_{\mathbf k}^{\mathrm{sgn}}}
 =2^{k_{\mathrm{co}}}\frac{n!}{\prod_i(u_i!)^{k_i}k_i!}
   2^{-\sum_i k_i\binom{u_i}{2}},                        \tag{5.18}
\]
Keeping every \(u_i!\) exact and applying (1.2) only to \(n!\) and the five
profile-multiplicity factorials \(k_i!\) then gives

\[
 \log \Exp{Z_{\mathbf k}^{\mathrm{sgn}}}
 =L_{S_5}(n,k_{\mathrm{co}})+{\log 2}k_{\mathrm{co}}+O({\log n})
 \ge c_Z(\log n)^2                                               \tag{5.19}
\]

for a phase-independent \(c_Z>0\). Appendix~\ref{app:profile-rounding-proof} verifies that the fixed
displacement above the root exceeds the rounding and Stirling losses. Every
\(u_i\) tends to infinity, so the independent and complete
declarations of one part are disjoint. Thus
\(Z_{\mathbf k}^{\mathrm{sgn}}>0\)
implies an actual cocoloring with \(k_{\mathrm{co}}\) parts.

This fixes a deterministic profile for the rest of the proof. The next
proposition collects the exact constraints, the first-moment bound and the
root separation that the second-moment argument will use.

\begin{proposition}[The selected five-size profile]
\label{prop:selected-profile-contract}
For each sufficiently large $n$, choose $k_{\mathrm{co}}$ by (5.13) and
$\mathbf k=(k_2,k_3,k_4,k_5,k_6)$ by the deterministic rounding and correction
above. There are phase-independent constants $c_*>0$, $c_Z>0$ and an
integer $n_0$ such that, for every integer $n\ge n_0$, the following hold.
\begin{enumerate}[label=\textup{(\roman*)}]
\item With $u_i=\alpha-i$ for $2\le i\le6$, the counts $k_i$ are positive
integers and satisfy exactly
\[
  \sum_{i=2}^6 k_i=k_{\mathrm{co}},
  \qquad
  \sum_{i=2}^6u_i k_i=n.
\]
Moreover, $k_i/k_{\mathrm{co}}\ge c_*$ for every $i$, and
$k_{\mathrm{co}}\asymp n/\log n$. The mean deficit satisfies
\[
 T:=\alpha-\frac{n}{k_{\mathrm{co}}}
 =T_0+O\!\left(\frac{\log\log n}{\log n}\right)\le39/10.
\]
\item For $G_n\sim G(n,1/2)$, the marked cocoloring count defined in (5.18)
obeys
\[
  \log\Exp{Z_{\mathbf k}^{\mathrm{sgn}}}\ge c_Z(\log n)^2.
\]
A positive count gives a cocoloring with exactly
$k_{\mathrm{co}}$ parts.
\item The ordinary and signed roots satisfy
\[
 r_+(n)-r_5^{\mathrm{co}}(n)
 =\frac{(\log 2)^2}{4}A_5(\delta_n)\frac n{(\log n)^3}
   +O\!\left(\frac{n\log\log n}{(\log n)^4}\right).
\]
For the integer chromatic lower location $k_\chi^-$ in (4.1),
\[
\begin{aligned}
 k_\chi^- -k_{\mathrm{co}}
 &=\left[\frac{(\log 2)^2}{4}A_5(\delta)+o(1)\right]
       \frac n{(\log n)^3}\\
 &\ge\left[\frac{(\log 2)^2\gamma_5}{4}-o(1)\right]
       \frac n{(\log n)^3}.
\end{aligned}                                                   \tag{5.20}
\]
\end{enumerate}
Here $\delta=\delta_n$. The choices, constants and error sequences are
deterministic. The implied constants and vanishing errors are uniform over
the induced phases, including integer sequences approaching either endpoint.
\end{proposition}

\begin{proof}
Equation (5.16) removes both rounding errors exactly.
Appendix~\ref{app:profile-rounding-proof} gives
$k_i\ge(c_p/2)k_{\mathrm{co}}$ after one phase-independent threshold;
take $c_*=c_p/2$. The root corridor and (5.13)
give the order of $k_{\mathrm{co}}$ and the displayed estimate for $T$.
Since $q>56/81$, one has
$T_0\le1+2/q<109/28=39/10-1/140$. The displayed error therefore gives
$T\le39/10$ after one phase-independent threshold. This window concerns
the selected target, not every point of $K_*$. Equation (5.19) proves the first-moment
bound, and the paragraph following it identifies a realized marked cocoloring
with a cocoloring. The root separation is (5.11). Finally, (4.1) and (5.13)
change its two roots by $O(\log n)$ and $O(1)$ respectively, giving (5.20).
\end{proof}

\section{Second moment of the marked cocoloring count}\label{exact-signed-second-moment-representation}

A large first moment alone does not give a useful probability of a
cocoloring. The second-moment estimates in Sections~6--9 supply that
probability through a deterministic exponent $\Lambda_n$ satisfying
\[
  \frac{\mathbb E[(Z_{\mathbf k}^{\mathrm{sgn}})^2]}
       {(\mathbb E Z_{\mathbf k}^{\mathrm{sgn}})^2}
  \le e^{\Lambda_n}
\]
and give the scale hierarchy
\[
  k_{\mathrm{co}}\asymp\frac n{\log n},\qquad
  r_+-r_5^{\mathrm{co}}\asymp\frac n{(\log n)^3},\qquad
  \Lambda_n=o\!\left(\frac n{(\log n)^4}\right).
\]
Once the last estimate is available, its amplification cost (by vertex
exposure the uncovered set has expected size at most
$\sqrt{(n-1)\Lambda_n/2}$, \eqref{eq:amplification-uncovered-mean}, and
(10.3) colors any vertex set $S$ with $C_0|S|/\log n+n^{1/3}$ colors)
satisfies
\[
  \frac{\sqrt{n\Lambda_n}}{\log n}
  =o\!\left(\frac n{(\log n)^3}\right),
\]
which is smaller than the root separation.

Fix the deterministic five-size profile from
Proposition~\ref{prop:selected-profile-contract}. The marked cocoloring count
has a large first moment. To bound from below the probability that it is
positive, we must determine which edge constraints two partitions share and
whether their class marks, which we also call signs, prescribe consistent
constraints. We begin by counting the compatible pairs of markings. For partitions sharing exactly \(\ell\) whole classes,
\citet[Proposition~6]{heckel-2025-difference} used the bound
\(2^{2k-\ell}\) to transfer selected ordinary-coloring estimates to
cocolorings. Here we retain the full overlap matrix. Lemma~\ref{lemma-6.1-exact-sign-sum} counts the
compatible markings using the connected components formed by cells of size
at least two, including cells that are whole classes in both partitions.

\subsection*{Overlap law and compatible signs}

Temporarily distinguish the class slots within each of the five types.
This multiplies the marked partition count by the deterministic factor
\(\prod_i k_i!\), and therefore leaves its normalized second moment unchanged. Let $\mathcal I_{\mathrm{row}}$ and $\mathcal I_{\mathrm{col}}$ be the row- and column-slot index
sets. Choose two independent uniform ordered partitions of the profile, write
them as $(V_a)_{a\in\mathcal I_{\mathrm{row}}}$ and
$(W_b)_{b\in\mathcal I_{\mathrm{col}}}$, and set
$s_a:=|V_a|$ and $t_b:=|W_b|$. Define
\[
  \mathbf s:=(s_a)_{a\in\mathcal I_{\mathrm{row}}},
  \qquad
  \mathbf t:=(t_b)_{b\in\mathcal I_{\mathrm{col}}},
\]
and let $\mathcal R(\mathbf s,\mathbf t)$ be the finite set of nonnegative
integer matrices with row margins $\mathbf s$ and column margins $\mathbf t$.
Put

\[
 r_{ab}=|V_a\cap W_b|.                                   \tag{6.1}
\]

The overlap matrix has the exact law

\[
 p(r)=\frac{\prod_a s_a!\prod_b t_b!}
            {n!\prod_{a,b}r_{ab}!},                      \tag{6.2}
\]

Both degree lists are the multiset with $k_i$ copies of $u_i$,
$2\le i\le6$. We can obtain the same law
by placing \(s_a\) stubs at each row slot and \(t_b\) stubs at each column
slot, then choosing a uniform perfect matching of the two \(n\)-stub sets.
Each fixed ordered column partition arises from the same number,
$\prod_b t_b!$, of column-stub orderings. Thus this configuration model
induces exactly the law (6.2).

Let \(H(r)\) be the simple bipartite graph whose edges are cells
with \(r_{ab}\ge 2\), omitting isolated slot vertices.
Write $H:=H(r)$ and

\[
 W=\sum_{a,b}\binom{r_{ab}}2,\qquad
 \beta(H)=|E(H)|-|V(H)|+c(H).                             \tag{6.3}
\]
We use the standard convention \(c(\varnothing)=0\), so that
\(\beta(\varnothing)=0\).

Figure~\ref{fig:signed-support-mechanism} illustrates how cells of size
at least two determine the support graph, and how its cycle-space factor combines
with the local cell factors.

\begin{figure}[htbp]
\centering
\begin{tikzpicture}[
  x=0.70cm,
  y=0.62cm,
  >=Latex,
  every node/.style={font=\scriptsize},
  entry/.style={minimum width=0.58cm,minimum height=0.45cm,inner sep=0pt},
  cyclecell/.style={fill=blue!10,rounded corners=0.8pt},
  treecell/.style={fill=orange!16,rounded corners=0.8pt},
  vertex/.style={circle,draw=black,fill=white,inner sep=1.8pt},
  cycleedge/.style={draw=blue!60!black,line width=0.9pt},
  treeedge/.style={draw=orange!80!black,line width=0.9pt,dashed}
]
  \begin{scope}
    \foreach \x/\y/\z/\style in {
      0.45/2.55/2/cyclecell,1.35/2.55/2/cyclecell,2.25/2.55/0/,
      0.45/1.65/2/cyclecell,1.35/1.65/3/cyclecell,2.25/1.65/0/,
      0.45/0.75/0/,1.35/0.75/0/,2.25/0.75/2/treecell}{
      \node[entry,\style] at (\x,\y) {$\z$};
    }
    \draw[line width=0.55pt]
      (-0.02,0.38)--(-0.18,0.38)--(-0.18,2.92)--(-0.02,2.92);
    \draw[line width=0.55pt]
      (2.72,0.38)--(2.88,0.38)--(2.88,2.92)--(2.72,2.92);
    \foreach \y/\a in {2.55/1,1.65/2,0.75/3}{
      \node[left] at (-0.35,\y) {$a_{\a}$};
    }
    \foreach \x/\b in {0.45/1,1.35/2,2.25/3}{
      \node at (\x,3.25) {$b_{\b}$};
    }
    \node[font=\footnotesize] at (1.35,3.92) {overlap table $r$};
  \end{scope}

  \draw[->,line width=0.65pt] (3.35,1.65) -- (5.55,1.65)
    node[midway,above,align=center] {$r_{ab}\ge2$};

  \begin{scope}[shift={(6.15,0)}]
    \node[vertex] (a1) at (0,2.35) {$a_1$};
    \node[vertex] (b1) at (2,3.10) {$b_1$};
    \node[vertex] (a2) at (4,2.35) {$a_2$};
    \node[vertex] (b2) at (2,1.60) {$b_2$};
    \node[vertex] (a3) at (1,0.48) {$a_3$};
    \node[vertex] (b3) at (3,0.48) {$b_3$};
    \draw[cycleedge] (a1)--(b1)--(a2)--(b2)--cycle;
    \draw[treeedge] (a3)--(b3);
    \node[font=\footnotesize] at (2,3.92) {support graph $H(r)$};
    \node[align=center] at (2,-0.24)
      {$|E|=5,\ |V|=6,\ c=2,\ \beta=1$};
  \end{scope}

  \draw[->,line width=0.65pt] (10.75,1.65) -- (12.45,1.65);
  \node[align=left,anchor=west,font=\footnotesize] at (12.80,1.70)
    {$\displaystyle W=7,\qquad \text{local factor}=4,$\\[0.45ex]
     $\displaystyle \text{normalized factor}
       =2^{7+2-6}=8=4\cdot2^1.$};
\end{tikzpicture}
\caption{From an overlap table to the signed support factor. The shaded entries
are precisely those with multiplicity at least two. Their support consists of
one four-cycle (solid blue) and one isolated edge (dashed orange), so its binary
cycle space has dimension one. Compatible signs are constant on each connected
component; the colors and line styles mark the two components and carry no sign
information. The displayed arithmetic verifies (6.4) for this example.}
\label{fig:signed-support-mechanism}
\end{figure}

For a fixed pair of partitions, divide the expected number of compatible
pairs of valid markings by the square of the one-partition expectation.
The next lemma expresses this normalized factor in terms of the overlap
table.

\begin{lemma}[Exact sign sum]
\label{lemma-6.1-exact-sign-sum}

For every overlap matrix $r\in\mathcal R(\mathbf s,\mathbf t)$, the normalized
local factor of the corresponding pair of partitions is

\[
 A_\zeta(r)=2^{W+c(H)-|V(H)|}
 =\left(\prod_{a,b}g(r_{ab})\right)2^{\beta(H)},          \tag{6.4}
\]
where

\[
 g(0)=g(1)=g(2)=1,\qquad
 g(x)=2^{\binom x2-1}\quad(x\ge3).                       \tag{6.5}
\]

Consequently

\[
 \frac{\Exp{(Z_{\mathbf k}^{\mathrm{sgn}})^2}}
      {\Exp{Z_{\mathbf k}^{\mathrm{sgn}}}^2}
 =\sum_{r\in\mathcal R(\mathbf s,\mathbf t)}
      p(r)A_\zeta(r).                                    \tag{6.6}
\]
Here the sum ranges over the matrix set $\mathcal R(\mathbf s,\mathbf t)$
defined above.

\end{lemma}

\begin{proof}
Fix the two partitions. Let \(B=\sum_a \binom{s_a}{2}\) be the number of
edge indicators prescribed by each partition. A cell of size at least
two contains an edge whose presence is prescribed by both its row and its
column. Their signs must therefore agree. Repeating this condition along
the support edges forces a constant sign on each component of \(H\).
Conversely, every edge bit prescribed by both partitions lies in a cell of
size at least two, so signs that are constant on the components of \(H\)
prescribe it consistently. \Needspace{5\baselineskip}The number of compatible pairs
of sign assignments, including the free signs on slots isolated from $H$, is

\[
 2^{2k_{\mathrm{co}}-|V(H)|+c(H)}.
\]

For each compatible pair, \(W\) edge bits are prescribed twice, so the
probability of all constraints is
\(2^{-(2B-W)}\). Division by the square of the expected number
of valid markings of one fixed partition
\((2^{k_{\mathrm{co}}}2^{-B})^2\)
gives

\[
 \frac{2^{2k_{\mathrm{co}}-|V(H)|+c(H)}2^{-(2B-W)}}
      {(2^{k_{\mathrm{co}}}2^{-B})^2}
 =2^{W+c(H)-|V(H)|},
\]

which is the first expression in (6.4).

For every support edge, split \(\binom{r}{2}\) as
\((\binom{r}{2}-1)+1\). Since
\(|E|-|V|+c=\beta\),

\[
 \prod_{a,b}g(r_{ab})=2^{W-|E(H)|},
\]

and multiplying by \(2^{\beta(H)}\) gives
the second expression.

To average, let $\mathcal P$ be the finite set of ordered partitions with
the selected slot sizes, and for $P\in\mathcal P$ let $Z_P$ be the number
of valid markings of $P$, so that $\mathbb E Z_P=2^{k_{\mathrm{co}}}2^{-B}$
for every $P$. By the computation above,
$\mathbb E[Z_PZ_Q]=(2^{k_{\mathrm{co}}}2^{-B})^2A_\zeta(r(P,Q))$, where
$r(P,Q)$ is the overlap matrix of $P$ and $Q$. Expanding the square of
$\sum_PZ_P$ therefore gives
\[
 \frac{\mathbb E[(\sum_PZ_P)^2]}{(\mathbb E\sum_PZ_P)^2}
 =\frac1{|\mathcal P|^2}\sum_{P,Q\in\mathcal P}A_\zeta(r(P,Q)),
\]
the average of $A_\zeta(r(P,Q))$ over independent uniform $P$ and $Q$.
Averaging over (6.2) proves (6.6).
\end{proof}

Call an edge set even if every vertex has even degree in that edge set.
The identity

\[
 2^{\beta(H)}
 =\#\{F\subseteq E(H):\deg_F(v)\text{ is even for every }v\} \tag{6.7}
\]

counts these sets: they form the binary cycle space, whose dimension is
\(\beta(H)\). In Section~9 we will restrict an even set to the edges outside
the exposed large-cell matching. That restriction is injective, allowing
us to bound the cycle-space and local-cell contributions together.

\subsection*{Joint multiplicity bounds for overlap cells}

The overlap cells are dependent, so separate one-cell probability bounds
are not enough. We instead count stub pairs that meet all the prescribed
cell demands at once. The next estimate bounds that joint event; it is
used only in Section~\ref{sec:residual-attachments}.

\begin{lemma}[Joint prescribed-cell bound]
\label{lemma-6.2-joint-prescribed-cell-bound}

Let $\mathcal I$ and $\mathcal J$ be finite row- and column-index sets, and
let $(d_a)_{a\in\mathcal I}$ and $(d'_b)_{b\in\mathcal J}$ be nonnegative
integer degree lists satisfying
\[
  \sum_{a\in\mathcal I}d_a=
  \sum_{b\in\mathcal J}d'_b=m_0.
\]
Choose a uniform perfect matching of the two $m_0$-stub sets, write
$\mathbb P_{\mathrm{match}}$ for its law, and let
$r_{ab}$ be the number of matched pairs in cell $(a,b)$. For a finite set
$D\subseteq\mathcal I\times\mathcal J$, prescribe integer demands
$x_{ab}\ge1$ on $D$. Extend $x_{ab}=0$ off $D$, and put

\[
 x=\sum_Dx_{ab},\quad D_a=\sum_bx_{ab},\quad
 D'_b=\sum_ax_{ab}.
\]

In the feasible case

\[
 x\le m_0,\qquad D_a\le d_a\ \ (\forall a),\qquad
 D'_b\le d'_b\ \ (\forall b),
\]

we have

\[
 \mathbb P_{\mathrm{match}}(r_{ab}\ge x_{ab}\ \text{for every }(a,b)\in D)
 \le\frac{\prod_a(d_a)_{D_a}\prod_b(d'_b)_{D'_b}}
          {(m_0)_x\prod_{(a,b)\in D}x_{ab}!}.             \tag{6.8}
\]

If any feasibility condition fails, the prescribed event is empty and is
handled separately. In particular, in all cases with \(m_0>0\), with
\(\mathrm{e}:=\exp(1)\) and
\(\theta_{ab}=\mathrm{e}\, d_a d'_b/m_0\),
its probability is at most

\[
 \prod_{(a,b)\in D}\frac{\theta_{ab}^{x_{ab}}}{x_{ab}!}.  \tag{6.9}
\]

\end{lemma}

\begin{proof}
If a feasibility condition fails, no matching realizes all demands. In the
feasible case, assigning disjoint row stubs to the demanded cells gives
\[
  \prod_a\frac{(d_a)_{D_a}}{\prod_b x_{ab}!}
\]
choices. Independently selecting the column stubs gives
\[
  \prod_b\frac{(d'_b)_{D'_b}}{\prod_a x_{ab}!}
\]
choices. For each cell \(ab\), the selected row and column stubs can be paired
in \(x_{ab}!\) ways. Hence the product over all cells cancels one of the two
copies of \(\prod_{(a,b)\in D}x_{ab}!\), and the total number of choices of prescribed stub pairs
is

\[
 \frac{\prod_a(d_a)_{D_a}\prod_b(d'_b)_{D'_b}}
      {\prod_{(a,b)\in D}x_{ab}!}.
\]

A fixed set of stub pairs occurs with probability \(1/(m_0)_x\): after
its \(x\) pairs are exposed, the remaining matching is uniform on the unused
stubs. Summing over the possible prescribed pair sets gives (6.8) by a
union bound.

It remains to simplify the global denominator.  For \(1\le x\le m_0\),

\[
 \binom{m_0}{x}
 =\prod_{j=0}^{x-1}\frac{m_0-j}{x-j}
 \ge\left(\frac{m_0}{x}\right)^x,
 \qquad
 x!\ge\left(\frac{x}{\mathrm{e}}\right)^x.
\]

The factorial inequality follows, for example, from
\(\sum_{j=1}^x\log j\ge\int_1^x\log t\,dt=x\log x-x+1\).
Consequently

\[
 (m_0)_x=x!\binom{m_0}{x}\ge(m_0/\mathrm{e})^x,          \tag{6.10}
\]

while \((m_0)_0=1=(m_0/\mathrm e)^0\) covers \(x=0\). Combining (6.10) with
\((d)_r\le d^r\) proves (6.9). Equation~(6.8) retains the single global falling
factorial before the coarser bound (6.9) is taken.
\end{proof}

\section{Weights of common whole classes}\label{exact-partial-diagonals}

Fix the profile of Proposition~\ref{prop:selected-profile-contract}.
A specified collection of whole classes common to both partitions is a
\emph{partial diagonal}. In each selected pair, the row and column classes
have the same vertex set. We call such a choice a \emph{selection}.

For each vector of selected class counts, let $D(\ell)$ be the sum of the
weights over its selections. Section~8 uses these common-class weights
as the reference for large overlap cells. We prove that their sum over all
subprofiles is $1+o(1)$ uniformly in the phase.

For a selected common subprofile \(\ell=(\ell_i)\), with
$i\in\{2,3,4,5,6\}$, put

\[
 \ell_\bullet=\sum_i\ell_i,\qquad
 m=\sum_i u_i\ell_i,\qquad
 b_\ell=\sum_i\binom{u_i}{2}\ell_i.                           \tag{7.1}
\]

Here \(0\le\ell_i\le k_i\) coordinatewise.  Since the full profile has
\(\sum_i u_i k_i=n\), this automatically gives \(m\le n\); thus every
factorial in the counting formula below has its ordinary combinatorial
meaning.

\paragraph{Marked weights.}
Summed over all choices of an $m$-vertex support in $[n]$, the first moment of the marked count of these partial partitions is

\[
 Y_{\ell}^{\mathrm{sgn}}
 =2^{\ell_\bullet}\frac{n!}{(n-m)!\prod_i(u_i!)^{\ell_i}\ell_i!}
       2^{-b_\ell}.                                           \tag{7.2}
\]

For each type $i$, choose $\ell_i$ row slots, $\ell_i$ column slots and a
bijection between them, in
$\binom{k_i}{\ell_i}^2\ell_i!=(k_i)_{\ell_i}^2/\ell_i!$ ways. For two
independent uniform ordered partitions, the matched slots have identical
vertex sets with probability $\prod_i(u_i!)^{\ell_i}/(n)_m$. Each such
common class is an isolated edge of $H(r)$, so by (6.4) it carries the
factor $g(u_i)=2^{\binom{u_i}2-1}$. The product, summed over the slot
choices, is the exact \emph{marked} common-subprofile weight
\[
 D(\ell)=\prod_i\frac{(k_i)_{\ell_i}^2\,(u_i!\,g(u_i))^{\ell_i}}{\ell_i!}
   \cdot\frac1{(n)_m}
 =\frac{\prod_i\binom{k_i}{\ell_i}^2}{Y_\ell^{\mathrm{sgn}}}.          \tag{7.3}\label{eq:marked-common-subprofile-weight}
\]

Equivalently, $D(\ell)$ is the expected total weight, for two independent
uniform ordered partitions, of the selections of $\ell_i$ row classes and
$\ell_i$ column classes of each type that are paired with identical vertex
sets, where a selection has weight $\prod_ig(u_i)^{\ell_i}$, the product of
the local factors of its common classes.

The weight $D(\ell)$ accounts only for the selection of whole common
classes and their local signed factors. The other cell factors and the
cycle-space factor remain to be estimated in Sections~8--9.

Different selections need not be disjoint: an overlap containing several
common blocks contributes to each corresponding selection. The sums below
are sums of these nonnegative weights, so no disjointness is used.
In particular,
\(D(0)=1\) and \(D(\mathbf k)=1/\Exp{Z_{\mathbf k}^{\mathrm{sgn}}}\).
Section~\ref{app:common-recurrences} derives two exact recurrences for these
weights: a forward one, which adds one common class to a selection, and a
reverse one, which starts from the complete profile and removes classes.

In Heckel's admissible-phase argument, the corresponding control of partial
profiles enters through the tame-coloring estimates of
\citet[Lemmas~5.1--5.3, 6.3--6.5, and~7.20]{heckel-panagiotou-2023},
as used in \citet[Proposition~5]{heckel-2025-difference}; their hypotheses are the
source of the lower \(\mu_\alpha\)-window in that theorem.  The next lemma
supplies the uniform five-size estimate needed here, including both corners
and every intermediate mass.

\begin{lemma}[All common subprofiles]
\label{lemma-7.1-all-common-subprofiles}

For the selected integer profile, uniformly over the full phase interval,

\[
 \sum_{0\le\ell_i\le k_i}D(\ell)=1+o(1),               \tag{7.7}
\]

where the error is bounded by one deterministic sequence tending to zero.

\end{lemma}

\emph{Proof outline.} The empty selection contributes one. We bound all
other terms in three ranges: the forward recurrence \eqref{eq:common-forward-recurrence} handles small
selected mass, a negative leading rate handles intermediate mass, and the
reverse recurrence \eqref{eq:common-reverse-recurrence} handles a nearly full selected profile. In the last
range, the large first moment of the complete profile supplies the decay.

The main point in the intermediate range is that the leading exponential
rate is negative. Put \(K=k_{\mathrm{co}}\), \(p_i=k_i/K\), \(y_i=\ell_i/K\),
and \(z_i=p_i-y_i\). Define
\[
 Y=\sum_i y_i,\qquad R=\sum_i z_i=1-Y,\qquad
 T=\sum_i i p_i,\qquad I_r=\sum_i i z_i.
\]
Thus \(Y\) is the selected class fraction and \(R\) is the residual class
fraction. The exact weight formula and the factorial estimates in
Section~\ref{app:common-subprofile-details} give, when the residual vertex
fraction \(\rho=(n-m)/n\) exceeds \(1/32\),
\[
 \log D(\ell)\le K\alpha\Phi_T
       +C_{\mathrm{diag}}KY\log(\mathrm e/Y)+C_{\mathrm{diag}}\log n,\qquad
 \Phi_T=R\log R+\frac{\log2}{2}(I_r-TR),
\]
with a phase-independent constant \(C_{\mathrm{diag}}\).
All deficits lie in $\{2,3,4,5,6\}$, so
\[
 I_r-TR\le(6-T)R,\qquad I_r-TR\le(T-2)(1-R).
\]
Multiplying by $1-R$ and $R$, respectively, and adding gives
$I_r-TR\le4R(1-R)$. Also $\rho>1/32$ implies $R\ge1/64$ eventually,
and Proposition~\ref{prop:selected-profile-contract} gives $T\le39/10$.
Lemma~\ref{lem:five-size-rate} therefore yields
$\Phi_T\le-Y/10000$ throughout this range.

Put $\eta:=\log\log n/(32\log n)$. In the intermediate range $m>\eta n$,
$n-m>n/32$, one has $Y\ge\eta/2$ eventually, so the two error terms are
$o(K\alpha Y)$ uniformly in the phase (see \eqref{eq:G-diag-error}). Hence
$\log D(\ell)\le-K\alpha\eta/40000$ eventually, and since there are at most
$(K+1)^5$ subprofiles and $K\alpha\eta/\log n\to\infty$, their total
contribution tends to zero. The two recurrences cover the other ranges.

Section~\ref{app:common-subprofile-details} gives the three ranges
deterministic, phase-independent error sequences
$\varepsilon_n^{\mathrm{empty}}$, $\varepsilon_n^{\mathrm{central}}$,
$\varepsilon_n^{\mathrm{full}}\to0$ and proves the quantitative form of (7.7)
used in Section~\ref{sec:residual-attachments}:
\begin{equation*}
  1\le\sum_{0\le\ell_i\le k_i}D(\ell)
  \le e^{\varepsilon_n^{\mathrm{empty}}}+\varepsilon_n^{\mathrm{central}}
     +\varepsilon_n^{\mathrm{full}}.
  \tag{7.7a}\label{eq:common-subprofile-quantitative}
\end{equation*}

\section{Large overlap cells and endpoint comparisons}
\label{sec:canonical-high-cells}

A large overlap cell need not contain a whole class. Our reference is
full containment, where every stub of the smaller endpoint class is
matched inside the cell. For a fixed set of large cells, we first sum over
their labeled stub realizations. The remaining costs come from missing
stubs in these cells and unequal endpoint class sizes. We sum over the
missing-stub deficits, then group the full-containment weights by endpoint
type to control the unequal sizes.
The outcome is Proposition~\ref{prop:bare-high-skeleton}: summed over all
large-cell configurations, these weights are at most
$\exp\{O(n^{3/4}\log n)\}$, well inside the budget $o(n/(\log n)^4)$ for
$\Lambda_n$ from Section~\ref{the-five-size-signed-first-moment-advantage}.

The common-class weights from Section~7 have total mass $1+o(1)$.
Here we sum the weights of cells larger than half the class-size cap;
Section~9 bounds the remaining local factors and the cycle-space factor.

\subsection{Large cells and their labeled realizations}
\label{subsec:large-cell-realizations}

Let $U:=\alpha-2$ be the largest class size in the fixed five-size profile and put
\[
  R_0:=\left\lfloor\frac U2\right\rfloor.
\]
A cell is \emph{large} if its multiplicity exceeds $R_0$. Two large cells cannot
share a row or a column: otherwise that row or column would contain more than
$U$ stubs, because $2(R_0+1)>U$. Hence all large cells form a canonical
bipartite matching, denoted by $M$.

The matching $M$ records which pairs of class slots form large cells.
A labeled stub realization specifies the individual stub pairs within
those cells. The endpoints of $M$ are disjoint, so choosing pairs in one
cell of $M$ uses none of the stubs available to another. The realization
counts therefore multiply.

For $e\in M$, let $s_e$ and $t_e$ be its endpoint block sizes and set
\[
  m_e:=\min\{s_e,t_e\},
  \qquad
  d_e:=|s_e-t_e|,
  \qquad
  h_e:=m_e-j_e,
\]
where $j_e$ is the realized multiplicity. Thus $m_e$ is the full-containment
multiplicity, $h_e$ is the deficit, and $j_e=m_e-h_e$. Since $j_e>U/2$ and
$m_e\le U$,
\begin{equation}
  2h_e<m_e.
  \label{eq:half-deficit}
\end{equation}
Put $J:=\sum_{e\in M}j_e$.

For fixed $(M,j)$, let $\mathfrak F(M,j)$ be the finite set of labeled partial
stub matchings containing exactly $J$ pairs, all in cells of $M$, with exactly
$j_e$ pairs in each cell $e\in M$.  Thus
\(\mathfrak F(M,j)\) records the labeled pairs, not just their cell counts.
Let $\Pi$ be the uniform perfect matching of the two $n$-stub sets from
Section~\ref{exact-signed-second-moment-representation}; a fixed set of $J$
stub pairs lies in $\Pi$ with probability $1/(n)_J$. For
$0\le j_e\le\min\{s_e,t_e\}$ put
\[
  w(M,j):=\sum_{S\in\mathfrak F(M,j)}\mathbb P(S\subseteq\Pi)
          \prod_{e\in M}g(j_e).
\]

A size-$j$ partial matching between labeled blocks of sizes $s$ and $t$ can be
chosen in
\[
  \binom sj\binom tj j!
  =
  \frac{(s)_j(t)_j}{j!}
\]
ways: choose its two endpoint sets and then the bijection between them.

\begin{proposition}[Aggregate weight of a fixed cell matching]
\label{prop:aggregate-high-weight}
Let $M$ be a fixed matching support arising from two copies of the fixed
feasible $n$-vertex profile. For a multiplicity vector
$j=(j_e)_{e\in M}$ satisfying
$0\le j_e\le\min\{s_e,t_e\}$ for every $e\in M$, one has $J\le n$, and the
total configuration probability multiplied by the exposed local factor, summed
over \(\mathfrak F(M,j)\), is
\begin{equation}
  w(M,j)
  =
  \frac{
    \displaystyle\prod_{e\in M}(s_e)_{j_e}(t_e)_{j_e}
  }{
    \displaystyle(n)_J\prod_{e\in M}j_e!
  }
  \prod_{e\in M}g(j_e).
  \label{eq:aggregate-high-weight}
\end{equation}
\end{proposition}

\begin{proof}
Because $M$ is a matching, distinct selected cells use disjoint row and column
stub sets. Therefore
\[
  |\mathfrak F(M,j)|
  =\prod_{e\in M}\frac{(s_e)_{j_e}(t_e)_{j_e}}{j_e!}.
\]
A prescribed set of $J$ stub pairs occurs with probability $(n)_J^{-1}$ in the
uniform bipartite configuration matching. Multiplying this probability by the
displayed cardinality and by the exposed local factor
$\prod_{e\in M}g(j_e)$ gives \eqref{eq:aggregate-high-weight}.
The sum is over prescribed partial matchings and can count one full
matching more than once. Section~9 enforces exact cell multiplicities
through the residual-completion condition.
\end{proof}

A pair $(M,j)$ is \emph{feasible} when $M$ is a matching between classes and
\[
  R_0<j_e\le m_e\qquad(e\in M).
\]
For a feasible pair $(M,j)$, write
\[
  w_{\mathrm{hi}}(M,j):=w(M,j).
\]
The same exposed weight appears in the exact conditional decomposition
of Section~9.

\begin{remark}
The matching hypothesis is essential. If two selected cells shared a row or
a column, their endpoint choices would compete for the same stubs and the
one-cell realization counts would not multiply.
\end{remark}

Figure~\ref{fig:canonical-overlap-decomposition} shows how we split the
overlap table. To recover the full labeled matching, we also retain the
exposed stub pairs, as in
Proposition~\ref{prop:canonical-exact-overlap-decomposition}. The last
line separates the exposed weight $w_{\mathrm{hi}}$, defined above, from
the residual factor $\mathcal A$, which accounts for the unexposed matching
and is defined in Section~9. We bound these two factors separately.

\begin{figure}[htbp]
\centering
\begin{tikzpicture}[
  x=0.50cm,
  y=0.48cm,
  >=Latex,
  every node/.style={font=\scriptsize},
  highcell/.style={draw=blue!55!black,fill=blue!10,line width=0.65pt}
]
  \begin{scope}[shift={(0,0)}]
    \draw[gray!65] (0,0) grid (5,4);
    \foreach \x/\y/\z in {0/3/j_1,2/2/j_2,4/0/j_3}{
      \filldraw[highcell] (\x,\y) rectangle ++(1,1);
      \node at (\x+0.5,\y+0.5) {$\z$};
    }
    \node[font=\footnotesize] at (2.5,4.65) {$r=(r_{ab})$};
    \node[align=center,text width=3.2cm] at (2.5,-0.85)
      {full overlap table};
  \end{scope}

  \node[font=\large] at (6.0,2) {$=$};

  \begin{scope}[shift={(7,0)}]
    \draw[gray!65] (0,0) grid (5,4);
    \foreach \x/\y/\z in {0/3/j_1,2/2/j_2,4/0/j_3}{
      \filldraw[highcell] (\x,\y) rectangle ++(1,1);
      \node at (\x+0.5,\y+0.5) {$\z$};
    }
    \node[font=\footnotesize] at (2.5,4.65) {$j_M$};
    \node[align=center,text width=3.4cm] at (2.5,-0.85)
      {large cells supported on $M$, with $j=r|_M$};
  \end{scope}

  \node[font=\large] at (13.25,2) {$+$};

  \begin{scope}[shift={(14.5,0)}]
    \draw[gray!65] (0,0) grid (5,4);
    \foreach \x/\y in {0/3,2/2,4/0}{
      \filldraw[highcell] (\x,\y) rectangle ++(1,1);
      \node at (\x+0.5,\y+0.5) {$0$};
    }
    \node[text=gray!75!black] at (1.5,1.5) {$\le R_0$};
    \node[text=gray!75!black] at (3.5,3.5) {$\le R_0$};
    \node[font=\footnotesize] at (2.5,4.65) {$r'=(r'_{ab})$};
    \node[align=center,text width=3.5cm] at (2.5,-0.85)
      {capped residual $r'$, with $r'|_M=0$};
  \end{scope}

  \node[font=\small] at (9.75,-3.10) {$\displaystyle
    \frac{\mathbb E\!\left[(Z_{\mathbf k}^{\mathrm{sgn}})^2\right]}
         {\left(\mathbb E Z_{\mathbf k}^{\mathrm{sgn}}\right)^2}
    =\sum_{(M,j)}
      \underbrace{w_{\mathrm{hi}}(M,j)}_{\text{large-cell factor}}
      \qquad
      \underbrace{\mathcal A(M,j)}_{\text{residual factor}}$};
\end{tikzpicture}
\caption{Canonical large-cell extraction. The highlighted cells
\(M=\{(a,b):r_{ab}>R_0\}\) form a matching, with multiplicities \(j=r|_M\).
Extending \(j\) by zero gives \(r=j_M+r'\), where \(r'\) vanishes on \(M\)
and is capped by \(R_0\) elsewhere. The displayed identity is
Proposition~\ref{prop:canonical-exact-overlap-decomposition}.}
\label{fig:canonical-overlap-decomposition}
\end{figure}

Proposition~\ref{prop:canonical-exact-overlap-decomposition} lifts this
table decomposition to labeled configuration matchings by retaining the
selected-pair set \(S\in\mathfrak F(M,j)\) and the residual matching
\(\Pi_{\mathrm{res}}\). This gives the normalized factorization shown in
the figure. Blank cells have no displayed multiplicity. Outside \(M\), the
original cell equals its residual cell in \(r'\), is bounded by \(R_0\),
and has zero entry in \(j_M\).

\subsection{Summing the deficits}
\label{subsec:deficit-sum}

We first fix the matching $M$ and sum over the multiplicities of its cells,
comparing each large cell with full containment.
Write
\[
  \mathbf m:=(m_e)_{e\in M},
  \qquad
  w_{\mathrm{full}}(M):=w(M,(m_e)_{e\in M}),
  \qquad
  J_*:=\sum_{e\in M}m_e,
  \qquad
  h_\bullet:=J_*-J=\sum_{e\in M}h_e.
\]

We label the five selected class sizes by \emph{types} $i\in\{0,1,2,3,4\}$:
type $i$ has size $u_i:=\alpha-2-i$ and $\kappa_i:=k_{i+2}$ classes.
Through Section~\ref{sec:residual-attachments}, $u_i$ is this shifted size.
Assume $\alpha>8$; since $\alpha>\alpha_0-1$, this holds beyond a
phase-independent threshold. For
$i,j\in\{0,1,2,3,4\}$ let $m_{ij}:=\min\{u_i,u_j\}$ and put
\[
  \rho_{ij}
  :=
  \frac{n m_{ij}}{2^{\lfloor(3m_{ij}-1)/4\rfloor}},
  \qquad
  \rho_{\Sigma}:=\sum_{i,j=0}^{4}\rho_{ij}.
\]

\begin{proposition}[Fixed-support all-deficit bound]
\label{prop:fixed-support-all-deficit}
Write $\mathbf m=(m_e)_{e\in M}$. The sum below runs over vectors
$h=(h_e)_{e\in M}$ of nonnegative integers satisfying $m_e-h_e>R_0$ for every
$e\in M$. If $\rho_{\Sigma}\le1$, then
\begin{equation}
  \sum_h
    w(M,\mathbf m-h)
  \le
  w_{\mathrm{full}}(M)
  \left(1+(\alpha+1)\rho_{\Sigma}\right)^{|M|}.
  \label{eq:fixed-support-all-deficit}
\end{equation}
\end{proposition}

Section~\ref{app:deficit-sum} proves the proposition. Lowering the
multiplicity of a large cell with endpoint sizes $m\le m+d$ from $m$ to
$m-h$ multiplies its realization count and local factor by at most
$m^h2^{-hm+h(h+1)/2}$, while the single global falling factorial changes by
at most a factor $n^h$. Since $2h<m$, one has
$hm-h(h+1)/2\ge h\lfloor(3m-1)/4\rfloor$, so each missing stub costs at most
a factor $nm2^{-\lfloor(3m-1)/4\rfloor}\le\rho_\Sigma$, and
$\rho_\Sigma=O(n^{-1/4}\log n)$ by \eqref{eq:rho-sum-small}. Summing the at
most $\alpha+1$ deficits of each cell independently gives the proposition.

Define
$K:=\sum_{i=0}^{4}\kappa_i=k_{\mathrm{co}}$, the number of row-class slots.
Projection onto the row endpoint is injective on a block matching, and therefore
\begin{equation}
  |M|\le K.
  \label{eq:support-cardinality}
\end{equation}

\subsection{Regrouping by endpoint table}

After summing the deficits, we must sum the full-containment reference
weights. These depend on how many matched class pairs have each pair of
endpoint types. Record those counts in an endpoint table and group the
matching supports by that table. We can then compare its row and column
margins with the common-class weights from Section~7.

For a block matching $M$, let
\[
 \ell_{ij}(M):=\#\{e\in M:\text{$e$ joins row type $i$ to column type $j$}\},
 \qquad
 L(M):=(\ell_{ij}(M))_{0\le i,j\le4}.
\]
With this type indexing, for
$L=(\ell_{ij})\in\mathbb N_0^{5\times5}$ write
\[
  r=(r_i)_{i=0}^4,\qquad r_i=\sum_j\ell_{ij},
  \qquad
  c=(c_j)_{j=0}^4,\qquad c_j=\sum_i\ell_{ij}.
\]
Call $L$ \emph{feasible} if $r_i\le\kappa_i$ and $c_j\le\kappa_j$ for every
$i,j$. Such a table is realized by choosing the indicated row and column
block slots and pairing them cell by cell; conversely, every block matching
has a feasible table. For a feasible endpoint table $L$, define
\[
  W(L)
  :=
  \sum_{M:\,L(M)=L}w_{\mathrm{full}}(M).
\]

Partitioning the finite set of block matchings by $L(M)$ gives
\begin{equation}
  \sum_M w_{\mathrm{full}}(M)
  =
  \sum_{L\ \mathrm{feasible}}W(L).
  \label{eq:reference-grouping}
\end{equation}

Whenever $\rho_{\Sigma}\le1$, combining
Proposition~\ref{prop:fixed-support-all-deficit},
\eqref{eq:support-cardinality}, and
\eqref{eq:reference-grouping} yields, after enlarging the nonnegative sum
from canonical large-cell matchings to all block matchings,
\begin{equation}
  \sum_{(M,j)\ \mathrm{feasible}}w_{\mathrm{hi}}
  \le
  \left(\sum_{L\ \mathrm{feasible}}W(L)\right)
  \left(1+(\alpha+1)\rho_{\Sigma}\right)^K.
  \label{eq:coarse-realized-table-reduction}
\end{equation}

The finite argument through
\eqref{eq:coarse-realized-table-reduction} applies to any fixed set of endpoint
types for which the selected cells form a matching.  The five consecutive sizes
enter through the numerical deficit bound and the endpoint-table comparison
below.

\subsection{Comparison with the common-class weights}

When the endpoint sizes agree, full
containment means that a whole class is common to both partitions, the
case estimated in Section~7. For a general endpoint table, we compare its
weight with the common-class weights of the row and column margins,
including the corresponding multinomial factors. The reference value is
their geometric mean; unequal endpoint sizes introduce a small correction.

For a full-containment table $L=(\ell_{ij})$, set
\[
\begin{aligned}
  r&=(r_i)_{i=0}^4,& r_i&=\sum_j\ell_{ij},&
  c&=(c_j)_{j=0}^4,& c_j&=\sum_i\ell_{ij},\\
  J(L)&=\sum_{i,j}m_{ij}\ell_{ij}.
\end{aligned}
\]
The number of block matchings with table $L$ is
\[
 \frac{\prod_i(\kappa_i)_{r_i}\prod_j(\kappa_j)_{c_j}}
      {\prod_{i,j}\ell_{ij}!}.
\]
Indeed, row and column slot allocations initially produce two copies of the
cell-factorial denominator, while the bijections between the selected slots
inside each type cell contribute one copy back.  Thus the exact reference
weight is
\begin{equation}
  W(L)=
  \frac{\prod_i(\kappa_i)_{r_i}\prod_j(\kappa_j)_{c_j}}
       {\prod_{i,j}\ell_{ij}!}
  \frac1{(n)_{J(L)}}
  \prod_{i,j}
  \left[
    \frac{(u_i)_{m_{ij}}(u_j)_{m_{ij}}}{m_{ij}!}g(m_{ij})
  \right]^{\ell_{ij}}.
  \label{eq:endpoint-reference-weight}
\end{equation}
Let
\[
  \mathcal V:=
  \{v=(v_i)_{i=0}^4\in\mathbb N_0^5:0\le v_i\le\kappa_i
    \text{ for every }i\}.
\]
For $v\in\mathcal V$, define
\begin{equation}
  D(v)=
  \frac{\prod_i(\kappa_i)_{v_i}^2}{\prod_i v_i!}
  \frac{\prod_i[u_i!g(u_i)]^{v_i}}{(n)_{m(v)}},
  \qquad
  m(v)=\sum_i u_iv_i.
  \label{eq:one-sided-diagonal-weight}
\end{equation}
This is the common-subprofile weight
\eqref{eq:marked-common-subprofile-weight} with $\ell_{i+2}=v_i$: since
$\kappa_i=k_{i+2}$ and $u_i=\alpha-(i+2)$,
\eqref{eq:one-sided-diagonal-weight} is the first expression in
\eqref{eq:marked-common-subprofile-weight}.
Below, $D(r)$ and $D(c)$ evaluate
\eqref{eq:one-sided-diagonal-weight} on the margins in
$0\le v_i\le\kappa_i$.

\begin{samepage}
For the endpoint pair $(i,j)$ put $d_{ij}:=|u_i-u_j|$, so its endpoint
sizes are $m_{ij}$ and $m_{ij}+d_{ij}$ as in
Lemma~\ref{lem:exact-local-ratio}, and define
\[
  Q_{ij}:=
  (n+1)^{d_{ij}/2}
  \frac{\sqrt{(m_{ij}+d_{ij})_{d_{ij}}}}{d_{ij}!}
  2^{-\{d_{ij}m_{ij}+\binom{d_{ij}}2\}/2}.
\]
Thus $Q_{ii}=1$. Since $m_{ij}+d_{ij}\le U$, $m_{ij}\ge U-4$ and
$d_{ij}\le4$,
\begin{equation}
  Q_{ij}\le\frac{(\tau_n^{\mathrm{end}})^{d_{ij}}}{d_{ij}!},
  \qquad
  \tau_n^{\mathrm{end}}=
  4\sqrt{\frac{(n+1)U}{2^U}}
  =O\!\left(\frac{(\log n)^{3/2}}{\sqrt n}\right).
  \label{eq:endpoint-eta}
\end{equation}
\end{samepage}

Let
\[
  A_L=\frac{\prod_i r_i!}{\prod_{i,j}\ell_{ij}!},
  \qquad
  C_L=\frac{\prod_j c_j!}{\prod_{i,j}\ell_{ij}!},
  \qquad
  Q^L=\prod_{i,j}Q_{ij}^{\ell_{ij}}.
\]

\begin{lemma}[Endpoint-table product comparison]
\label{lem:endpoint-table-product-comparison}
Assume $\alpha>8$ and that $(\kappa_i,u_i)_{i=0}^4$ is the fixed feasible
profile, so $\sum_{i=0}^4\kappa_i u_i=n$. Then every feasible endpoint table $L$
satisfies
\begin{equation}
  W(L)^2
  \le
  \bigl(D(r)A_LQ^L\bigr)\bigl(D(c)C_LQ^L\bigr).
  \label{eq:endpoint-table-product-comparison}
\end{equation}
\end{lemma}

Section~\ref{app:endpoint-comparison-proof} proves the lemma. After division
by $(D(r)A_L)(D(c)C_L)$, the profile multiplicities and table factorials
cancel, each cell of endpoint type $(i,j)$ leaves the factor
$Q_{ij}^2/(n+1)^{d_{ij}}$, and the remaining ratio of falling factorials is
at most $(n+1)^{\sum_{i,j}d_{ij}\ell_{ij}}$.

\begin{proposition}[Endpoint-table sum]
\label{prop:endpoint-table-sum}
There is an absolute constant $C_{\mathrm{end}}>0$ such that, for all
sufficiently large $n$, uniformly in the phase,
\begin{equation}
  \sum_{L\ \mathrm{feasible}}W(L)
  \le
  (1+4\tau_n^{\mathrm{end}})^K
  \sum_{v\in\mathcal V}D(v)
  \le
  \exp\!\left\{C_{\mathrm{end}}\sqrt{n\log n}\right\}.
  \label{eq:endpoint-table-sum}
\end{equation}
All sums in this proposition are finite.
\end{proposition}

\begin{proof}
For all sufficiently large $n$, one has $\tau_n^{\mathrm{end}}\le1$ by
\eqref{eq:endpoint-eta}; hence $Q_{ij}\le\tau_n^{\mathrm{end}}$ whenever
$i\ne j$.
By the arithmetic--geometric mean inequality and
\eqref{eq:endpoint-table-product-comparison},
\[
  2W(L)
  \le
  D(r)A_LQ^L+D(c)C_LQ^L.
\]
Sum the row term first, keeping its margin $r$ fixed. We can drop the
column constraints and sum over all nonnegative integer tables with
row margin $r$, since $A_LQ^L\ge0$. The multinomial theorem then gives
\[
  \sum_{\substack{L\ \mathrm{feasible}\\
                  \operatorname{row}(L)=r}}A_LQ^L
  \le
  \sum_{\substack{L\in\mathbb N_0^{5\times5}\\
                  \operatorname{row}(L)=r}}A_LQ^L
  =
  \prod_i\left(\sum_jQ_{ij}\right)^{r_i}
  \le
  (1+4\tau_n^{\mathrm{end}})^K.
\]
Multiply by $D(r)$ and sum over the row margin $r$. The result is at most
$(1+4\tau_n^{\mathrm{end}})^K\sum_{v\in\mathcal V}D(v)$, with $v$ a dummy
subprofile variable. Exchanging rows and columns gives the same bound for
the second term. Thus the endpoint-table choices cost only the displayed
correction factor. The remaining sum over margins is the common-subprofile
sum of Section~7:
$\sum_{v\in\mathcal V}D(v)=1+o(1)$ uniformly in the phase, with all
common-class overlaps included. Finally, $K=\Theta(n/\log n)$ and
\eqref{eq:endpoint-eta} imply
$\tau_n^{\mathrm{end}}K=O(\sqrt{n\log n})$.
\end{proof}

\subsection{The large-cell bound across the phase}

By (2.1) and $\alpha_0-1<\alpha\le\alpha_0$, the integer center for the
independent-set size satisfies, uniformly in the phase,
\[
  \alpha
  =
  \left(\frac{2}{\log 2}+o(1)\right)\log n.
\]
In particular, for all sufficiently large $n$,
\begin{equation}
  \frac52\log n\le\alpha.
  \label{eq:coarse-phase-corridor}
\end{equation}
Every endpoint minimum satisfies $m_{ij}\ge\alpha-6$, and, since
$4\lfloor x/4\rfloor\ge x-3$ for integers $x$ and $3m_{ij}-1\ge3\alpha-19$,
\[
  3\alpha-22
  \le
  4\left\lfloor\frac{3m_{ij}-1}{4}\right\rfloor.
\]
Using $\log 2>2/3$, $3\alpha-22>0$ and \eqref{eq:coarse-phase-corridor}, we
obtain
\begin{equation}
  \frac54\log n-\frac{11}{3}
  \le
  (\log 2)
  \left\lfloor\frac{3m_{ij}-1}{4}\right\rfloor.
  \label{eq:coarse-log-budget}
\end{equation}
Hence
\[
  2^{\lfloor(3m_{ij}-1)/4\rfloor}
  \ge
  e^{-11/3}n^{5/4}.
\]
Since every $m_{ij}=O(\log n)$,
\begin{equation}
  \rho_{\Sigma}
  =O\!\left(\frac{\log n}{n^{1/4}}\right)
  \longrightarrow0.
  \label{eq:rho-sum-small}
\end{equation}
Thus $\rho_{\Sigma}\le1$ for all sufficiently large $n$.

Let $\Sigma_n^{\mathrm{hi}}$ denote the sum of exposed weights:
\[
  \Sigma_n^{\mathrm{hi}}
  :=
  \sum_{(M,j)\ \mathrm{feasible}}w_{\mathrm{hi}}(M,j),
\]
where the sum is over the feasible pairs $(M,j)$ defined above.

\begin{proposition}[Bound for the exposed weights]
\label{prop:bare-high-skeleton}
There is a phase-independent deterministic sequence
$\varepsilon_n^{\mathrm{hi}}\ge0$ with
$\varepsilon_n^{\mathrm{hi}}\to0$ such that, for all sufficiently large $n$,
\begin{equation}
  \Sigma_n^{\mathrm{hi}}
  \le
  \exp\!\left\{
    \varepsilon_n^{\mathrm{hi}}\frac{n}{(\log n)^4}
  \right\}.
  \label{eq:bare-skeleton-final}
\end{equation}
\end{proposition}

\begin{proof}
The deficit factor in \eqref{eq:coarse-realized-table-reduction} satisfies
\[
  \log\left(1+(\alpha+1)\rho_{\Sigma}\right)^K
  \le
  K(\alpha+1)\rho_{\Sigma}
  =O(n^{3/4}\log n).
\]
Proposition~\ref{prop:endpoint-table-sum} contributes
$O(\sqrt{n\log n})$ to the logarithm. Both quantities are
$o(n/(\log n)^4)$. Combining these estimates with
\eqref{eq:coarse-realized-table-reduction} proves the proposition.
\end{proof}

\section{Residual factors and the second-moment bound}
\label{sec:residual-attachments}

We keep the selected profile of Proposition~\ref{prop:selected-profile-contract}
and the endpoint indexing from Section~8. The exposed matching \(M\)
has pairs of class slots as its edges. The labeled stub pairs chosen
inside those cells are denoted by \(S\).

Fix a feasible pair $(M,j)$ and condition on a labeled realization of its
exposed stub pairs being present. The unused row and column stubs are then
matched uniformly. An allowed completion adds no pair in an exposed cell
and respects the residual cap elsewhere. We estimate the local and
cycle-space factors jointly, since they depend on this same residual
matching, and include the indicator that the completion is allowed.
A bound uniform over all feasible pairs $(M,j)$ can then be taken outside
the exposed-weight sum from Section~8.

Throughout this section, write $Z:=Z_{\mathbf k}^{\mathrm{sgn}}$ for the
selected marked cocoloring count.

\subsection{Exact decomposition after exposing the large cells}

Fix a feasible pair $(M,j)$ and recall $J=\sum_{e\in M}j_e$. Put
\[
  m_0:=n-J.
\]
Let $(d_a)_a$ and $(d'_b)_b$ be the residual row and column degrees. Their two
sums equal $m_0$, and every degree is at most the phase cap $U$.
Explicitly,
\[
 d_a=s_a-\sum_{b:(a,b)\in M}j_{ab},\qquad
 d'_b=t_b-\sum_{a:(a,b)\in M}j_{ab}.
\]

The unexposed pairs form a uniform bipartite configuration matching with these
residual degrees; when $m_0=0$, this means the unique empty matching. Write
$\mathbb P_{\mathrm{res}}$ and $\mathbb E_{\mathrm{res}}$ for its law and
expectation. Let
$r'_{ab}$ be its cell counts and let $H_{\mathrm{res}}$ be the simple support
graph of cells with $r'_{ab}\ge2$. The event $\mathcal E(M,j)$ specifies which
residual matchings complete the chosen large cells. Outside the exposed
matching it imposes the cap $r'_{ab}\le\lfloor U/2\rfloor$; in a cell of
$M$, no further pair is allowed. The exposed data already include every pair
in each large cell, so an additional pair would change its multiplicity.
Thus
\[
 \mathcal E(M,j)=
 \{r'_{ab}=0\ ((a,b)\in M),\quad
   r'_{ab}\le R_0\ ((a,b)\notin M)\}.
\]
More explicitly, with $\Pi$ as in
Section~\ref{subsec:large-cell-realizations}, for every
$S\in\mathfrak F(M,j)$ and every event $\mathcal B$ of the residual
matching,
\begin{equation}
  \mathbb P\{S\subseteq\Pi,\ \Pi\setminus S\in\mathcal B\}
  =
  \frac{1}{(n)_J}\,\mathbb P_{\mathrm{res}}(\mathcal B).
  \label{eq:exposed-residual-factorization}
\end{equation}
Different labeled realizations leave the same residual degree lists. A
blockwise bijection between the unused row stubs, together with one between
the unused column stubs, identifies their residual matching spaces and
preserves every cell count. The residual expectation therefore depends on
$M$ and the degree lists alone. Define
\begin{equation}
  \mathcal A(M,j)
  :=
  \mathbb E_{\mathrm{res}}\!\left[
    \left(\prod_{a,b}g(r'_{ab})\right)
    2^{\beta(M\cup H_{\mathrm{res}})}
    \mathbf 1_{\mathcal E(M,j)}
  \right].
  \label{eq:residual-attachment}
\end{equation}
Only the residual configuration matching is random in this expectation.

\begin{proposition}[Canonical exact overlap decomposition]
\label{prop:canonical-exact-overlap-decomposition}
Let $n$ be finite and let the fixed selected profile be feasible, with phase cap
$U\ge2$. Then
\begin{equation}
  \frac{\mathbb E Z^2}{(\mathbb E Z)^2}
  =
  \sum_{(M,j)}w_{\mathrm{hi}}(M,j)\,\mathcal A(M,j),
  \label{eq:exact-attachment-decomposition}
\end{equation}
where the sum ranges over the finite family of feasible pairs $(M,j)$.
\end{proposition}

Section~\ref{app:decomposition-proof} proves the proposition by a bijection.
A full matching $\Pi$ determines its large cells
$M(\Pi)=\{(a,b):r_{ab}(\Pi)>R_0\}$, which form a matching because
$2(R_0+1)>U$; their multiplicities $j(\Pi)$; the set
$S(\Pi)\in\mathfrak F(M,j)$ of its stub pairs in those cells; and the
residual matching $\Pi\setminus S(\Pi)$, which satisfies $\mathcal E(M,j)$.
Conversely, every feasible $(M,j)$, every $S\in\mathfrak F(M,j)$ and every
residual matching in $\mathcal E(M,j)$ glue to exactly one full matching.
For fixed $(M,j,S)$, \eqref{eq:exposed-residual-factorization} makes the
total contribution $(n)_J^{-1}\prod_{e\in M}g(j_e)\,\mathcal A(M,j)$, and
summing over $S$ gives $w_{\mathrm{hi}}(M,j)\mathcal A(M,j)$.

The case $m_0=0$ is complete without introducing cell activities. The residual
matching is empty, every residual local factor is one, and
$H_{\mathrm{res}}=\varnothing$. Since $M$ is a matching,
$\beta(M)=0$, so
\begin{equation}
  \mathcal A(M,j)=1
  \qquad(m_0=0).
  \label{eq:zero-residual-attachment}
\end{equation}
In all subsequent subsections we assume $m_0>0$. In view of
Proposition~\ref{prop:bare-high-skeleton}, it remains to bound
$\mathcal A(M,j)$ uniformly in the positive-residual case.

\subsection{Threshold expansion and cell activities}

We expand the integrand of \eqref{eq:residual-attachment} into products of
threshold indicators, bound each product by the joint estimate of
Lemma~\ref{lemma-6.2-joint-prescribed-cell-bound}, and collect the per-cell
contributions into two activities $\lambda_{ab}\le q_{ab}$.

For every row--column pair $(a,b)$, put
\begin{equation}
  \theta_{ab}:=
  \frac{\mathrm e\,d_ad'_b}{m_0}.
  \label{eq:theta}
\end{equation}
This intensity is defined for cells of $M$ as well as for cells outside
$M$. Both activities are set to zero on $M$. Let
\[
  \Delta_x:=g(x)-g(x-1),
\]
so $\Delta_x\ge0$: indeed $\Delta_3=3$, while
$g(x)/g(x-1)=2^{x-1}>1$ for $x\ge4$.
For $(a,b)\notin M$, define
\begin{equation}
  \lambda_{ab}
  :=
  \sum_{x=3}^{R_0}\Delta_x\frac{\theta_{ab}^x}{x!},
  \qquad
  q_{ab}
  :=
  \frac{\theta_{ab}^2}{2}+\lambda_{ab}.
  \label{eq:lambda-q}
\end{equation}
On $M$, set $\lambda_{ab}=q_{ab}=0$.
The first activity accounts for increments of the local factor above
one. The second also includes the threshold-two requirement for an edge
of the residual support.

For a finite array of nonnegative integer demands $x=(x_{ab})$ supported
outside $M$, apply Lemma~\ref{lemma-6.2-joint-prescribed-cell-bound} to the
residual degree lists and
$\operatorname{supp}(x):=\{(a,b):x_{ab}>0\}$.  It gives, in all cases,
\begin{equation}
  \mathbb P_{\mathrm{res}}(r'_{ab}\ge x_{ab}\text{ for all }a,b)
  \le
  \prod_{(a,b)\in\operatorname{supp}(x)}
    \frac{\theta_{ab}^{x_{ab}}}{x_{ab}!}.
  \label{eq:prescribed-cell-residual}
\end{equation}
Because this estimate is joint, it applies simultaneously to cells sharing a
row or a column; no independence between cells is asserted.

For $0\le r\le R_0$,
\[
  g(r)=1+\sum_{x=3}^{R_0}\Delta_x\mathbf 1_{\{r\ge x\}}.
\]
Using the row- and column-slot sets defined in Section~6, put
$E_0:=(\mathcal I_{\mathrm{row}}\times\mathcal I_{\mathrm{col}})\setminus M$, and let
$\mathfrak E(M)$ be the family of even subsets of $M\cup E_0$. For an edge
$e=(a,b)$, write
$r'_e:=r'_{ab}$, $\theta_e:=\theta_{ab}$,
$\lambda_e:=\lambda_{ab}$, and $q_e:=q_{ab}$. For
$F\in\mathfrak E(M)$, define
\[
\begin{split}
  \Phi_F(r')
  :=
  &\prod_{e\in F\setminus M}
  \left(
    \mathbf 1_{\{r'_e\ge2\}}
    +
    \sum_{x=3}^{R_0}\Delta_x\mathbf 1_{\{r'_e\ge x\}}
  \right)\\
  &\times
  \prod_{e\in E_0\setminus F}
  \left(
    1+
    \sum_{x=3}^{R_0}\Delta_x\mathbf 1_{\{r'_e\ge x\}}
  \right).
\end{split}
\]
On the event $\mathcal E(M,j)$, expansion of the cycle-space
cardinality and of all local factors gives the exact identity
\[
  \left(\prod_{a,b}g(r'_{ab})\right)
  2^{\beta(M\cup H_{\mathrm{res}})}
  =
  \sum_{F\in\mathfrak E(M)}\Phi_F(r').
\]
To verify it, note that evenness of an edge set does not depend on the
ambient graph and that $E(H_{\mathrm{res}})\subseteq E_0$ on
$\mathcal E(M,j)$. By (6.7), $2^{\beta(M\cup H_{\mathrm{res}})}$ counts the
$F\in\mathfrak E(M)$ with $r'_e\ge2$ for all $e\in F\setminus M$:
\[
 2^{\beta(M\cup H_{\mathrm{res}})}
 =\sum_{F\in\mathfrak E(M)}\prod_{e\in F\setminus M}\mathbf 1_{\{r'_e\ge2\}}.
\]
Multiply by $\prod_{a,b}g(r'_{ab})$, whose factors on $M$ equal $g(0)=1$;
expand each factor on $E_0$ by the formula for $g(r)$ (valid as
$r'_e\le R_0$), and use
$\mathbf 1_{\{r\ge2\}}g(r)=\mathbf 1_{\{r\ge2\}}
+\sum_{x=3}^{R_0}\Delta_x\mathbf 1_{\{r\ge x\}}$.

This identity uses both parts of $\mathcal E(M,j)$: off-$M$ cells are
capped at $R_0$, and $r'$ vanishes on $M$. We next multiply by
$\mathbf 1_{\mathcal E(M,j)}$ and bound the expectation term by term.
On $\mathcal E(M,j)^c$ we use only this expansion,
without identifying it with the uncapped local factor.

\begin{lemma}[Fixed even-set expansion]
\label{lem:fixed-even-set-expansion}
For every $F\in\mathfrak E(M)$,
\[
  \mathbb E_{\mathrm{res}}\!\left[
    \Phi_F(r')\mathbf 1_{\mathcal E(M,j)}
  \right]
  \le
  \prod_{e\in F\setminus M}q_e
  \prod_{e\in E_0\setminus F}(1+\lambda_e).
\]
\end{lemma}

\begin{proof}
Expand $\Phi_F$ by choosing one summand from each factor. For a cell in
$F\setminus M$, choose either the threshold-two term or one higher
increment, never both. For a cell outside $F\cup M$, choose either the
constant term or one higher increment. Apply
\eqref{eq:prescribed-cell-residual} once to all the demands in each
resulting product. The threshold-two term contributes $\theta_e^2/2$;
the higher increments sum to $\lambda_e$; and the constant choice
contributes one. Every product is nonnegative and
$\mathbf 1_{\mathcal E(M,j)}\le1$, so dropping the indicator of
$\mathcal E(M,j)$ can only increase its expectation.
\end{proof}

Summing over $F$ and then inserting the missing factors
$1+\lambda_e\ge1$ gives
\begin{equation}
  \mathcal A(M,j)
  \le
  \left(\prod_{e\in E_0}(1+\lambda_e)\right)
  \sum_{F\in\mathfrak E(M)}
    \prod_{e\in F\setminus M}q_e.
  \label{eq:lambda-even-factorization}
\end{equation}

\subsection{Restricting even edge sets outside the exposed matching}

To bound the sum over even edge sets, we first discard their edges in the
exposed matching, which have unit weight. The retained edges determine
the entire even set. We can therefore bound the sum by allowing every
subset of the edges outside the matching and expanding the resulting
product.

\begin{lemma}[Even-set restriction bound]
\label{lem:restriction-product}
Let \(H\) be a finite graph, let \(M\subseteq E(H)\) be a matching, and
let \(q_e\ge0\) for \(e\in E(H)\setminus M\). Then
\[
 \sum_{\substack{F\subseteq E(H)\\F\text{ even}}}
   \prod_{e\in F\setminus M}q_e
 \le\prod_{e\in E(H)\setminus M}(1+q_e).
\]
\end{lemma}

\begin{proof}
Two even edge sets with the same restriction outside \(M\) have an even
symmetric difference contained in \(M\). A nonempty subset of a matching
has a vertex of degree one, so the difference must be empty.
Restriction outside \(M\) is therefore injective. Enlarge the resulting
sum to all subsets of \(E(H)\setminus M\) and expand the product.
\end{proof}

Apply the lemma to the complete bipartite graph on the row and column slots, with edge set \(M\cup E_0\). It gives
\begin{equation}
  \sum_{F\in\mathfrak E(M)}\prod_{e\in F\setminus M}q_e
  \le\prod_{e\in E_0}(1+q_e).
  \label{eq:even-family-product}
\end{equation}

Since $0\le\lambda_e\le q_e$, equations
\eqref{eq:lambda-even-factorization} and
\eqref{eq:even-family-product}, together with $1+x\le e^x$, imply
\begin{equation}
  \mathcal A(M,j)
  \le
  \exp\!\left(2\sum_{e\in E_0}q_e\right).
  \label{eq:q-only-attachment}
\end{equation}
Thus both the local factors and the cycle-space cardinality are controlled by
one total residual activity.

\subsection{Large residual mass: a quadratic activity bound}

We next bound the sum of the residual activities. When many stubs remain
unmatched, each cell intensity is small relative to the class-size cap.
In this range, the contribution from higher multiplicities is bounded
by a constant times the square of the intensity.

Assume
\begin{equation}
  2^U\le m_0^3.
  \label{eq:intrinsic-residual-regime}
\end{equation}
Then $m_0\ge2^{U/3}$ and, by the degree cap,
\[
  \theta_{ab}
  \le
  \mathrm e U^2 2^{-U/3}.
\]

\begin{lemma}[Quadratic activity bound]
\label{lem:q-quadratic}
There is an absolute constant $C_0$ such that, under
\eqref{eq:intrinsic-residual-regime},
\begin{equation}
  q_{ab}\le C_0\theta_{ab}^2
  \label{eq:q-quadratic}
\end{equation}
for every residual cell outside $M$.
\end{lemma}

The quadratic activity estimate is proved in Section~\ref{app:quadratic-activities}; summing it over the residual degree lists gives the bound below.

Because $\theta_{ab}$ was defined for every pair, the cell intensities satisfy
the exact identity
\[
  \sum_{a,b}\theta_{ab}^2
  =
  \frac{\mathrm e^2}{m_0^2}
  \left(\sum_a d_a^2\right)
  \left(\sum_b(d'_b)^2\right).
\]
Since every degree is at most $U$ and both degree sums equal $m_0$,
\[
  \sum_a d_a^2\le Um_0,
  \qquad
  \sum_b(d'_b)^2\le Um_0.
\]
On $M$ the activities are zero, while outside $M$ the preceding lemma applies.
Therefore
\begin{equation}
  \sum_{e\in E_0}q_e\le C_{\Sigma q}U^2
  \qquad\text{with }C_{\Sigma q}:=\mathrm e^2C_0.
  \label{eq:total-q}
\end{equation}
Together with \eqref{eq:q-only-attachment}, this proves
\begin{equation}
  \mathcal A(M,j)\le\exp(2C_{\Sigma q}U^2)
  \qquad\text{if }2^U\le m_0^3.
  \label{eq:intrinsic-attachment}
\end{equation}

\subsection{Small residual mass: a pointwise bound}

When few stubs remain unmatched, the quadratic activity estimate may not
apply. We instead bound the contribution of each residual matching
directly.
Suppose $2^U>m_0^3$. Then
\begin{equation}
  m_0<2^{U/3}\le2^{\lceil U/3\rceil}.
  \label{eq:small-residual-mass}
\end{equation}
The pointwise argument uses only the canonical cap and the matching
property, so it also applies outside this regime. Retaining the cap gives
a sharper estimate. On
$\mathcal E(M,j)$, write
\[
  H:=H_{\mathrm{res}},
  \qquad
  W':=\sum_{a,b}\binom{r'_{ab}}2.
\]
Every support edge is a cell of multiplicity at least two, so the local factor, including $g(2)=1$, is exactly
\[
  \prod_{a,b}g(r'_{ab})=2^{W'-|E(H)|}.
\]
Restriction of even edge sets outside the matching $M$ is injective by
the argument following Lemma~\ref{lem:restriction-product}.
No residual pair lies in an exposed cell, so $E(H)\cap M=\varnothing$; hence
\[
  2^{\beta(M\cup H)}\le2^{|E(H)|}.
\]
The edge-cardinality factors cancel. Since $U\ge2$, the cap
$R_0=\lfloor U/2\rfloor$ is at least one, and for $0\le r\le R_0$,
\[
  \binom r2\le\frac{R_0-1}{2}r.
\]
Summing over cells, using that their residual multiplicities sum to
$m_0$, gives
\[
  \left(\prod_{a,b}g(r'_{ab})\right)2^{\beta(M\cup H)}
  \le2^{W'}\le2^{(R_0-1)m_0/2}.
\]
The residual factor integrand vanishes outside $\mathcal E(M,j)$. Taking
expectations proves
\begin{equation}
  \mathcal A(M,j)\le2^{(R_0-1)m_0/2}
  \qquad\text{for every }m_0\ge0.
  \label{eq:complementary-attachment}
\end{equation}
We use this bound when the positive residual mass is small. It also covers
zero residual mass, empty residual support, and an empty canonical event.

The following proposition collects the bounds for zero, large and small
residual mass. The threshold depends on the class-size cap $U$, rather than
on the number of class-size types.

\begin{proposition}[Uniform residual factor]
\label{prop:uniform-residual-attachment}
For the fixed feasible selected profile with $U\ge2$, there is an absolute
constant $C_{\mathrm{att}}>0$ such that every feasible pair $(M,j)$
satisfies
\[
  \mathcal A(M,j)
  \le
  \begin{cases}
    1,&m_0=0,\\
    \exp(C_{\mathrm{att}}U^2),&m_0>0\text{ and }2^U\le m_0^3,\\
    2^{(R_0-1)m_0/2},&m_0>0\text{ and }2^U>m_0^3.
  \end{cases}
\]
\end{proposition}

\begin{proof}
The three cases are exhaustive.  Their bounds are
\eqref{eq:zero-residual-attachment},
\eqref{eq:intrinsic-attachment}, and
\eqref{eq:complementary-attachment}, with
$C_{\mathrm{att}}:=2C_{\Sigma q}$.
\end{proof}

\subsection{The global second-moment exponent}
\label{subsec:explicit-global-ledger}

Insert the bounds for exposed weights and residual factors into the
exact second-moment decomposition. Their logarithms add. We keep track
of the resulting exponent to verify that it is small enough for the
amplification step.

Combine the three deterministic partial-diagonal errors introduced in
Section~7 and estimated in Section~\ref{app:common-subprofile-details} by
setting
\begin{equation}
  \varepsilon_n^{\mathrm{pd}}
  :=
  e^{\varepsilon_n^{\mathrm{empty}}}-1
  +\varepsilon_n^{\mathrm{central}}
  +\varepsilon_n^{\mathrm{full}}.
  \label{eq:partial-diagonal-ledger}
\end{equation}
Then $\varepsilon_n^{\mathrm{pd}}\to0$ uniformly in the phase, and the disjoint
three-range assembly gives
\begin{equation}
  \sum_{v\in\mathcal V}D(v)\le1+\varepsilon_n^{\mathrm{pd}}.
  \label{eq:partial-diagonal-ledger-sum}
\end{equation}
All sums in this subsection are finite and range over the domains defined in
Sections~7--9.

With $\tau_n^{\mathrm{end}}$ from \eqref{eq:endpoint-eta}, put
\begin{equation}
\begin{split}
  \Gamma_n^{\mathrm{skel}}
  :={}&
  K\log\!\left(1+(\alpha+1)\rho_{\Sigma}\right)\\
  &+K\log\!\left(1+4\tau_n^{\mathrm{end}}\right)
  +\log\!\left(1+\varepsilon_n^{\mathrm{pd}}\right).
\end{split}
  \label{eq:explicit-skeleton-ledger}
\end{equation}
Inserting \eqref{eq:partial-diagonal-ledger-sum} into the first inequality
of \eqref{eq:endpoint-table-sum}, and the result into
\eqref{eq:coarse-realized-table-reduction}, gives, for all sufficiently
large $n$,
\begin{equation}
  \Sigma_n^{\mathrm{hi}}
  \le e^{\Gamma_n^{\mathrm{skel}}}.
  \label{eq:explicit-bare-skeleton-bound}
\end{equation}
As in the proofs of Propositions~\ref{prop:endpoint-table-sum}
and~\ref{prop:bare-high-skeleton}, the three terms of
$\Gamma_n^{\mathrm{skel}}$ are $O(n^{3/4}\log n)$, $O(\sqrt{n\log n})$ and
$o(1)$; so
$\varepsilon_n^{\mathrm{skel}}:=(\log n)^4\Gamma_n^{\mathrm{skel}}/n\to0$.

Let $C_{\mathrm{att}}$ be the constant from
Proposition~\ref{prop:uniform-residual-attachment}.
Define
\begin{equation}
  \Gamma_n^{\mathrm{att}}
  :=
  \max\!\left\{
    C_{\mathrm{att}}U^2,
    \frac{\log 2}{2}(R_0-1)2^{U/3}
  \right\}.
  \label{eq:attachment-ledger}
\end{equation}
Proposition~\ref{prop:uniform-residual-attachment} implies the uniform
bound
\begin{equation}
  \mathcal A(M,j)\le e^{\Gamma_n^{\mathrm{att}}}
  \label{eq:uniform-attachment-exponential}
\end{equation}
for every feasible pair $(M,j)$. For large residual mass we use
\eqref{eq:intrinsic-attachment}. For small positive residual mass,
$m_0<2^{U/3}$ and
\[
  2^{(R_0-1)m_0/2}
  =\exp\!\left(\frac{\log 2}{2}(R_0-1)m_0\right)
  \le\exp\!\left(\frac{\log 2}{2}(R_0-1)2^{U/3}\right).
\]
The zero-residual case contributes one.

By (2.1) and $\alpha_0-1<\alpha\le\alpha_0$, uniformly in the phase,
\[
  U=O(\log n),
  \qquad
  2^U=\Theta\!\left(\frac{n^2}{(\log n)^2}\right),
\]
so
\[
  \Gamma_n^{\mathrm{att}}
  =O\!\left(
    (\log n)^2+n^{2/3}(\log n)^{1/3}
  \right)
  =o\!\left(\frac{n}{(\log n)^4}\right).
\]
Thus the deterministic normalized residual factor error
\begin{equation}
  \varepsilon_n^{\mathrm{att}}
  :=
  \frac{(\log n)^4}{n}\Gamma_n^{\mathrm{att}}
  \longrightarrow0.
  \label{eq:normalized-attachment-error}
\end{equation}

\Needspace{8\baselineskip}
\begin{proposition}[Explicit normalized second-moment bound]
\label{prop:explicit-normalized-second-moment-ledger}
For the marked cocoloring count
$Z=Z_{\mathbf k}^{\mathrm{sgn}}$ from
Proposition~\ref{prop:selected-profile-contract}, define
\begin{equation}
  \Lambda_n
  :=
  \Gamma_n^{\mathrm{skel}}+\Gamma_n^{\mathrm{att}}
  =
  \left(
    \varepsilon_n^{\mathrm{skel}}
    +\varepsilon_n^{\mathrm{att}}
  \right)
  \frac{n}{(\log n)^4}.
  \label{eq:explicit-lambda-ledger}
\end{equation}
Then
\[
  \Lambda_n=O(n^{3/4}\log n)
  =o\!\left(\frac{n}{(\log n)^4}\right)
\]
and
for all sufficiently large $n$,
\begin{equation}
  1
  \le
  \frac{\mathbb E Z^2}{(\mathbb E Z)^2}
  \le e^{\Lambda_n}.
  \label{eq:explicit-second-moment-ledger}
\end{equation}
\end{proposition}

\begin{proof}
Insert \eqref{eq:uniform-attachment-exponential} into the exact decomposition
\eqref{eq:exact-attachment-decomposition}. Because all summands are
nonnegative, the uniform residual factor may be taken outside the finite
sum over $(M,j)$. Equation~\eqref{eq:explicit-bare-skeleton-bound} then gives
\[
  \frac{\mathbb E Z^2}{(\mathbb E Z)^2}
  \le
  e^{\Gamma_n^{\mathrm{att}}}
  \Sigma_n^{\mathrm{hi}}
  \le
  e^{\Gamma_n^{\mathrm{att}}+\Gamma_n^{\mathrm{skel}}}
  =e^{\Lambda_n}.
\]
The lower bound follows from $\operatorname{Var}(Z)\ge0$.
The component estimates above give the stated $O(n^{3/4}\log n)$ bound
for $\Lambda_n$, while the bound on $\varepsilon_n^{\mathrm{skel}}$ above and
\eqref{eq:normalized-attachment-error} give the small-$o$ bound.
\end{proof}

The contribution $\Gamma_n^{\mathrm{skel}}$ comes from common classes,
deficits and endpoint tables; $\Gamma_n^{\mathrm{att}}$ comes from residual
local factors and the cycle-space factor. Applying Paley--Zygmund with
their sum gives the probability bound used in the next section.
\section{Amplifying a cocoloring probability bound}\label{rare-event-amplification}

This section proves estimate (C) of the introduction: with high
probability, $\zeta(G_n)\le k_{\mathrm{co}}+a_n$ for a deterministic sequence
$a_n=o(n/(\log n)^3)$.

Proposition~\ref{prop:explicit-normalized-second-moment-ledger} and (1.4)
give the probability bound

\[
 \Prob{Z_{\mathbf k}^{\mathrm{sgn}}>0}\ge e^{-\Lambda_n}.       \tag{10.1}
\]

On the event $Z_{\mathbf k}^{\mathrm{sgn}}>0$, a realized marked partition
gives a cocoloring with $k_{\mathrm{co}}$ classes. Hence

\[
 \Prob{\zeta(G_n)\le k_{\mathrm{co}}}\ge e^{-\Lambda_n}.        \tag{10.2}
\]

The probability in (10.2) may be small. Following
\citet[Theorem~1]{heckel-2025-difference}, we use vertex-exposure
concentration, as in \citet[Theorem~1]{scott-2008-2017}, to find a large
induced subgraph with the desired cocoloring. We then color the vertices
outside it. Because this leftover set depends on the graph, we first
prove a coloring bound that holds simultaneously for every vertex set.
The second lemma uses that bound to control the mean number of additional
classes. Markov's inequality then gives an error of $o(n/(\log n)^3)$,
smaller than the root separation. Section~\ref{app:amplification-tail}
gives a stronger explicit tail bound.

\subsection*{Coloring every leftover vertex set}

\begin{lemma}[Simultaneous leftover coloring]
\label{lemma-10.1-simultaneous-leftover-coloring}

There are absolute constants \(C_0,c>0\) such that, for all sufficiently
large \(n\), with probability at least \(1-e^{-c\sqrt n}\), every
\(S\subseteq[n]\) satisfies

\[
 \chi(G_n[S])\le C_0\frac{|S|}{{\log n}}+n^{1/3}.               \tag{10.3}
\]

\end{lemma}

Section~\ref{app:leftover-coloring-proof} proves the lemma by the classical
greedy argument. By (1.5) and a union bound, outside an event of probability
at most $e^{-c\sqrt n}$ every set of $\lceil n^{1/4}\rceil$ vertices spans at
least a quarter of its pairs in the complement of $G_n$, and averaging
extends this to every larger set. In a set of at least $n^{1/3}$ vertices,
repeatedly choosing a vertex of maximum complement degree and keeping its
complement neighbors yields an independent set of $G_n$ whose size is at
least a constant multiple of $\log n$. Removing such sets one at a time, and
coloring the fewer than $n^{1/3}$ remaining vertices singly, gives (10.3).

\subsection*{Mean excess above the class-count bound}

We bound the mean number of vertices left outside a cocolorable set,
then color those vertices on the simultaneous density event.

\begin{lemma}[Mean bound for the cochromatic number]
\label{lem:mean-seed-amplification}
For all sufficiently large $n$, every deterministic integer $k\ge0$ and
every deterministic finite real $\Lambda\ge0$ satisfying
\[
 \Prob{\zeta(G_n)\le k}\ge e^{-\Lambda}
\]
obey the following bound. Put $u_0:=\lceil n^{1/4}\rceil$ and
\[
 \varepsilon_n^{\mathrm{dens}}
 :=\min\!\left\{1,\binom n{u_0}
                   \exp\!\left(-\frac{\binom{u_0}{2}}{16}\right)\right\}.
\]
With $x_+:=\max\{x,0\}$ and the same absolute $C_0$ as in
Lemma~\ref{lemma-10.1-simultaneous-leftover-coloring},
\[
 \Exp{(\zeta(G_n)-k)_+}
 \le\frac{C_0}{\log n}\sqrt{\frac{(n-1)\Lambda}{2}}
      +n^{1/3}+n\varepsilon_n^{\mathrm{dens}}.
\]
The eventual threshold is independent of $k$, $\Lambda$, and the phase.
\end{lemma}

\begin{proof}
Let $\mathcal D_n$ be the density event used in the proof of
Lemma~\ref{lemma-10.1-simultaneous-leftover-coloring}: every $u_0$-set has
edge density at least $1/4$ in the complement of $G_n$. By (1.5) and a
union bound,
\[
 \Prob{\mathcal D_n^c}\le\varepsilon_n^{\mathrm{dens}}.
\]
On $\mathcal D_n$, (10.3) holds for every vertex set.

To apply concentration, we measure the largest vertex set that can be
cocolored with at most the given number of classes. For $W\subseteq[n]$, let
$\zeta(W)=\zeta(G_n[W])$, with $\zeta(\varnothing)=0$, and define

\[
 S_{k}=\max\{|W|:W\subseteq[n],\ \zeta(W)\le k\}.  \tag{10.4}\label{eq:amplification-largest-set}
\]

Expose the graph in \(n-1\) independent vertex blocks, placing in block
\(v\) the edges from its vertex to earlier vertices. Two outcomes that differ
only in block \(v\) agree on all edges away from that vertex. Take a
maximizing set in the first outcome. If it contains the changed vertex,
delete it and restrict the cocoloring to the remaining set; otherwise keep
the set. Either way, it is feasible in the second outcome and has lost at
most one vertex. Interchanging the outcomes gives the reverse comparison.
Thus $S_{k}$ changes by at most one, and (1.3) applies
\citep[Theorem~3.1]{mcdiarmid-1998}.

\Needspace{6\baselineskip}
Since $S_{k}=n$ exactly when $\zeta(G_n)\le k$,
the assumed probability bound and the one-sided upper-tail form of (1.3) give

\[
 e^{-\Lambda}
 \le\Prob{S_{k}-\Exp{S_{k}}\ge n-\Exp{S_{k}}}
 \le\exp\!\left\{-\frac{2(n-\Exp{S_{k}})^2}{n-1}\right\}.
\]

Taking logarithms bounds the mean number of uncovered vertices:

\[
 n-\Exp{S_{k}}\le\sqrt{(n-1)\Lambda/2}.              \tag{10.5}\label{eq:amplification-uncovered-mean}
\]

Choose a maximizing set \(W\) and put
\(V_{\mathrm{left}}=[n]\setminus W\). Combining a cocoloring of \(G_n[W]\)
using at most \(k\) parts with an ordinary coloring of
\(G_n[V_{\mathrm{left}}]\) gives

\[
 \zeta(G_n)\le k+\chi(G_n[V_{\mathrm{left}}]).         \tag{10.6}\label{eq:amplification-leftover-coloring}
\]

On $\mathcal D_n$, (10.3) applies to the graph-dependent set
$V_{\mathrm{left}}$, so by \eqref{eq:amplification-leftover-coloring}
$(\zeta(G_n)-k)_+\le C_0(n-S_k)/\log n+n^{1/3}$; off $\mathcal D_n$,
$(\zeta(G_n)-k)_+\le n$. Since $n-S_k\ge0$,
\[
 \Exp{(\zeta(G_n)-k)_+}\le\frac{C_0}{\log n}\bigl(n-\Exp{S_k}\bigr)+n^{1/3}
   +n\,\Prob{\mathcal D_n^c},
\]
and \eqref{eq:amplification-uncovered-mean} with
$\Prob{\mathcal D_n^c}\le\varepsilon_n^{\mathrm{dens}}$ proves the assertion.
Finally,
\[
 \log\binom n{u_0}=O(n^{1/4}\log n),\qquad
 \binom{u_0}{2}=(1+o(1))\sqrt n/2,
\]
so $\varepsilon_n^{\mathrm{dens}}\le e^{-c\sqrt n}$ eventually for one
absolute $c>0$.
\end{proof}

For the probability bound (10.2), the explicit second-moment bound satisfies
$\Lambda_n=o(n/(\log n)^4)$. Consequently,
\[
 \Exp{(\zeta(G_n)-k_{\mathrm{co}})_+}
 =o\!\left(\frac n{(\log n)^3}\right).
\]
This also gives $\Exp{\zeta(G_n)}\le k_{\mathrm{co}}+o(n/(\log n)^3)$,
uniformly along the full integer sequence.

\Needspace{6\baselineskip}
\subsection*{Applying the probability bound}
The mean bound gives the high-probability conclusion by Markov's inequality.
Put \(s_n=n/(\log n)^3\) and let
\[
 b_n:=\frac{C_0}{\log n}\sqrt{\frac{(n-1)\Lambda_n}{2}}
       +n^{1/3}+n\varepsilon_n^{\mathrm{dens}},
 \qquad a_n:=\sqrt{b_ns_n}.
 \tag{10.7}\label{eq:amplification-markov-scales}
\]
Here \(b_n>0\) is a deterministic upper bound for
\(\mathbb E(\zeta(G_n)-k_{\mathrm{co}})_+\).
Since \(\Lambda_n=o(n/(\log n)^4)\) and
\(\varepsilon_n^{\mathrm{dens}}\le e^{-c\sqrt n}\),
\[
 b_n=o(s_n),\qquad a_n=o(s_n).
 \tag{10.8}\label{eq:amplification-markov-errors}
\]
Markov's inequality now gives
\[
 \mathbb P\bigl(\zeta(G_n)>k_{\mathrm{co}}+a_n\bigr)
 \le\frac{b_n}{a_n}=\sqrt{\frac{b_n}{s_n}}\longrightarrow0.
 \tag{10.9}\label{eq:amplification-markov-probability}
\]
All these sequences are deterministic, and the bounds hold uniformly in the
phase. The stronger tail estimate is
Lemma~\ref{lemma-10.2-amplification-from-a-seed} in
Section~\ref{app:amplification-tail}.

\subsection*{Uniformity across the phase}
\label{subsec:uniformity-phase}

The root estimates are uniform on $K_*$, and the errors above have
deterministic bounds independent of the phase. The residual estimates
cover all residual masses, and the leftover-coloring event holds for
every vertex set simultaneously. All eventual conditions, including
$\alpha>8$, $U\ge2$, $\rho_{\Sigma}\le1$ and
$\tau_n^{\mathrm{end}}\le1$, therefore hold after one deterministic
threshold. Section~\ref{app:amplification-tail} records the corresponding
uniformity for the explicit tail bound.

\section{The phase-resolved gap and a uniform constant}
\label{sec:final-assembly}

We now combine the three estimates (A)--(C) from the introduction.
The first-moment roots give a deterministic separation. With high
probability, the chromatic number lies above its lower location and the
cochromatic number lies below the amplified upper location. On the
intersection of these events, their separation gives the phase-resolved
gap. Continuity on the closed phase interval then gives a uniform
constant.

\begin{proof}[Proof of the main theorem]

Take $n$ beyond the common deterministic threshold at the end of
Section~\ref{rare-event-amplification}.
Proposition~\ref{prop:selected-profile-contract}(iii) gives, uniformly in
the phase,
\begin{equation}
  k_\chi^- - k_{\mathrm{co}}
  =\left[\frac{(\log 2)^2}{4}A_5(\delta_n)+o(1)\right]\frac{n}{(\log n)^3}.
  \label{eq:integer-location-gap}
\end{equation}
It combines (5.11) with $|k_\chi^- -r_+(n)|<\log n+2$ from (4.7) and
$16\le k_{\mathrm{co}}-r_5^{\mathrm{co}}<17$; the rounding (5.16) changes
only the class-size counts, not $k_{\mathrm{co}}$. By (4.6),
$\mathbb P(\chi(G_n)>k_\chi^-)\to1$.

The amplification result \eqref{eq:amplification-markov-probability} gives a deterministic sequence
$a_n=o(n/(\log n)^3)$ such that
\[
  \mathbb P\bigl(\zeta(G_n)\le k_{\mathrm{co}}+a_n\bigr)
  \longrightarrow1.
\]
Figure~\ref{fig:final-root-geometry} in the introduction shows the
deterministic locations used in this comparison.

To make the final error explicit, put
\[
  s_n:=\frac{n}{(\log n)^3},\qquad
  \theta_n:=\frac{k_\chi^- - k_{\mathrm{co}}}{s_n}
  -\frac{(\log 2)^2}{4}A_5(\delta_n),\qquad
  \eta_n:=\frac{a_n}{s_n}.
\]
By \eqref{eq:integer-location-gap} and \eqref{eq:amplification-markov-probability},
$\theta_n\to0$ and $\eta_n\to0$.  Hence the nonnegative deterministic error
\[
  \varepsilon_n^{\mathrm{gap}}:=|\theta_n|+\eta_n
\]
tends to zero. On the event
\[
  \{\chi(G_n)>k_\chi^-\}
  \cap
  \{\zeta(G_n)\le k_{\mathrm{co}}+a_n\},
\]
\[
  \chi(G_n)-\zeta(G_n)
  > k_\chi^- - k_{\mathrm{co}}-a_n
  \ge
  \left[
    \frac{(\log 2)^2}{4}A_5(\delta_n)
    -\varepsilon_n^{\mathrm{gap}}
  \right]s_n.
\]
By (4.6) and \eqref{eq:amplification-markov-probability}, a union bound shows that this intersection has
probability $1-o(1)$. Consequently,
\begin{equation}
  \mathbb P\!\left(
    \chi(G_n)-\zeta(G_n)
    \ge
    \left[
      \frac{(\log 2)^2}{4}A_5(\delta_n)
      -\varepsilon_n^{\mathrm{gap}}
    \right]
    \frac{n}{(\log n)^3}
  \right)
  \longrightarrow1.
  \label{eq:phase-resolved-final}
\end{equation}

On this tail, set $\varepsilon_n:=\varepsilon_n^{\mathrm{gap}}$ and define the
finitely many preceding terms arbitrarily as nonnegative reals. Thus
$(\varepsilon_n)_{n\ge2}$ is deterministic, nonnegative, and tends to zero.

\displayheading{The uniform coefficient}

Put $q:=\log 2$.
Lemma~\ref{lem:phase-gain-concavity} shows that $A_5$ is continuous and
nonconstant on $[0,1]$ and locates its smallest value at an endpoint, while
Lemma~\ref{lem:entropy-gap} shows that $A_5>\gamma_5>0$. Hence
\[
  m_5:=\min\{A_5(0),A_5(1)\}
      =\min_{0\le\delta\le1}A_5(\delta)>0,
  \qquad
  c_5=\frac{q^2}{4}m_5>0.
\]
Fix $c<c_5$. Since $\varepsilon_n^{\mathrm{gap}}\to0$, eventually
$\varepsilon_n^{\mathrm{gap}}<c_5-c$. Therefore the coefficient in
\eqref{eq:phase-resolved-final} is at least
$q^2A_5(\delta_n)/4-\varepsilon_n^{\mathrm{gap}}\ge
c_5-\varepsilon_n^{\mathrm{gap}}>c$ for all sufficiently large $n$.
It follows that
\[
  \mathbb P\!\left(
    \chi(G_n)-\zeta(G_n)
    \ge
    c\,\frac{n}{(\log n)^3}
  \right)
  \longrightarrow1.
\]
This proves the main theorem.

To bound the constant from below, Lemma~\ref{lem:phase-gain-sharp} gives
$m_5>\log(169/100)$. Consequently
\[
  c_5>\frac{(\log 2)^2}{4}\log\!\left(\frac{169}{100}\right)
  =0.0630268508\ldots,
\]
and the uniform bound holds in particular with
$c=\frac14(\log 2)^2\log(169/100)$. Since $n/(\log n)^3\to\infty$, the
chromatic--cochromatic difference tends to infinity with high probability
along the full sequence of integers.  We have proved the lower-bound direction
of the conjectured $n/(\log n)^3$ scale; a matching upper bound and the optimal
constant remain open.
\end{proof}

\subsection*{The remainder rate and summable failure probabilities}

\begin{remark}[Rate of the phase-resolved remainder]
\label{rem:phase-remainder-rate}
The deterministic error in the main theorem may be chosen so that
\[
 \varepsilon_n=O\!\left(\frac{\log\log n}{\log n}\right),
\]
with a phase-independent implied constant. Indeed, put \(L=\log n\) and
\(\ell=\log\log n\). The explicit bounds in
Section~\ref{subsec:explicit-global-ledger} give
\(\Lambda_n=O(n^{3/4}L)\). Hence the quantities in
\eqref{eq:amplification-markov-scales} satisfy
\[
 b_n=O(n^{7/8}L^{-1/2}),\qquad
 a_n=O(n^{15/16}L^{-7/4})=o(n\ell/L^4).
\]
Combining (5.11), (4.1), and (5.13) therefore gives
\[
 \theta_n=O(\ell/L),\qquad \eta_n=o(\ell/L).
\]
The error \(\varepsilon_n^{\mathrm{gap}}=|\theta_n|+\eta_n\) used above
has the asserted rate. Equivalently, the phase-resolved lower bound has
deterministic remainder \(O(n\log\log n/(\log n)^4)\).
All estimates are uniform over the phase.
\end{remark}

The pair bound in Section~\ref{a-valid-unrestricted-lower-location-for-chi}
sharpens the probability estimate enough to sum the failures over all \(n\).

\begin{corollary}[A summable bound on the failure probability]
\label{cor:phase-failure-probability}
Put \(q=\log2\), \(\beta=2/q+1/2\), \(L=\log n\), and
\[
 w_n:=\frac{L^\beta}{n}
       \left(\frac Ln\right)^{1-\delta_n}
     =n^{\delta_n-2}L^{\beta+1-\delta_n}.
\]
There are constants \(C_1,C_2>0\) and an integer \(N\), all independent of
the phase, such that every integer \(n\ge N\) satisfies
\[
 \begin{split}
 &\mathbb P\!\left(
  \chi(G_n)-\zeta(G_n)
  <\left[\frac{q^2}{4}A_5(\delta_n)
       -C_1\frac{\log\log n}{\log n}\right]\frac n{(\log n)^3}
 \right)
 \\[-2pt]
 &\qquad\le C_2\left(w_n^2+\frac{w_nL^{5/2}}{\sqrt n}+n^{-2}\right).
 \end{split}
\]
The right-hand side is \(O(n^{-3/2}(\log n)^{\beta+5/2})\), uniformly in
the phase, and is summable over \(n\). For each fixed \(c<c_5\), the
uniform lower bound in the main theorem with coefficient \(c\) has the same
failure bound, after increasing \(N\).
\end{corollary}

Section~\ref{app:failure-probability-proof} proves the corollary. It applies
the explicit tail bound of Lemma~\ref{lemma-10.2-amplification-from-a-seed}
with tail parameter $2\log n$, together with the sharper chromatic failure
bound (4.5) and the expansion (2.3).

\begin{corollary}[Almost-sure eventual gap]\label{cor:almost-sure-gap}
Fix any joint law of \((G_n)_{n\ge1}\) whose \(n\)-th marginal is
\(G(n,1/2)\). With probability one, for every sufficiently large \(n\),
\[
 \chi(G_n)-\zeta(G_n)
 \ge\left[\frac{(\log2)^2}{4}A_5(\delta_n)
       -C_1\frac{\log\log n}{\log n}\right]\frac n{(\log n)^3}.
\]
For each fixed \(c<c_5\), the uniform lower bound in the main theorem with
coefficient \(c\) also holds almost surely for every sufficiently large \(n\).
\end{corollary}
\begin{proof}
Under the fixed joint law, let \(F_n\) be failure of the phase-resolved bound.
The preceding corollary bounds each marginal failure probability and gives
\(\sum_n\mathbb P(F_n)<\infty\). The first Borel--Cantelli lemma therefore
implies that only finitely many \(F_n\) occur. Independence across \(n\) is
not required. The eventual threshold may depend on the realization; the
assertion is for each fixed joint law.
\end{proof}

\paragraph{AI assistance disclosure.}
OpenAI GPT-5.6 Sol, accessed through ChatGPT and Codex, was used to probe the
written proof for counterexamples, check symbolic consistency, and generate
candidate Lean code. Codex also assisted with supplementary proof arguments
for the finite-dual approximation and residual overlap bounds. Aristotle,
Harmonic's automated theorem-proving system \citep{achim-et-al-2025},
performed automated Lean proof search and a clean-environment replay. Every
formal output incorporated into the released development was reviewed and
rebuilt under the pinned Lean toolchain. Machine-checked claims rest on the
Lean kernel, and manuscript-only claims on the written proof. Samuil Petkov
directed the work, selected the claims and arguments, reviewed the manuscript
and formalization, and assumes responsibility for the accuracy, integrity,
citations, and final submission.

\renewcommand{\bibsection}{\section*{References}}
\begingroup
\footnotesize
\bibliographystyle{plainnat}
\bibliography{references}
\endgroup

\appendix

\section{Uniform first-moment estimates}
\label{app:first-moment-estimates}

This appendix proves the three first-moment lemmas whose proofs were
deferred: the phase expansion of Lemma~\ref{lem:phase-expansion}, the root,
slope and dual estimates of Lemma~\ref{lem:uniform-root-slope-dual}, and the
independent-set pair bound of Lemma~\ref{lem:independent-pair-tail}. The
first two are used throughout the paper, and the third removes the size cap
in Section~\ref{a-valid-unrestricted-lower-location-for-chi}; in its proof,
as there, $X_s$ counts the independent $s$-subsets of $G_n$ and
$\mu_s=\mathbb E X_s$. We use the notation of Section~\ref{the-complete-independence-number-phase}:
$q=\log 2$, $L=\log n$, $\ell=\log\log n$ (not a number of common classes),
$C=1+\log q-q$, $\alpha=\lfloor\alpha_0\rfloor$, $\delta=\alpha_0-\alpha$
and $b=1-\delta$. In the proof of Lemma~\ref{lem:uniform-root-slope-dual},
$n$ is large enough that $\alpha\ge10$, as in
Section~\ref{continuous-profile-roots}. Every threshold and implied constant
below is independent of the phase, including along sequences approaching
either endpoint of $[0,1]$.

\subsection{Uniform Stirling expansion and adjacent-size ratios}
\label{app:phase-expansion-details}

\begin{proof}[Proof of Lemma~\ref{lem:phase-expansion}]
We first control the remainder in the logarithmic count at the integer
center. Exact adjacent-size ratios will then give the upper and lower
size estimates. Because $\alpha=O(L)$ uniformly in the phase, one has
$\alpha<n/2$ for all sufficiently large $n$. For $0\le x\le1/2$,
$-2x\le\log(1-x)\le-x$. Hence
\[
  \log(n)_\alpha
  =\alpha L+\sum_{j=0}^{\alpha-1}\log\!\left(1-\frac jn\right)
  =\alpha L+R_{n,\alpha},
  \qquad
  |R_{n,\alpha}|\le\frac{\alpha^2}{n}.
\]
Using the uniform Stirling remainder
\[
  \log(\alpha!)
  =\alpha\log\alpha-\alpha
   +\frac12\log(2\pi\alpha)+O(1/\alpha)
\]
in
\[
  \log\mu_\alpha
  =\log(n)_\alpha-\log(\alpha!)-q\binom\alpha2
\]
gives
\begin{equation}
  \log\mu_\alpha
  =
  \alpha\left(
    L-\log\alpha+1-\frac q2(\alpha-1)
  \right)
  -\frac12\log(2\pi\alpha)
  +O(1/L).
  \label{eq:D-log-mu}
\end{equation}
All implied constants and remainder bounds here are uniform in $\delta$.

Write
\[
  \alpha=\frac{2L}{q}(1+x_n),
  \qquad
  x_n:=\frac{-\ell+C+qb/2}{L}.
\]
Uniformly for the actual range $0<b\le1$ (and on its closure), one has
$|x_n|\le1/2$ for all sufficiently large $n$. The Taylor formula
$\log(1+x)=x+O(x^2)$ therefore gives
\[
  \log\alpha
  =\ell+q-\log q
   +\frac{-\ell+C+qb/2}{L}
   +O\!\left(\frac{\ell^2}{L^2}\right).
\]
Since
\[
  \frac q2(\alpha-1)=L-\ell+C-\frac{q\delta}{2},
\]
we obtain the uniform expansion
\begin{equation}
  L-\log\alpha+1-\frac q2(\alpha-1)
  =
  \frac{q\delta}{2}
  +\frac{\ell-C-qb/2}{L}
  +O\!\left(\frac{\ell^2}{L^2}\right).
  \label{eq:D-bracket}
\end{equation}
Multiplying by
$\alpha=2(L-\ell+C)/q+b$ and collecting terms yields
\[
\begin{split}
  \alpha\left(
    L-\log\alpha+1-\frac q2(\alpha-1)
  \right)
  ={}&
  \delta L+\left(\frac2q-\delta\right)\ell\\
  &+\delta C+\frac q2\delta(1-\delta)
    -\frac{2C}{q}-(1-\delta)
    +O\!\left(\frac{1+\ell^2}{L}\right).
\end{split}
\]
Also
\[
  -\frac12\log(2\pi\alpha)
  =-\frac12\ell-\frac12\log(2\pi)-\frac q2
    +\frac12\log q+O(\ell/L).
\]
Thus (2.2) holds with
\begin{equation}
\begin{split}
  K(\delta):={}&
  \delta C+\frac q2\delta(1-\delta)-\frac{2C}{q}-(1-\delta)\\
  &-\frac12\log(2\pi)-\frac q2+\frac12\log q.
\end{split}
  \label{eq:D-K}
\end{equation}
This function is continuous and therefore bounded on $[0,1]$. Combining
the remainders in \eqref{eq:D-log-mu}, \eqref{eq:D-bracket} and the
expansion of $\log(2\pi\alpha)$ bounds their absolute value by
$C'(1+\ell^2)/L$ for an absolute $C'$. Choose
$C_{\mathrm{ph}}\ge C'$; then
$|E_n^{\mathrm{ph}}|\le\varepsilon_n^{\mathrm{ph}}$ as required in (2.2),
with $E_n^{\mathrm{ph}}$ denoting that combined remainder.

For integers $1\le s<n$, the adjacent-size ratios are exact:
\begin{equation}
  \frac{\mu_{s+1}}{\mu_s}
  =\frac{n-s}{s+1}2^{-s},
  \qquad
  \frac{\mu_{s-1}}{\mu_s}
  =\frac{s}{n-s+1}2^{s-1}.
  \label{eq:D-adjacent}
\end{equation}
The displayed formula for $\alpha$ gives the exact phase representation
\[
  2^\alpha
  =\exp(q\alpha)
  =\exp(2C+qb)\frac{n^2}{L^2}.
\]
Since $0<b\le1$, the multiplicative factor $\exp(2C+qb)$ is bounded
above and below by positive absolute constants. Therefore, for every fixed
$M$, uniformly over integers $s$ with $|s-\alpha|\le M$,
\[
  \frac{\mu_{s+1}}{\mu_s}=\Theta(L/n),
  \qquad
  \frac{\mu_{s-1}}{\mu_s}=\Theta(n/L).
\]
In particular,
\begin{equation}
  \mu_{\alpha+2}
  =\mu_\alpha\,\Theta(L^2/n^2),
  \qquad
  \mu_{\alpha-2}
  =\mu_\alpha\,\Theta(n^2/L^2).
  \label{eq:D-two-step}
\end{equation}
Taking logarithms and using (2.2) proves the first line of (2.3). Since
$\delta\le1$, its leading term is at most $-L$, and the coefficient of
$\ell$ is uniformly bounded; this proves the second line.

For the lower adjacent size, (2.2) and \eqref{eq:D-two-step} give
\[
\begin{split}
  \log\mu_{\alpha-2}
  ={}&2L+\left(\frac2q-\frac52\right)\ell
      +\delta(L-\ell)+K(\delta)+O(1).
\end{split}
\]
For large $n$, $L-\ell>0$, while $K$ is bounded below. Exponentiating gives
(2.4), with one absolute constant $c$ and one threshold valid for
the full phase interval.
\end{proof}

\subsection{Proof of the uniform root estimates}
\label{app:uniform-root-proof}

\begin{proof}[Proof of Lemma~\ref{lem:uniform-root-slope-dual}]
We first compare the finite and limiting duals, which proves (3.9)--(3.9b);
then we prove the identity (3.8); finally we locate the roots and bound the
slopes, which proves (3.6)--(3.7). The quantities $h_n$, $S^{(n)}$,
$Z_{n,S}$, $M_{n,S}$, $\lambda_{n,S}$ and $\mathcal F_{n,S}$ depend on $n$
only through the integer $\alpha$, and $T$ ranges over the fixed interval
$K_*$; so the first two steps are uniform in the phase automatically. In
the last step the phase enters only through $b=1-\delta\in[0,1]$ in (2.1),
that is, through $s_0$ and $T_0\in[2/q,1+2/q]$.

Dividing (3.2) by $k=n/s$ and writing $p_i=k_i/k$ gives the formula below.
All sums in this formula are over $S^{(n)}$. The maximum is over nonnegative
weights with total mass one and mean deficit $\alpha-s$.
\[
  \frac{L_S(n,n/s)}{n/s}
  =(s-1)L-s+1+\log s
   +\max
      \left[-\sum_i p_i\log p_i
            -\sum_i p_i\log d_{\alpha-i}\right].
\]
The definition of the remainder in (3.1a) gives
\begin{equation}
  h_n(i)=-\frac q2i^2+O(i^2/\alpha)
  \quad\text{for every fixed }i,
  \qquad
  h_n(i)\le-\frac q2i^2
  \quad(i\ge-1).
  \label{eq:root-curvature}
\end{equation}
The first estimate controls each fixed finite set of deficits uniformly.
The second applies to the whole finite support.

\displayheading{Finite and limiting duals and optimizing tilts}
We first choose the exponential tilt so that the weighted distribution has
the required mean. The target center satisfies $2/q\le T_0\le1+2/q$, so it
lies in the interior of $K_*$. Every target in
$K_*$ is also strictly inside the feasible mean range of both supports.

Choose $\Lambda>0$ so that
\[
  M_S(-\Lambda)<\min K_*<\max K_*<M_S(\Lambda)
  \qquad(S=S_+,S_5).
\]
The limiting mean maps are strictly increasing, and their endpoint limits
bracket $K_*$, so this choice is possible. Before inverting the mean maps,
we compare the finite weights with their limiting Gaussian weights. Write
\[
  w_i(\lambda):=\exp\!\left(\lambda i-\frac q2i^2\right),
  \qquad
  \rho_n(i):=-h_n(i)-\frac q2i^2
\]
for indices in the finite support. Formula (3.1a) shows that this loss in
the exponent is nonnegative, and
\[
  0\le\rho_n(-1)=\log(1+1/\alpha)\le1/\alpha.
\]
For $0\le i\le\alpha/2$, use $-\log(1-x)\le2x$ on $0\le x\le1/2$ to get
\[
  0\le\rho_n(i)
  =-\sum_{r=0}^{i-1}\log\!\left(1-\frac r\alpha\right)
  \le\frac{i(i-1)}{\alpha}
  \le\frac{i^2}{\alpha}.
\]
Since $0\le1-e^{-\rho_n(i)}\le\rho_n(i)$, for $m=0,1,2$ the retained low
indices satisfy
\[
  \sum_{\substack{i\in S^{(n)}\\i\le\alpha/2}}
  |i|^m\left|e^{\lambda i+h_n(i)}-w_i(\lambda)\right|
  \le\frac{C_m}{\alpha}
  \qquad(|\lambda|\le\Lambda).
\]
Here $|i|^0$ means one. The constants are finite because every polynomial
times $\exp(\Lambda|i|-qi^2/2)$ is summable.

For the remaining indices, \eqref{eq:root-curvature} bounds the difference by the
limiting weight itself. The same tail also contains every index omitted by
the finite cutoff. For sufficiently large $\alpha$,
\[
  \sum_{i>\alpha/2}i^m e^{\Lambda i-qi^2/2}
  \le\sum_{i>\alpha/2}i^m e^{-qi^2/4}
  \le C_m e^{-c\alpha^2}.
\]
Comparison with a Gaussian integral proves the last bound: the fixed
polynomial can be absorbed by weakening the Gaussian exponent. The tail for
$S_5$ is eventually empty. Combining the low-index and tail bounds, the
zeroth, first and second unnormalized tilted moments differ by
$O(1/\alpha)$ uniformly on $[-\Lambda,\Lambda]$, for both supports. The
zeroth moments also have a common positive lower bound. Indeed, retain the
$i=2$ term, which occurs in both finite supports; the tilt interval and
$h_n(2)$ are bounded. We may therefore take logarithms and ratios to obtain,
uniformly on this interval,
\[
\begin{split}
  \log Z_{n,S}-\log Z_S&=O(1/\alpha),\\
  M_{n,S}-M_S&=O(1/\alpha),\\
  M'_{n,S}-M'_S&=O(1/\alpha).
\end{split}
\]

For all sufficiently large $n$, the finite support $S^{(n)}$ contains at
least two indices, and hence
$M'_{n,S}(\lambda)=\operatorname{Var}_{n,S,\lambda}(i)>0$ for every
$\lambda\in\mathbb R$. Likewise,
$M'_S(\lambda)=\operatorname{Var}_{S,\lambda}(i)>0$ for every $\lambda$.
Thus both mean maps are globally strictly increasing. Take the minimum
below over $S\in\{S_+,S_5\}$ and $|\lambda|\le\Lambda$. Continuity and
compactness give
\[
  v_*:=\frac12
  \min_{S,\lambda}
  M'_S(\lambda)>0.
\]
For all sufficiently large $n$, one has
$M'_{n,S}(\lambda)\ge v_*$ on the same interval, and the finite mean maps
also bracket $K_*$. Hence, for every $T\in K_*$, the equations
\[
  M_{n,S}(\lambda_{n,S}(T))=T,
  \qquad
  M_S(\lambda_S(T))=T
\]
have unique solutions in $[-\Lambda,\Lambda]$. The derivative lower bound
makes the inverse mean maps uniformly $v_*^{-1}$-Lipschitz. Consequently
\[
  \sup_{T\in K_*}
  |\lambda_{n,S}(T)-\lambda_S(T)|
  \le
  v_*^{-1}\sup_{|\lambda|\le\Lambda}
  |M_{n,S}(\lambda)-M_S(\lambda)|
  =O(1/\alpha).
\]

For the finite supports, the entropy functional in (3.8a) is strictly
concave on its feasible simplex. Its unique optimizer is
\[
  p_{n,S,T}(i)
  =\frac{
    \exp\{\lambda_{n,S}(T)i+h_n(i)\}
  }{Z_{n,S}(\lambda_{n,S}(T))},
\]
and its dual value is exactly
\[
  \mathcal F_{n,S}(T)
  =\log Z_{n,S}(\lambda_{n,S}(T))
   -\lambda_{n,S}(T)T.
\]

For the countable support $S_+$, strict concavity alone does not establish
attainment. We verify it using the tilted distribution. Fix $T\in K_*$,
write $\lambda=\lambda_{S_+}(T)$, and set
\[
 \pi_\lambda(i):=
 \frac{\exp\{\lambda i-q i^2/2\}}{Z_{S_+}(\lambda)}.
\]
Write
\[
  D_{\mathrm{KL}}(p\Vert\pi_\lambda)
  :=\sum_i p_i\log\frac{p_i}{\pi_\lambda(i)},
\]
with the convention that a summand with $p_i=0$ is zero.
The tilted distribution defined above has mean $T$. For every feasible
probability vector $p$ with finite quadratic moment,
\begin{equation}
 -\sum_i p_i\log p_i-\frac q2\sum_i i^2p_i
 =\log Z_{S_+}(\lambda)-\lambda T
  -D_{\mathrm{KL}}(p\Vert\pi_\lambda).
 \label{eq:root-gibbs-identity}
\end{equation}
The identity follows first for finitely supported $p$ and then by truncation;
finite quadratic moment also makes the entropy finite by comparison with a
Gaussian law.  We interpret the objective as $-\infty$ when the quadratic
moment diverges.  Since relative entropy is nonnegative and vanishes only at
$p=\pi_\lambda$, \eqref{eq:root-gibbs-identity} proves finiteness, attainment, uniqueness, and the
dual formula for $S_+$.  The same identity also recovers the finite-support
case. The dual objectives for the finite and limiting supports attain
their minima in the same tilt interval. Their difference at a fixed tilt
is $\log Z_{n,S}-\log Z_S$, so the difference of their minimum values is
at most the uniform norm of this quantity. Together with the tilt estimate,
this proves the uniform convergence $\mathcal F_{n,S}\to\mathcal F_S$ on
$K_*$ and
\[
  \varepsilon_n^{\mathrm{dual}}=O(1/\alpha).
\]
On the finite support $S_5$, substitute the exponent and tilt error bounds into
the optimizer formula to obtain the coordinatewise rate (3.9b). The limiting
coordinates are positive and continuous in $T$ on the compact set $K_*$.
They therefore have a positive minimum, and the coordinatewise rate transfers
this lower bound to the finite optimizers.

\displayheading{Exact cancellation between supports}
We next return to the counting exponent. Substitution of \eqref{eq:root-affine} gives
\begin{equation}
\begin{split}
  \Psi_{n,S}(s)
  &:=\frac{L_S(n,n/s)}{n/s}\\
  &=(s-1)L-s+1+\log s
    +A_n+B_n(\alpha-s)
    +\mathcal F_{n,S}(\alpha-s).
\end{split}
  \label{eq:root-counting-exponent}
\end{equation}
Fix a class count $k$ feasible for both supports. They then have the same
mean class size $s$ and mean deficit $T=\alpha-s$. The affine terms in
\eqref{eq:root-counting-exponent} cancel, leaving exactly the difference of
the two finite dual values in (3.8). This comparison uses the same numbers
of parts and vertices on both sides, before taking either limiting dual.

\displayheading{The exponent at the phase center and its nearby slope}
To locate the zero, we estimate the exponent at the phase center and its
derivative nearby. At $s=s_0$, let $T_0=\alpha-s_0$.
Since $A_n=-q\binom\alpha2-\log(\alpha!)$ and $\alpha=s_0+T_0$, the terms of
\eqref{eq:root-counting-exponent} involving $A_n$ and $B_n$ satisfy
\[
 -q\binom\alpha2+\Bigl(q\alpha-\frac q2\Bigr)T_0
 =-\frac q2s_0^2+\frac q2s_0+\frac q2T_0^2,
 \qquad
 (T_0-\alpha)\log\alpha+\alpha=s_0(1-\log\alpha)+T_0 .
\]
Applying (1.2) to $\log(\alpha!)$ and adding $(s_0-1)L-s_0+1+\log s_0$ gives
\begin{equation}
\begin{split}
  \Psi_{n,S}(s_0)-\mathcal F_{n,S}(T_0)
  ={}&s_0\left(
    L-\log\alpha-\frac q2s_0+\frac q2
  \right)-L+T_0+1+\log s_0+\frac q2T_0^2\\
  &\hspace{42mm}
  -\frac12\log(2\pi\alpha)+O(1/L).
\end{split}
  \label{eq:root-stirling-expansion}
\end{equation}
Equation (2.1) gives the exact identity
\[
  \frac{qs_0}{2}=L-\ell+\log q-q.
\]
Since $\alpha=s_0+T_0$ and $T_0$ is uniformly bounded,
\[
  \log\alpha
  =\log s_0+\frac{T_0}{s_0}+O(L^{-2}),
  \qquad
  \log s_0=\ell+q-\log q+O(\ell/L).
\]
Consequently
\[
\begin{split}
  s_0\left(
    L-\log\alpha-\frac q2s_0+\frac q2
  \right)
  &=-T_0+\frac q2s_0+O(\ell)\\
  &=-T_0+L-\ell+\log q-q+O(\ell).
\end{split}
\]
Substitution in \eqref{eq:root-stirling-expansion} cancels the terms $L$ and $T_0$ and leaves
$O(\ell)$ uniformly in the phase. The dual term in \eqref{eq:root-counting-exponent} is $O(1)$ uniformly
on $K_*$, by the Gaussian majorant above. Therefore
\begin{equation}
  \Psi_{n,S}(s_0)=O(\ell)
  \label{eq:root-central-value}
\end{equation}
uniformly for both supports.

The phase formula for $\alpha$ and the expansion of $\log\alpha$ also give
\begin{equation}
  B_n=2L-\ell+O(1)
  \label{eq:root-scaled-exponent}
\end{equation}
uniformly in $\delta$. Since $M'_{n,S}\ge v_*>0$ on $[-\Lambda,\Lambda]$,
the inverse function theorem makes $\lambda_{n,S}$ differentiable on $K_*$.
Differentiating $\mathcal F_{n,S}(T)=\log Z_{n,S}(\lambda_{n,S}(T))
-\lambda_{n,S}(T)T$ and using $M_{n,S}(\lambda_{n,S}(T))=T$ gives
\[
  \frac{d}{dT}\mathcal F_{n,S}(T)=-\lambda_{n,S}(T).
\]
Differentiating \eqref{eq:root-counting-exponent}, with $T=\alpha-s$, yields
\begin{equation}
  \Psi'_{n,S}(s)
  =L-1+s^{-1}-B_n+\lambda_{n,S}(T)
  =-L+O_A(\ell)
  \label{eq:root-scaled-derivative}
\end{equation}
uniformly whenever $|s-s_0|\le A\ell/L$. Indeed, the image of this corridor
under $T=\alpha-s$ lies inside $K_*$ for all sufficiently large $n$, because
$K_*$ has a fixed margin around the complete range of $T_0$.

\displayheading{Roots and slopes in the phase corridor}
Set
\[
  \Psi_{n,S,c}(s):=\Psi_{n,S}(s)+c.
\]
By \eqref{eq:root-central-value}, there is an absolute $C_0$ such that
$|\Psi_{n,S,c}(s_0)|\le C_0\ell$ for both supports and every
$0\le c\le q$. By \eqref{eq:root-scaled-derivative}, after one phase-independent threshold,
\[
  \Psi'_{n,S,c}(s)\le-\frac12L
\]
throughout each fixed corridor. Choose $A_*>2C_0$. Integration gives
\[
\begin{aligned}
  \Psi_{n,S,c}\!\left(s_0-A_*\frac\ell L\right)
  &\ge\left(\frac{A_*}{2}-C_0\right)\ell>0,\\
  \Psi_{n,S,c}\!\left(s_0+A_*\frac\ell L\right)
  &\le\left(C_0-\frac{A_*}{2}\right)\ell<0.
\end{aligned}
\]
The opposite endpoint signs give a zero, and strict decrease in $s$ makes
it unique on this interval. Since $k=n/s>0$, this gives the
\emph{corridor} root $r_{S,c}$ of $L_S(n,k)+ck=0$ in (3.6). The uniqueness
claim is restricted to this corridor.

Finally,
\[
  L_S(n,k)+ck=k\Psi_{n,S,c}(s),
  \qquad
  \frac{ds}{dk}=-\frac sk.
\]
Hence
\begin{equation}
\begin{split}
  \frac{\partial}{\partial k}\{L_S(n,k)+ck\}
  &=\Psi_{n,S,c}(s)-s\Psi'_{n,S}(s)\\
  &=\frac2qL^2+O_A(L\ell),
\end{split}
  \label{eq:root-final-slope}
\end{equation}
because $s=(2/q)L+O_A(\ell)$, \eqref{eq:root-scaled-derivative} holds, and
$\Psi_{n,S,c}(s)=O_A(\ell)$ throughout the corridor. This proves (3.7) and
$\varepsilon_{n,A}^{\mathrm{slope}}\to0$ with one threshold for
the full phase interval.
\end{proof}

\subsection{Proof of the independent-set pair bound}
\label{app:pair-bound}

\begin{proof}[Proof of Lemma~\ref{lem:independent-pair-tail}]
Write \((x)_j=x(x-1)\cdots(x-j+1)\), with \((x)_0=1\).
For independently and uniformly sampled \(s\)-subsets \(S,T\) of \([n]\), let
\(p_j=\mathbb P(|S\cap T|=j)\). Counting ordered distinct pairs gives
\[
 \mathbb E[X_s(X_s-1)]
 =\mu_s^2\sum_{j=0}^{s-1}p_j\,2^{\binom j2}.
\]
For \(0\le j\le s/2\), a union bound over \(j\)-subsets of \(S\) gives
\[
 p_j\le\binom sj\frac{(s)_j}{(n)_j}
 \le\frac1{j!}\left(\frac{s^2}{n-s}\right)^j.
\]
Since \(2^{(j-1)/2}\le2^{s/4}\), this part of the sum is at most \(e^a\).

For \(j>s/2\), write \(r=s-j\), so \(1\le r<s/2\).
Conditional on a fixed \(S\) being independent in \(G_n\), the expected
number of independent partners with this overlap is
\[
 \binom sr\binom{n-s}{r}
 2^{-r(2s-r-1)/2}
 \le\frac{b^r}{(r!)^2}.
\]
Indeed, the partner requires \(r(2s-r-1)/2\) additional absent edges, and
\(s-(r+1)/2\ge3s/4-1/2\). Summing over the first set contributes \(\mu_s\)
times this conditional expectation. The large-overlap contribution is
therefore at most
\(\mu_s\sum_{r\ge1}b^r/(r!)^2\le\mu_s(e^b-1)\).
The two ranges exhaust the distinct pairs; the diagonal \(j=s\) is excluded.
\end{proof}

\Needspace{8\baselineskip}
\section{Bounds for the five-size phase gain}
\label{app:phase-gain}

This appendix proves the curvature of the phase gain
(Lemma~\ref{lem:phase-gain-concavity}), the four-size comparison
(Lemma~\ref{lem:four-size-entropy-certificate}) and the two lower bounds for
$A_5$ (Lemmas~\ref{lem:entropy-gap} and~\ref{lem:phase-gain-sharp}). It also
records numerical values of $A_5$ at the endpoints of the phase interval,
which are not used in any proof. We use the notation of
Section~\ref{the-five-size-signed-first-moment-advantage}: $q=\log 2$,
$w_i(\lambda)=\exp(\lambda i-qi^2/2)$ and, for $S\in\{S_+,S_4,S_5\}$, the
partition function $Z_S$, the mean map $M_S$, the tilt $\lambda_S(T)$ and the
dual value $\mathcal F_S(T)$. By (2.0),
$T_\delta=1+2/q-\delta\in[2/q,1+2/q]\subset(14/5,4)$ for $0\le\delta\le1$.
All arguments are deterministic and involve no limit in $n$.

\subsection{Proof of the curvature of the phase gain}
\label{app:phase-gain-concavity}

\begin{proof}[Proof of Lemma~\ref{lem:phase-gain-concavity}]
The Gaussian factor permits termwise differentiation locally in $\lambda$.
Since $M'_S(\lambda_S(T))=V_S(T)>0$, differentiating the inverse mean equation and the dual
value gives
\[
 \lambda_S'(T)=\frac1{V_S(T)},\qquad
 \mathcal F_S'(T)=-\lambda_S(T),\qquad
 \mathcal F_S''(T)=-\frac1{V_S(T)}.
\]

For the variance comparison, fix $T$ and write $p_+$ and $p_5$ for the two
probability vectors in (5.1), extending $p_5$ by zero off $S_5$. On $S_5$,
\[
 \frac{p_5(i)}{p_+(i)}
 =\frac{Z_{S_+}(\lambda_{S_+}(T))}
        {Z_{S_5}(\lambda_{S_5}(T))}
   e^{(\lambda_{S_5}(T)-\lambda_{S_+}(T))i}.
\]
Within the five-size support, this ratio is monotone in the index.
Thus the indices where $d_i:=p_5(i)-p_+(i)>0$ form a consecutive block
$\{j,\ldots,k\}\subseteq S_5$. This block is nonempty because $p_+$ has
positive mass outside $S_5$ and both vectors have total mass one. Put
$a=j-1/2$ and $b=k+1/2$. Then $(i-a)(i-b)d_i\le0$ for every $i\in S_+$:
the quadratic factor is negative on the block and positive at all other
integer indices. The inequality is strict at $i=j$. Both vectors have mass one and mean $T$,
so the constant and linear terms of the quadratic cancel in the sum:
\[
 V_{S_5}(T)-V_{S_+}(T)
 =\sum_{i\in S_+}i^2d_i
 =\sum_{i\in S_+}(i-a)(i-b)d_i<0.
\]
All sums converge absolutely. The nonstrict inequality is also a standard
consequence of the equal-mean relative log-concavity comparison
\citep[Theorem~12]{Yu2010RelativeLogConcavity}; the quadratic argument
above verifies strictness here.

Use $T_\delta'= -1$ in the derivative formulas above. Since
$T_\delta\in(2,5)$ on a neighborhood of $[0,1]$, the resulting formulas
remain valid at both endpoints. The minimum is attained at an endpoint
because the phase gain is strictly concave.
\end{proof}

\subsection{Tilt brackets and tails for the entropy bound}
\label{app:entropy-certificate-details}

\begin{proof}[Proof of Lemma~\ref{lem:four-size-entropy-certificate}]
Fix \(\delta\in[0,1]\), put
\(T_0=1+2/{\log 2}-\delta\), and abbreviate

\[
 \lambda_4:=\lambda_{S_4}(T_0).
\]

We compare the two dual values by evaluating the unrestricted dual at the
four-size optimizer. First we bound the tilt; these bounds will control the
weights omitted from the four-size support. Put \(\vartheta=2^{-1/2}\).
At tilt \(2\log 2\), the weights \(w_i\), \(i=2,3,4,5\), divided by the
first, are \(1,\vartheta,1/4,\vartheta/16\). Hence
\begin{equation}
 M_{S_4}(2\log 2)=2+\frac{19\vartheta/16+1/2}{5/4+17\vartheta/16}
 <\frac{14}{5}<\frac2{\log 2}\le T_0,
 \label{eq:E-low-tilt}
\end{equation}
because four-fifths of the denominator minus the numerator is
\((40-27\vartheta)/80>0\), and \(\log 2<7/10\) by (2.0).

At \(\lambda=9{\log 2}/2\), put \(t=2^{-1/8}\). The identity
\[
 \frac{9\log 2}{2}i-\frac{\log 2}{2}i^2
 =\frac{\log 2}{8}\{81-(2i-9)^2\}
\]
gives
\[
 w_i(9{\log 2}/2)=2^{81/8}t^{(2i-9)^2}.
\]
Thus the four weights, in the order
\(i=2,3,4,5\), are proportional to

\[
 t^{25},\quad t^9,\quad t,\quad t.
\]

\Needspace{6\baselineskip}
The numerator of \(M_{S_4}(9{\log 2}/2)-4\) is therefore

\begin{equation}
 t-t^9-2t^{25}=t(1-t^8-2t^{24})=t/4>0,                 \label{eq:E-high-tilt}
\end{equation}
and the denominator is positive.  Hence this mean is greater than four,
whereas \(T_0\le1+2/{\log 2}<4\).  Since \(M_{S_4}\) is strictly
increasing, the two mean comparisons give

\begin{equation}
 2{\log 2}<\lambda_4<9{\log 2}/2.                                     \label{eq:E-tilt-bracket}
\end{equation}

For a tilt \(\lambda\), define the omitted low- and high-deficit ratios

\[
 R_{\mathrm{low}}(\lambda):=
 \frac{\sum_{i=-1}^{1}w_i(\lambda)}{Z_{S_4}(\lambda)},
 \qquad
 R_{\mathrm{high}}(\lambda):=
 \frac{\sum_{i\ge6}w_i(\lambda)}{Z_{S_4}(\lambda)}.
\]

Let \(M_A(\lambda)\) denote the mean of \(i\) under the weights
\(w_i(\lambda)\) restricted to an index set \(A\). Differentiation gives
\[
\begin{aligned}
 \frac{d}{d\lambda}\log R_{\mathrm{low}}(\lambda)
 &=M_{\{-1,0,1\}}(\lambda)-M_{S_4}(\lambda)
 \le 1-2=-1,\\
 \frac{d}{d\lambda}\log R_{\mathrm{high}}(\lambda)
 &=M_{\{6,7,\ldots\}}(\lambda)-M_{S_4}(\lambda)
 \ge 6-5=1.
\end{aligned}
\]
Thus \(R_{\mathrm{low}}\) decreases and \(R_{\mathrm{high}}\) increases
on \(\mathbb R\). To bound their sum throughout the tilt interval, we split
the interval at the intermediate reference tilt and use the appropriate
endpoint bound for each ratio. The calculations below give:

\begin{equation}
 R_{\mathrm{low}}(2{\log 2})<\frac{51}{100},\quad
 R_{\mathrm{high}}(3{\log 2})<\frac1{50},\quad
 R_{\mathrm{low}}(3{\log 2})<\frac3{25},\quad
 R_{\mathrm{high}}(9{\log 2}/2)<\frac14.                 \label{eq:E-ratio-bounds}
\end{equation}

At \(2{\log 2}\),

\[
 R_{\mathrm{low}}(2{\log 2})
 =\frac{\vartheta+1/4+\vartheta/16}
        {1+\vartheta+1/4+\vartheta/16}<51/100.
\]

At \(\lambda=3\log 2\),
\[
 w_i(3\log 2)
 =\exp\!\left((\log 2)\left(3i-\frac{i^2}{2}\right)\right)
 =2^{9/2}2^{-(i-3)^2/2}.
\]
After canceling the common factor \(2^{9/2}\),
the weight at index \(i\) is \(\vartheta^{(i-3)^2}\). Thus the four
\(S_4\)-weights, for \(i=2,3,4,5\), are
\[
 \vartheta,\quad 1,\quad \vartheta,\quad \vartheta^4,
\]
and
\[
 2^{-9/2}Z_{S_4}(3\log 2)
 =1+2\vartheta+\vartheta^4=\frac54+2\vartheta.
\]
For the omitted low indices \(i=-1,0,1\), the weights are
\[
 (\vartheta^{16},\vartheta^9,\vartheta^4)
 =\left(\frac1{256},\frac{\vartheta}{16},\frac14\right).
\]
For the omitted high indices \(i\ge6\), they are
\[
 \vartheta^9+\vartheta^{16}+\sum_{r\ge5}\vartheta^{r^2}.
\]
The first term in the remaining sum is
\(\vartheta^{25}=\vartheta/4096\), and the ratio between successive terms is
at most \(\vartheta^{11}<1/32\). Therefore
\[
 \sum_{r\ge5}\vartheta^{r^2}
 \le\frac{\vartheta^{25}}{1-1/32}<\frac1{3968}.
\]
\Needspace{6\baselineskip}
It follows that
\[
 R_{\mathrm{low}}(3\log 2)
 =\frac{1/256+\vartheta/16+1/4}{5/4+2\vartheta}<\frac3{25},
\]
and
\[
 R_{\mathrm{high}}(3\log 2)
 \le\frac{\vartheta/16+1/256+1/3968}{5/4+2\vartheta}<\frac1{50},
\]
where the final inequalities follow from
\(7/10<\vartheta<71/100\).

Finally, at $\lambda=9\log 2/2$,
\[
  \frac{w_i(9\log 2/2)}{2^{81/8}}=t^{(2i-9)^2}.
\]
After dividing by the common factor $t$ of the $i=4,5$ weights, the high
indices $i=6,7,8$ contribute $t^8,t^{24},t^{48}$.
From \(i=8\) onward the ratio between successive terms is at most
\(1/16\), so the tail beginning with \(t^{48}\) is less than \(1/60\).
The normalized \(S_4\) denominator is \(2+t^8+t^{24}\). Hence
\(R_{\mathrm{high}}(9\log 2/2)<1/4\) follows from

\begin{equation}
 3(t^8+t^{24})+4/60<2,                                  \label{eq:E-high-ratio}
\end{equation}
because \(t^8=1/2\) and \(t^{24}=1/8\).

If \(\lambda_4\le 3{\log 2}\), monotonicity and \eqref{eq:E-ratio-bounds} give
\(R_{\mathrm{low}}(\lambda_4)+R_{\mathrm{high}}(\lambda_4)<53/100\); if
\(\lambda_4\ge 3{\log 2}\), they give
\(R_{\mathrm{low}}(\lambda_4)+R_{\mathrm{high}}(\lambda_4)<37/100\).
By the dual representation in
Section~\ref{the-five-size-signed-first-moment-advantage},

\[
 \mathcal F_{S_4}(T_0)
 =\log Z_{S_4}(\lambda_4)-\lambda_4T_0,
\]
whereas evaluating the \(S_+\) dual function at the same parameter gives

\[
 \mathcal F_{S_+}(T_0)
 \le \log Z_{S_+}(\lambda_4)-\lambda_4T_0.
\]

Subtracting these two relations yields

\[
 \begin{aligned}
 D_4(\delta)
 &\le \log\frac{Z_{S_+}(\lambda_4)}{Z_{S_4}(\lambda_4)}\\
 &=\log\!\bigl(1+R_{\mathrm{low}}(\lambda_4)
                    +R_{\mathrm{high}}(\lambda_4)\bigr)
 <\log(153/100).
 \end{aligned}
\]

Since \(S_4\) is a subset of \(S_+\), the loss is
nonnegative. Subtracting from \(\log 2\) completes the proof.
\end{proof}

\subsection{Proof of the uniform entropy bound}
\label{app:entropy-gap-proof}

\begin{proof}[Proof of Lemma~\ref{lem:entropy-gap}]
Write $\vartheta=2^{-1/2}<3/4$. At tilt $2q$, the five weights, after
division by their first weight, are
\[
 1,\quad\vartheta,\quad1/4,\quad\vartheta/16,\quad1/256.
\]
Consequently
\[
 M_{S_5}(2q)=2+
 \frac{19\vartheta/16+33/64}{321/256+17\vartheta/16}<14/5.
\]
Indeed, four-fifths of the denominator minus the numerator is
$(39-27\vartheta)/80>0$. At tilt $4q$, symmetry about deficit four gives
$M_{S_5}(4q)=4$. Since
$T_\delta\in[2/q,1+2/q]\subset(14/5,4)$, strict monotonicity of the mean
gives $2q<\lambda_{S_5}(T_\delta)<4q$.

Let $p_4,p_5$ be the optimizers at this same mean, extending $p_4$ by zero
at deficit six. For $\lambda_5:=\lambda_{S_5}(T_\delta)$ one has
$\log p_5(i)=\lambda_5i-\frac q2i^2-\log Z_{S_5}(\lambda_5)$; since $p_4$ has
mean $T_\delta$ and attains $\mathcal F_{S_4}(T_\delta)$, this gives
\[
 \mathcal F_{S_5}(T_\delta)-\mathcal F_{S_4}(T_\delta)
 =\sum_{i=2}^5p_4(i)\log\frac{p_4(i)}{p_5(i)}.
\]
Put $r_i=p_5(i)/(1-p_5(6))$ for $i=2,3,4,5$. These numbers sum to one.
The inequality $\log x\le x-1$ gives
$\sum_i p_4(i)\log(p_4(i)/r_i)\ge0$, so the displayed expression is at
least $-\log(1-p_5(6))$.
The mass at six increases strictly with the tilt, since its derivative is
$p_5(6)(6-M_{S_5})>0$. Hence
\[
 p_5(6)>\frac1{321+272\vartheta}>\frac1{525}.
\]
The unrestricted dual cancels in $A_5-A_4$, proving the first inequality.
Combining it with Lemma~\ref{lem:four-size-entropy-certificate} gives
$A_5>\log(200/153)+\log(525/524)=\log(8750/6681)=\gamma_5$.
Finally $S_5\subset S_+$ gives $D_5\ge0$, and $D_5=q-A_5$ gives its
upper bound.
\end{proof}

\subsection{Proof of the sharp bound for the phase gain}
\label{app:phase-gain-sharp-proof}

\begin{proof}[Proof of Lemma~\ref{lem:phase-gain-sharp}]
For $y>0$ put $\lambda(y):=q/2+\log y$. Then
$w_i(\lambda(y))=y^i2^{-i(i-1)/2}$ for every integer $i$. Write
\[
 P(y):=Z_{S_5}(\lambda(y))
 =\frac{y^2}{2}+\frac{y^3}{8}+\frac{y^4}{64}+\frac{y^5}{1024}
  +\frac{y^6}{32768},
 \qquad
 M(y):=M_{S_5}(\lambda(y)).
\]
For rational $y$, both $P(y)$ and
$M(y)$ are rational. Moreover, $M$ is increasing, because $M_{S_5}$ and
$\lambda$ are. For $y\le513/50$ one has $M(y)<4$
(see Table~\ref{tab:phase-gain-sharp-endpoints}), so $\lambda(y)$ solves the
mean equation at $M(y)$ and
$\mathcal F_{S_5}(M(y))=\log P(y)-\lambda(y)M(y)$.

\emph{The five-size dual.} The function $\lambda\mapsto\log Z_{S_5}(\lambda)$
is convex with derivative $M_{S_5}(\lambda)$. Its tangent line at
$\lambda(y)$ gives
$\log Z_{S_5}(\lambda)-\lambda M(y)\ge\mathcal F_{S_5}(M(y))$ for every
real $\lambda$. Let $y_a<y_b\le513/50$ and $T=\theta M(y_a)+(1-\theta)M(y_b)$
with $\theta\in[0,1]$. Writing $\log Z_{S_5}(\lambda)-\lambda T$ as the same
convex combination of the two tangent expressions and taking the infimum
over $\lambda$ gives
\[
 \mathcal F_{S_5}(T)\ge\theta\mathcal F_{S_5}(M(y_a))
   +(1-\theta)\mathcal F_{S_5}(M(y_b)).
\]
(This is the concavity of $\mathcal F_{S_5}$ from the proof of
Lemma~\ref{lem:phase-gain-concavity}.)

\emph{The unrestricted dual.} Let $0<z\le16$. For $i\ge7$,
$w_{i+1}(\lambda(z))/w_i(\lambda(z))=z2^{-i}\le1/8$, so the terms with
$i\ge7$ sum to at most $\frac87w_7(\lambda(z))=\frac87z^72^{-21}$. Hence
\[
 Z_{S_+}(\lambda(z))\le U(z):=\frac1{2z}+1+z+P(z)+\frac{8z^7}{7\cdot2^{21}},
\]
and evaluating the infimum at $\lambda(z)$ gives
$\mathcal F_{S_+}(T)\le\log U(z)-\lambda(z)T$.

\emph{Reduction to the endpoints.} If
$T_\delta=\theta M(y_a)+(1-\theta)M(y_b)$, the two bounds give
\[
 A_5(\delta)\ge q-\log U(z)+\lambda(z)T_\delta
   +\theta\mathcal F_{S_5}(M(y_a))+(1-\theta)\mathcal F_{S_5}(M(y_b)).
\]
The right-hand side is affine in $\theta$. At $\theta=1$ it equals
$\log\bigl(2P(y_a)/U(z)\bigr)+M(y_a)\log(z/y_a)$, and at $\theta=0$ it is
the same expression with $y_b$ in place of $y_a$. Thus
\[
 A_5(\delta)-\log\frac{169}{100}\ge\min\{E(y_a,z),E(y_b,z)\},
 \qquad
 E(y,z):=\log\frac{200P(y)}{169\,U(z)}+M(y)\log\frac zy,
\]
for every $\delta$ with $M(y_a)\le T_\delta\le M(y_b)$.

\emph{Two cells.} Use the cells
\[
 (y_a,y_b,z)=\Bigl(\frac{1907}{500},\frac{703}{100},\frac{527}{100}\Bigr),
 \qquad
 (y_a,y_b,z)=\Bigl(\frac{703}{100},\frac{513}{50},\frac{213}{25}\Bigr).
\]
The endpoint values after this proof show that
$M(1907/500)<2/q$ and $M(513/50)>1+2/q$, and that $E$ is positive at the
four endpoint pairs. Since $T_\delta\in[2/q,1+2/q]$ for $\delta\in[0,1]$
and $M$ is increasing, every $T_\delta$ lies in one of the two cells.
\end{proof}

\subsection{Endpoint values for the sharp phase-gain bound}
\label{app:phase-gain-sharp-endpoints}

The values below complete the proof of Lemma~\ref{lem:phase-gain-sharp}
and use its notation. For rational \(y\) and \(z\), the numbers \(P(y)\),
\(M(y)\) and \(U(z)\) are rational, so only the logarithms require
estimates. Every logarithm below is bounded from below by
\[
 \log m\ge\log c+2\sum_{j=0}^{5}\frac{u^{2j+1}}{2j+1}
 \qquad\text{whenever } m\ge c\,\frac{1+u}{1-u},\ 0\le u<1,
\]
where \(c\) is the largest element of \(\{1/4,1/2,1,2\}\) not exceeding
\(m\) and \(u:=(m-c)/(m+c)\). For \(c\ne1\) we use
\(0.6931471803<\log 2<0.6931471808\); these follow from
\(\log 2=2\sum_{j\ge0}3^{-2j-1}/(2j+1)\) by keeping nine terms and bounding
the remaining terms by a geometric series of ratio \(1/9\).

\emph{Breakpoints.} The bounds for \(\log 2\) give
\(2/\log 2>2.8853900\) and \(1+2/\log 2<3.8853901\). Direct evaluation gives
\(M(1907/500)<2.8853656\) and \(M(513/50)>3.8856023\). This proves
\(M(1907/500)<2/q\) and \(M(513/50)>1+2/q\).

\emph{Endpoint values.} Table~\ref{tab:phase-gain-sharp-endpoints} lists,
for each endpoint pair, a lower bound for each logarithm in \(E(y,z)\) and a
bound for \(M(y)\). The bound for \(M(y)\) is a lower bound when
\(\log(z/y)>0\) and an upper bound when \(\log(z/y)<0\), so the last column,
computed from the displayed entries, is a lower bound for \(E(y,z)\).

\begin{table}[ht]
\centering
\caption{Endpoint values for Lemma~\ref{lem:phase-gain-sharp}.}
\label{tab:phase-gain-sharp-endpoints}
\small
\begin{tabular}{cccccc}
\(y\) & \(z\) & bound for \(M(y)\) &
\(\log\dfrac{200P(y)}{169\,U(z)}\ge\) & \(\log\dfrac zy\ge\) &
\(E(y,z)\ge\)\\[4pt]
\hline
\(1907/500\) & \(527/100\) & \(\ge2.8853655\) & \(-0.9318854\) & \(0.3233518\) & \(0.00110\)\\
\(703/100\) & \(527/100\) & \(\le3.4578900\) & \(0.9981835\) & \(-0.2881564\) & \(0.00177\)\\
\(703/100\) & \(213/25\) & \(\ge3.4578899\) & \(-0.5607627\) & \(0.1922296\) & \(0.10394\)\\
\(513/50\) & \(213/25\) & \(\le3.8856024\) & \(0.8263169\) & \(-0.1858365\) & \(0.10423\)
\end{tabular}
\end{table}

All four lower bounds for \(E\) are positive, which completes the proof of
Lemma~\ref{lem:phase-gain-sharp}. The same cells, endpoint inequalities and
logarithm bounds are checked by the Lean kernel; see
Appendix~\ref{app:lean-reproducibility}.

\subsection{Numerical values at the endpoints of the phase interval}
\label{app:phase-gain-numerics}

By Lemma~\ref{lem:phase-gain-concavity}, the minimum of $A_5$ on $[0,1]$ is
attained at an endpoint. As a numerical
observation that is not used in any proof: solving the saddle-point
equations for $\mathcal F_{S_+}$ and $\mathcal F_{S_5}$ in $40$-digit
floating-point arithmetic, with the sum over $S_+$ truncated at $i\le399$, gives
$A_5(0)\approx0.659497$ and $A_5(1)\approx0.525995$. The minimum therefore
appears to be $A_5(1)$, and the bound of Lemma~\ref{lem:phase-gain-sharp}
appears to be within $0.0013$ of it.

\Needspace{8\baselineskip}
\section{Rounding the selected profile and its first moment}
\label{app:profile-rounding-proof}

This appendix completes the proof of
Proposition~\ref{prop:selected-profile-contract}. It checks that the
correction (5.16) keeps every count positive and costs
$O(1/k_{\mathrm{co}})$ in the counting exponent, and that the fixed offset
in (5.13) leaves the margin needed for (5.19). We use the notation of
Section~\ref{the-five-size-signed-first-moment-advantage}, including
$p_i^{(n)}$, $\widetilde k_i$, $e_0$, $e_1$, $\Phi_n$ and the constant $c_p$
of (5.15); every threshold is independent of the phase.

\begin{proof}[Rounding and first-moment estimates in Proposition~\ref{prop:selected-profile-contract}]
The correction satisfies \(\Delta k_2+\Delta k_3=-e_0\) and
\(2\Delta k_2+3\Delta k_3=-e_1\), so it removes both constraint errors
by changing \(O(1)\) counts. To estimate the resulting loss in the
counting exponent, put
\(\mathbf k^*=(k_{\mathrm{co}}p_i^{(n)})_{i=2}^6\) and let
\(\Delta^{\mathrm{tot}}=\mathbf k-\mathbf k^*\) be the total rounding
and correction displacement. If $f_i:=\widetilde k_i-k_i^*$ for
$i=4,5,6$, so that $|f_i|\le1$, the two conservation equations give exactly
\[
 \Delta^{\mathrm{tot}}=
 (f_4+2f_5+3f_6,\,-2f_4-3f_5-4f_6,\,f_4,f_5,f_6).
\]
Thus $\|\Delta^{\mathrm{tot}}\|_\infty\le9$ and
$\|\Delta^{\mathrm{tot}}\|_2^2\le120$. Set $C_\Delta=9$.
The displacement preserves both conservation equations, so it lies in the
tangent space of the feasible affine plane.

To apply Taylor's formula, we must keep the segment from the optimizer to
the corrected profile inside the positive orthant. By (5.15),
$k_i^*\ge c_pk_{\mathrm{co}}$; hence, uniformly in $t\in[0,1]$ and in the phase,
\[
  k_i^*+t\Delta_i^{\mathrm{tot}}
  \ge c_pk_{\mathrm{co}}-C_\Delta
  \ge \frac{c_p}{2}k_{\mathrm{co}}
\]
for all sufficiently large $n$. In particular the correction preserves
positivity. At the exact constrained optimizer the linear term in every
tangent direction vanishes, while at a positive vector $x$ the Hessian of the
entropy term is $-\operatorname{diag}(1/x_i)$. Its operator norm along the
rounding segment is therefore at most $2/(c_pk_{\mathrm{co}})$. Taylor's
formula with integral remainder gives

\[
 \left|L_{\mathbf k^*+\Delta^{\mathrm{tot}}}-L_{\mathbf k^*}\right|
 \le \frac12\sup_{0\le t\le1}
 \left\|\nabla^2L_{\mathbf k^*+t\Delta^{\mathrm{tot}}}\right\|_{\mathrm{op}}
 \|\Delta^{\mathrm{tot}}\|_2^2
 =O(1/k_{\mathrm{co}}).
\]

The fixed displacement above the root
provides the margin: $\Phi_n(r_5^{\mathrm{co}})=0$ and
$16\le k_{\mathrm{co}}-r_5^{\mathrm{co}}<17$. By Equation~(3.7), with
$S=S_5$ and $c=\log 2$, there is a phase-independent $c_1>0$ such that
$\Phi_n'(k)\ge c_1(\log n)^2$ throughout this interval for all sufficiently
large $n$.  Integration therefore gives
\[
  L_{S_5}(n,k_{\mathrm{co}})+(\log 2)k_{\mathrm{co}}
  =\Phi_n(k_{\mathrm{co}})
  \ge16c_1(\log n)^2.
\]
Rounding within the constraint plane loses $O(k_{\mathrm{co}}^{-1})$ in the
exponent, and the six Stirling approximations contribute $O(\log n)$ in
total. Both errors are smaller than the margin just obtained, so (5.19)
holds with $c_Z:=8c_1$ for all sufficiently large $n$.
\end{proof}

\section{Common whole classes}
\label{app:common-classes}

This appendix proves Lemma~\ref{lemma-7.1-all-common-subprofiles}: the
common-class weights of Section~\ref{exact-partial-diagonals} sum to
$1+o(1)$. It derives exact one-class recurrences for these weights, proves
an elementary rate inequality for five consecutive deficits, and then bounds
three ranges of selected mass. We use the deficit indexing of
Section~\ref{exact-partial-diagonals}: types are $i\in\{2,\ldots,6\}$ with
class sizes $u_i=\alpha-i$ (not the shifted sizes of
Sections~\ref{sec:canonical-high-cells}--\ref{sec:residual-attachments}),
and $\ell=(\ell_i)$ is a selected common subprofile, not $\log\log n$, which
we write out. The profile $\mathbf k$ is that of
Proposition~\ref{prop:selected-profile-contract}, so
$k_i/k_{\mathrm{co}}\ge c_*$ and $\alpha-n/k_{\mathrm{co}}\le39/10$. All
thresholds and error sequences are deterministic and independent of the
phase.

\subsection{Recurrences for the common-class weights}
\label{app:common-recurrences}

The weights $D(\ell)$ of (7.3) satisfy two exact one-class recurrences.
Whenever $\ell_i<k_i$, one may add a common block of type $i$.  Let $e_i$
denote the $i$th coordinate vector.  Since
$n-m=\sum_j u_j(k_j-\ell_j)\ge u_i$, the denominator below is nonzero.
\Needspace{6\baselineskip}
Moreover,

\[
 \frac{\binom{k_i}{\ell_i+1}^2}{\binom{k_i}{\ell_i}^2}
 =\frac{(k_i-\ell_i)^2}{(\ell_i+1)^2},
\]

while (7.2) gives

\[
 \frac{Y_{\ell+e_i}^{\mathrm{sgn}}}
      {Y_\ell^{\mathrm{sgn}}}
 =\frac{2(n-m)_{u_i}}
        {u_i!(\ell_i+1)2^{\binom{u_i}{2}}}
 =\frac{2\mu_{u_i}(n-m)}{\ell_i+1}.
\]

Dividing these two ratios gives

\begin{equation}
 \frac{D(\ell+e_i)}{D(\ell)}
 =\frac{(k_i-\ell_i)^2}
        {2(\ell_i+1)\mu_{u_i}(n-m)}.                     \label{eq:common-forward-recurrence}
\end{equation}

At the opposite endpoint, put $h=\mathbf k-\ell$,
$h_\bullet=\sum_i h_i$, $v=\sum_i u_i h_i$, and
$b_h=\sum_i\binom{u_i}{2}h_i$. Define the finite first-moment factor
\[
 M_h(v):=
 2^{h_\bullet}\frac{v!}{\prod_i(u_i!)^{h_i}h_i!}\,2^{-b_h}.
\]
Exact factorial cancellation gives

\begin{equation}
 D(\mathbf k-h)=\frac{B(h)}{\Exp{Z_{\mathbf k}^{\mathrm{sgn}}}},\qquad
 B(h)=\left(\prod_i\binom{k_i}{h_i}\right)M_h(v), \label{eq:common-complementary-identity}
\end{equation}

and, whenever \(h_i<k_i\),

\[
 \frac{\binom{k_i}{h_i+1}}{\binom{k_i}{h_i}}
 =\frac{k_i-h_i}{h_i+1},\qquad
 \frac{M_{h+e_i}(v+u_i)}
      {M_h(v)}
 =\frac{2\mu_{u_i}(v+u_i)}{h_i+1}.
\]

Therefore

\begin{equation}
\frac{B(h+e_i)}{B(h)}
 =\frac{2(k_i-h_i)\mu_{u_i}(v+u_i)}{(h_i+1)^2}.          \label{eq:common-reverse-recurrence}
\end{equation}

\subsection{A uniform rate for five deficit sizes}
\label{app:five-size-rate}

The intermediate range rests on the following elementary inequality; in its
application below, $R$ is the residual class fraction, $T$ the mean deficit
and $D=I_r-TR$.

\Needspace{6\baselineskip}

\begin{lemma}[A uniform rate for five deficit sizes]\label{lem:five-size-rate}
Let $1/64\le R\le1$, $T\le39/10$, and
$D\le4R(1-R)$, $D\le(T-2)(1-R)$. Then
\[
 R\log R+\frac{q}{2}D\le-\frac{1-R}{10000}.
\]
\end{lemma}
\begin{proof}
Put $a=12/25$ and $L_0=111176/151959$. With $t=13/37$,
\[
 -\log a=2\int_0^t\frac{du}{1-u^2}
 >2(t+t^3/3)=L_0.
\]
For $R\le a$, concavity of $\log$ on the chord from $R$ to $1$ gives
$\log R\le(1-R)\log a/(1-a)$. Using $q<7/10$ and $R\ge1/64$,
\[
 R\log R+(q/2)D
 \le-R(1-R)\left(\frac{L_0}{1-a}-\frac75\right)
 \le-\frac{5287}{48626880}(1-R).
\]
For $R\ge a$, convexity of $x\log x$ on $[a,1]$ gives
$R\log R\le(1-R)a\log a/(1-a)$. Therefore
\[
 R\log R+(q/2)D
 \le-(1-R)\left(\frac{aL_0}{1-a}-\frac{133}{200}\right)
 =-\frac{104751}{10130600}(1-R).
\]
Both rational constants exceed $1/10000$. These arguments include $R=1$,
where the desired bound is zero.
\end{proof}

\subsection{Three-range estimates for common subprofiles}
\label{app:common-subprofile-details}

\begin{proof}[Proof of Lemma~\ref{lemma-7.1-all-common-subprofiles}]
We use the original deficit indexing of Section~7 throughout this proof.
Set
\[
  \eta:=\frac{\log\log n}{32\log n}.
\]
For all sufficiently large $n$ we have $\eta<31/32$. Split the selected
mass into three disjoint ranges: $m\le\eta n$; $m>\eta n$ with
$n-m>n/32$; and $n-m\le n/32$. The empty subprofile in the first range
contributes one. We show that every other contribution vanishes in total,
using recurrence \eqref{eq:common-forward-recurrence} for the first range, a rate estimate for the second,
and reverse recurrence \eqref{eq:common-reverse-recurrence} for the third.

\displayheading{Small selected mass}

Let
\begin{equation}
  \xi_i=\frac{k_i^2}{2\mu_{u_i}(n)},
  \qquad
  \Xi_{\mathrm{empty}}=\sum_i\xi_i.
  \label{eq:G-xi}
\end{equation}
For $2\le i\le6$ one has $u_i\le\alpha-2$, and \eqref{eq:D-adjacent} with
the estimate after it gives $\mu_{s-1}>\mu_s$ for $|s-\alpha|\le6$ and large
$n$; hence $\mu_{u_i}(n)\ge\mu_{\alpha-2}$. With (2.4) and
$k_i=\Theta(n/\log n)$ there is a phase-independent $C_\xi$ such that
\begin{equation}
  \Xi_{\mathrm{empty}}\le C_\xi(\log n)^{-(2/\log 2-1/2)}.
  \label{eq:G-empty-bound}
\end{equation}
For $m\le\eta n$, we compare every denominator in the recurrence with
the corresponding first moment on the full ambient vertex set.

Fix a final subprofile $\ell$ with selected mass $m\le\eta n$, and expose its
blocks in any order. At an intermediate selected mass $m'\le m$, the exact
ratio of first moments is
\[
  \frac{\mu_{u_i}(n)}{\mu_{u_i}(n-m')}
  =\frac{(n)_{u_i}}{(n-m')_{u_i}}
  =\prod_{r=0}^{u_i-1}
    \frac{n-r}{n-m'-r}.
\]
Uniformly in the phase, $u_i=O(\log n)$ whereas
$\eta n\gg\log n$. Hence, for all sufficiently large $n$,
$m'+r\le2\eta n$ throughout this product, and therefore
\[
  \frac{\mu_{u_i}(n)}{\mu_{u_i}(n-m')}
  \le(1-2\eta)^{-u_i}.
\]
The phase expansion and $\log 2>2/3$ give $u_i\le3\log n$ eventually. Since
$\eta\le1/4$ and $-\log(1-2\eta)\le4\eta$,
\begin{equation}
  (1-2\eta)^{-u_i}
  \le\exp(4\eta u_i)
  \le\exp\!\left(\frac38\log\log n\right)
  =(\log n)^{3/8}.
  \label{eq:G-ratio}
\end{equation}
Every threshold in this estimate is independent of the phase.

Suppose the current partial profile is $a$ and the next selected block has type
$i$. The exact recurrence \eqref{eq:common-forward-recurrence}, the inequality
$(k_i-a_i)^2\le k_i^2$, and \eqref{eq:G-ratio} give
\[
  \frac{D(a+e_i)}{D(a)}
  \le
  \frac{(\log n)^{3/8}\xi_i}{a_i+1}.
\]
Multiplying these one-block bounds gives the same estimate for every
exposure order:
\[
  D(\ell)
  \le
  \prod_i
  \frac{\bigl((\log n)^{3/8}\xi_i\bigr)^{\ell_i}}
       {\ell_i!}.
\]
All summands are nonnegative. We may therefore sum over every nonnegative
integer choice of the five type counts, rather than only the selected-mass
range. Expanding the resulting exponential series yields
\begin{equation}
  1
  \le
  \sum_{m\le\eta n}D(\ell)
  \le
  \exp\!\left((\log n)^{3/8}\Xi_{\mathrm{empty}}\right).
  \label{eq:G-small-mass}
\end{equation}
The lower bound is the empty marked subprofile.

The exponent tends to zero with a fixed power margin.  By
\eqref{eq:G-empty-bound},
\[
  (\log n)^{3/8}\Xi_{\mathrm{empty}}
  =O\!\left((\log n)^{-(2/q-7/8)}\right),
\]
and $q<7/10$ gives
\[
  \frac2q-\frac78
  >\frac{20}{7}-\frac78
  =\frac{111}{56}>0.
\]
Thus, with the phase-independent
$\varepsilon_n^{\mathrm{empty}}:=C_\xi(\log n)^{-(2/q-7/8)}\to0$,
\[
  1
  \le
  \sum_{m\le\eta n}D(\ell)
  \le e^{\varepsilon_n^{\mathrm{empty}}}
  =1+o(1).
\]

\displayheading{Intermediate selected mass}

Use the normalized variables $p_i,y_i,z_i,Y,R,T,I_r$ of
Section~\ref{exact-partial-diagonals} and put $I:=\sum_i iy_i$, so that
$I_r=T-I$ and $T=\alpha-n/k_{\mathrm{co}}$.
The residual vertex fraction is
\begin{equation}
 \rho=\frac{n-m}{n}
 =R+\frac{I-TY}{\alpha-T}.
 \label{eq:G-rho}
\end{equation}

\displayheading{A uniform estimate for the marked weights}
Before approximating any factorial, equations (7.2) and (7.3) give the exact
identity
\[
\begin{split}
\log D(\ell)={}&
2\sum_i\{\log(k_i!)-\log(\ell_i!)-\log((k_i-\ell_i)!)\}\\
&-\ell_\bullet\log 2-\log(n!)+\log((n-m)!)
 +\sum_i\ell_i\log(u_i!)+\sum_i\log(\ell_i!)
 +(\log 2)b_\ell.
\end{split}
\]
Some selected or remaining type counts may vanish. We therefore use
Stirling's estimate in a form valid also at zero: for an absolute constant
$C_{\mathrm S}>0$,
\[
  \log(t!)=t\log t-t+\sigma(t),
  \qquad
  |\sigma(t)|\le C_{\mathrm S}\log(t+2)
  \qquad(t\in\mathbb N_0),
\]
with the convention $0\log0=0$. Apply this estimate to the factorials that
count vertices and classes: \(n!\), \((n-m)!\), \(k_i!\), \(\ell_i!\), and
\((k_i-\ell_i)!\). Keep the class-size factorials \(u_i!\) exact, because
their coefficients depend on the subprofile. We approximate only a fixed
number of factorials, each with argument at most $n$. Their total remainder
is therefore $O(\!\log n)$ uniformly, even when some $\ell_i$ or
$k_i-\ell_i$ is zero.

Set $K=k_{\mathrm{co}}$, so $k_i=Kp_i$, $\ell_i=Ky_i$, and
$k_i-\ell_i=Kz_i$. The profile-multiplicity factorials then contribute
\[
\begin{split}
 &2\sum_i\log(k_i!)-\sum_i\log(\ell_i!)
      -2\sum_i\log((k_i-\ell_i)!)\\
 &\quad={}
 K\!\left[
  2\sum_i p_i\log p_i-\sum_i y_i\log y_i
  -2\sum_i z_i\log z_i-Y+Y\log K
 \right]+O(\!\log n).
\end{split}
\]
Since $m=K(\alpha Y-I)=K\sum_i u_i y_i$, $n=K(\alpha-T)$, and
$n-m=n\rho$, the two ambient factorials contribute
\[
 -\log(n!)+\log((n-m)!)
 =n\rho\log\rho
   +K\sum_i y_i\bigl(u_i-u_i\log n\bigr)+O(\!\log n).
\]
Finally, $\ell_\bullet=K\sum_i y_i$ and
$b_\ell=K\sum_i y_i\binom{u_i}{2}$, so the class-size and sign terms are
exactly
\[
 K\sum_i y_i\left[
   \log(u_i!)+(\log 2)\binom{u_i}{2}-\log 2
 \right].
\]
Adding these three displayed contributions and distributing
$Y\log K=\sum_i y_i\log K$ gives
\begin{equation}
\begin{split}
 \log D(\ell)={}&n\rho\log\rho\\
 &+k_{\mathrm{co}}\sum_i\{2p_i\log p_i
     -2(p_i-y_i)\log(p_i-y_i)
     -y_i\log y_i-y_i+y_iE_i\}\\
 &+O(\!\log n),
\end{split}
 \label{eq:G-logD}
\end{equation}
where
\begin{equation}
 E_i=\log k_{\mathrm{co}}+\log(u_i!)+u_i-u_i\log n
       +(\log 2)\binom{u_i}{2}-\log 2.
 \label{eq:G-Ei}
\end{equation}

\Needspace{4\baselineskip}
The entropy bound must remain valid when a type is entirely selected or
entirely unselected. By
Proposition~\ref{prop:selected-profile-contract}\textup{(i)},
$p_i\ge c_*>0$. There is therefore an absolute $C_{\mathrm H}$ such that,
for $c_*\le p\le1$ and $0\le y\le p$,
\begin{equation}
  2p\log p-2(p-y)\log(p-y)-y\log y-y
  \le C_{\mathrm H}y\log\!\left(\frac{\mathrm e}{y}\right).
  \label{eq:G-entropy}
\end{equation}
Here and below, \(y\log(\mathrm e/y)\) takes its continuous value \(0\) at
\(y=0\); the analogous convention applies to \(Y\log(5\mathrm e/Y)\).
Indeed, when $y\le p/2$, the mean-value theorem bounds
$p\log p-(p-y)\log(p-y)$ by $C_{\mathrm H}y$, while $-y\log y$ supplies the only
singular term. When $y\ge p/2$, one has $y\ge c_*/2$, and the assertion follows
by compactness. Finally,
\[
  \sum_i y_i\log\!\left(\frac{\mathrm e}{y_i}\right)
  \le
  Y\log\!\left(\frac{5\mathrm e}{Y}\right),
\]
by the entropy bound on five coordinates. The resulting sum is
$O(Y\log(\mathrm e/Y))$ uniformly.

Since $u_{i+1}=u_i-1$, exact subtraction in \eqref{eq:G-Ei} gives
\begin{equation}
\begin{aligned}
 E_{i+1}-E_i
 &=\log n-\log(\alpha-i)-(\log 2)(\alpha-i)+\log 2-1\\
 &=-\frac{\log 2}{2}\alpha+O(1),
\end{aligned}
 \label{eq:G-Ediff}
\end{equation}
uniformly for $2\le i\le5$. For the second line, use
$(\log 2)\alpha/2=\log n-\log\log n+O(1)$ and
$\log(\alpha-i)=\log\log n+O(1)$.

The type coefficients vary strongly, but we only need an upper bound for
their average over the full profile. The complete first moment gives this
bound. Let $\bar E=\sum_i p_iE_i$. Taking $\ell=\mathbf k$ in
\eqref{eq:G-logD}, so that $\rho=0$, $z=0$ and
$D(\mathbf k)=1/\Exp{Z_{\mathbf k}^{\mathrm{sgn}}}$, and dividing by
$k_{\mathrm{co}}$ (the error $O(\log n)/k_{\mathrm{co}}$ is $o(1)$) gives
\begin{equation}
 -\frac1{k_{\mathrm{co}}}
   \log \Exp{Z_{\mathbf k}^{\mathrm{sgn}}}
 =\sum_i p_i\log p_i-1+\bar E+o(1).
 \label{eq:G-mean-E}
\end{equation}
The phase-uniform growth estimate (5.19) is eventually positive.  Solving
\eqref{eq:G-mean-E} for $\bar E$ and using
$-\sum_i p_i\log p_i\le\log5$ therefore gives
\[
  \bar E
  =1-\sum_i p_i\log p_i
    -\frac1{k_{\mathrm{co}}}\log\mathbb E Z_{\mathbf k}^{\mathrm{sgn}}
    +o(1)
  \le1+\log5+o(1).
\]
In particular, $\bar E\le C_E$ for one phase-independent constant. To expose
the affine structure in \eqref{eq:G-Ediff}, set
\[
  F_i:=E_i+\frac{(\log 2)\alpha}{2}i.
\]
Then $F_{i+1}-F_i=O(1)$. Because the support has only five consecutive
indices,
\[
  F_i=\sum_jp_jF_j+O(1)
     =\bar E+\frac{(\log 2)\alpha}{2}T+O(1).
\]
Consequently
\[
  E_i=\bar E-\frac{(\log 2)\alpha}{2}(i-T)+O(1),
\]
and therefore
\begin{equation}
 \sum_i y_iE_i
 \le\frac{(\log 2)\alpha}{2}(TY-I)+O(Y).
 \label{eq:G-sumyE}
\end{equation}
Only the upper bound on $\bar E$ is used here.

The two bounds are expressed in different variables: the ambient term uses
the remaining vertex fraction $\rho$, while the coefficient bound uses the
remaining class fraction $R$. We now compare these fractions. Since
$2\le i\le6$ and $2\le T\le6$,
\[
  |I-TY|=\left|\sum_i(i-T)y_i\right|\le4Y,
\]
so \eqref{eq:G-rho} gives
\[
  |\rho-R|\le\frac{4Y}{\alpha-T}=O(Y/\alpha).
\]
If $\rho\ge1/32$, then $Y\le1$ and the preceding bound implies
$R\ge1/64$ for all sufficiently large $n$. Hence $x\mapsto x\log x$ is
uniformly Lipschitz on the interval between $R$ and $\rho$. Moreover,
$-R\log R\le1-R=Y$. Choose a fixed Lipschitz constant $C_{\mathrm{lip}}>0$ for
$x\mapsto x\log x$ on $[1/64,1]$. Then
\[
\begin{aligned}
&\left|n\rho\log\rho
      -k_{\mathrm{co}}\alpha R\log R\right|\\
&\qquad\le
 k_{\mathrm{co}}(\alpha-T)C_{\mathrm{lip}}|\rho-R|
 +k_{\mathrm{co}}T|R\log R|
 \le (4C_{\mathrm{lip}}+6)k_{\mathrm{co}}Y,
\end{aligned}
\]
which proves
\begin{equation}
 n\rho\log\rho
 =k_{\mathrm{co}}\alpha R\log R+O(k_{\mathrm{co}}Y).
 \label{eq:G-rho-R}
\end{equation}
Under the standing condition $\rho\ge1/32$, combining \eqref{eq:G-logD},
\eqref{eq:G-entropy}, \eqref{eq:G-sumyE}, and \eqref{eq:G-rho-R} yields one
absolute constant $C_{\mathrm{diag}}$ such that
\begin{equation}
 \log D(\ell)
 \le k_{\mathrm{co}}\alpha\Phi_T(z)
     +C_{\mathrm{diag}}k_{\mathrm{co}}Y
       \log\!\left(\frac{\mathrm e}{Y}\right)
     +C_{\mathrm{diag}}\log n,
 \label{eq:G-logD-bound}
\end{equation}
where
\begin{equation}
 \Phi_T(z)=R\log R+\frac{\log 2}{2}(I_r-TR).
 \label{eq:G-PhiT}
\end{equation}
All constants and thresholds in this bound are independent of the phase.

\displayheading{A negative bound for the leading rate}
For $\rho\ge1/32$ we have $R\ge1/64$ eventually, and
Proposition~\ref{prop:selected-profile-contract}\textup{(i)} gives
$T\le39/10$. The two inequalities for $I_r-TR$ displayed after
Lemma~\ref{lemma-7.1-all-common-subprofiles} and
Lemma~\ref{lem:five-size-rate} therefore give
\begin{equation}
 \Phi_T(z)\le-\frac{Y}{10000}.
 \label{eq:G-negative-rate}
\end{equation}

\Needspace{6\baselineskip}
\displayheading{Summing the intermediate profiles}
Assume
\begin{equation}
  m>\eta n,
  \qquad
  n-m>n/32.
  \label{eq:G-intermediate}
\end{equation}
Then $\rho>1/32$. Moreover,
\[
  \frac mn=\frac{\alpha Y-I}{\alpha-T}
  \le\frac{\alpha Y}{\alpha-T}
\]
and $m/n>\eta$ imply that, for all sufficiently large $n$,
\[
  Y>\eta\frac{\alpha-T}{\alpha}\ge\frac\eta2.
\]
Define the deterministic error ratio
\begin{equation}
  \varepsilon_n^{\mathrm{diag}}
  :=
  \frac{C_{\mathrm{diag}}\log(2\mathrm e/\eta)}{\alpha}
  +
  \frac{2C_{\mathrm{diag}}\log n}{k_{\mathrm{co}}\alpha\eta}.
  \label{eq:G-diag-error}
\end{equation}
Since
\[
  \alpha=\Theta(\log n),
  \qquad
  k_{\mathrm{co}}=\Theta(n/\log n),
  \qquad
  \eta=\Theta(\log\log n/\log n),
\]
one has $\varepsilon_n^{\mathrm{diag}}\to0$, uniformly in the phase. Because
$Y\ge\eta/2$, equations \eqref{eq:G-logD-bound} and
\eqref{eq:G-negative-rate} give
\[
  \log D(\ell)
  \le
  -k_{\mathrm{co}}\alpha Y
   \left(\frac1{10000}-\varepsilon_n^{\mathrm{diag}}\right).
\]
For all sufficiently large $n$, the parenthesis is at least $1/20000$.
Moreover, $\alpha\eta=\Theta(\log\log n)$. Hence there is a phase-independent
$c>0$ such that
\begin{equation}
  D(\ell)
  \le
  \exp\{-c k_{\mathrm{co}}\log\log n\}.
  \label{eq:G-intermediate-bound}
\end{equation}
Fix phase-independent constants $0<c_K<C_K$ such that, eventually,
\[
 c_K\frac n{\log n}\le k_{\mathrm{co}}
 \le C_K\frac n{\log n}.
\]
Define the phase-independent envelope
\[
  \varepsilon_n^{\mathrm{central}}
  :=
  \left(C_K\frac n{\log n}+1\right)^5
  \exp\!\left\{-c c_K\frac n{\log n}\log\log n\right\}.
\]
There are at most $(k_{\mathrm{co}}+1)^5$ subprofiles, so the total
central-range contribution is at most
$\varepsilon_n^{\mathrm{central}}$. Since
$n\log\log n/\log n\gg\log n$, this deterministic sequence tends to zero.

\displayheading{Nearly full selected mass}

Put $h=\mathbf k-\ell$ and write
\[
  v(h):=\sum_i u_i h_i=n-m.
\]
The nearly full range, also called the full corner, is
\begin{equation}
  v(h)\le\frac n{32}.
  \label{eq:G-full-corner}
\end{equation}
In this range we use the exact reverse recurrence \eqref{eq:common-reverse-recurrence}, keeping the
ambient falling factorials unchanged.

The remaining vertex set has small first moments for each of the five
block sizes. For every integer $0\le w\le n$, the ratio below is zero when
$w<u_i$; otherwise
\[
  \frac{\mu_{u_i}(w)}{\mu_{u_i}(n)}
  =\frac{(w)_{u_i}}{(n)_{u_i}}
  =\prod_{r=0}^{u_i-1}\frac{w-r}{n-r}
  \le\left(\frac wn\right)^{u_i}.
\]
Equations (2.2) and \eqref{eq:D-adjacent} give
$\mu_{u_i}(n)\le n^{7+o(1)}$ uniformly in the phase. Moreover,
$u_i=(2/\log 2+o(1))\log n$, and
$\log 32=5\log 2$. Consequently, uniformly for $w\le n/32$,
\begin{equation}
  \mu_{u_i}(w)
  \le
  n^{7+o(1)}32^{-u_i}
  =n^{-3+o(1)}.
  \label{eq:G-mu-small}
\end{equation}
In particular, there is one phase-independent threshold after which
\begin{equation}
  \mu_{u_i}(w)\le n^{-5/2}
  \qquad
  (w\le n/32,\ 2\le i\le6).
  \label{eq:G-mu-small-unif}
\end{equation}
We use the weaker exponent $-5/2$ so that the next recurrence bound has a
single threshold valid for every phase.

Fix a final residual profile $h$ satisfying \eqref{eq:G-full-corner}, and
expose its blocks in any order. Let $a$ be the current residual profile and
$v(a)=\sum_i u_i a_i$. If the next block has type $i$, then
$a_i<h_i$ and therefore
\[
  v(a)+u_i\le v(h)\le n/32.
\]
The exact recurrence \eqref{eq:common-reverse-recurrence}, $k_i-a_i\le k_i\le n$, and
\eqref{eq:G-mu-small-unif} now give
\[
  \frac{B(a+e_i)}{B(a)}
  =
  \frac{2(k_i-a_i)\mu_{u_i}(v(a)+u_i)}{(a_i+1)^2}
  \le2n^{-3/2}<1
\]
for all sufficiently large $n$.  By convention $M_{\mathbf0}(0)=1$, and hence
$B(0)=1$.  Thus $B$ never increases along the path from $0$ to $h$, and
\begin{equation}
  B(h)\le B(0)=1.
  \label{eq:G-B}
\end{equation}
By the exact complementary identity \eqref{eq:common-complementary-identity} and the phase-uniform first-moment
growth estimate (5.19),
\begin{equation}
  D(\mathbf k-h)
  =\frac{B(h)}{\Exp{Z_{\mathbf k}^{\mathrm{sgn}}}}
  \le \exp\{-c_Z(\log n)^2\}.
  \label{eq:G-full-bound}
\end{equation}
There are at most $(k_{\mathrm{co}}+1)^5$ residual profiles, so
\[
  \sum_{v(h)\le n/32}D(\mathbf k-h)\le\varepsilon_n^{\mathrm{full}}
  :=\Bigl(C_K\frac n{\log n}+1\Bigr)^5e^{-c_Z(\log n)^2}\longrightarrow0.
\]

\displayheading{Summing the three ranges}
Since $\eta<31/32$, every $\ell$ with $n-m\le n/32$ has $m>\eta n$; hence
the three ranges at the start of the proof partition the box
$0\le\ell_i\le k_i$, and the three bounds give
$1\le\sum_{\ell}D(\ell)\le e^{\varepsilon_n^{\mathrm{empty}}}
+\varepsilon_n^{\mathrm{central}}+\varepsilon_n^{\mathrm{full}}=1+o(1)$
with deterministic phase-independent errors. This proves
\eqref{eq:common-subprofile-quantitative}, and hence (7.7).
 \end{proof}

\section{Large overlap cells}
\label{app:large-cells}

This appendix proves Proposition~\ref{prop:fixed-support-all-deficit} and
Lemma~\ref{lem:endpoint-table-product-comparison}, the two estimates of
Section~\ref{sec:canonical-high-cells} whose proofs are routine
calculations. We keep the notation of that section: $U=\alpha-2$,
$R_0=\lfloor U/2\rfloor$, the canonical matching $M$ of large cells, the
weights $w(M,j)$ and $w_{\mathrm{full}}(M)$, and, for $e\in M$, the
full-containment multiplicity $m_e$, the endpoint-size difference $d_e$ and
the deficit $h_e$, which satisfy $2h_e<m_e$ by \eqref{eq:half-deficit}.
Types are indexed by $i\in\{0,\ldots,4\}$, with shifted sizes
$u_i=\alpha-2-i$ and $\kappa_i=k_{i+2}$ classes, and $\alpha>8$, which holds
beyond a phase-independent threshold.

\subsection{Deficit ratios and the ambient denominator}
\label{app:deficit-ratios}

For endpoint sizes $m$ and $m+d$, define
\begin{equation}
  R_{m,d}(h)
  :=
  \frac{\binom mh}
       {(d+1)(d+2)\cdots(d+h)}
  2^{-hm+h(h+1)/2},
  \qquad
  R_{m,d}(0):=1.
  \label{eq:local-deficit-ratio}
\end{equation}
The product in the denominator is empty when $h=0$.

\begin{lemma}[Exact one-cell deficit ratio]
\label{lem:exact-local-ratio}
For endpoint sizes \(m\) and \(m+d\), define the local numerator
\[
 f_{m,d}(j):=\frac{(m)_j(m+d)_j}{j!}\,g(j).
\]
For integers \(m\ge1\), \(d\ge0\), and \(h\ge0\) with \(2h<m\),
\[
 \frac{f_{m,d}(m-h)}{f_{m,d}(m)}=R_{m,d}(h).
\]
\end{lemma}

\begin{proof}
The ratio of the partial matching counts is
\[
  \frac{(m)_{m-h}(m+d)_{m-h}}{(m-h)!}
  \frac{m!}{(m)_m(m+d)_m}
  =
  \frac{\binom mh}{(d+1)(d+2)\cdots(d+h)}.
\]
The local factors satisfy
\[
  \frac{g(m-h)}{g(m)}
  =
  2^{-hm+h(h+1)/2}.
\]
Multiplying the two identities proves the claim.
\end{proof}

Recall $\mathbf m$, $w_{\mathrm{full}}(M)$, $J_*$ and $h_\bullet$ from
Section~\ref{subsec:deficit-sum}.
After the one-cell ratios have been extracted, the only factor that still
couples different cells is the ambient denominator:
\begin{equation}
  \frac{(n)_{J_*}}{(n)_J}
  =(n-J)_{h_\bullet}
  \le n^{h_\bullet}.
  \label{eq:one-global-falling-factorial}
\end{equation}

\begin{lemma}[Aggregate deficit comparison]
\label{lem:aggregate-deficit-comparison}
For every $h=(h_e)_{e\in M}\in\mathbb N_0^M$ satisfying
$2h_e<m_e$ for each $e\in M$,
\begin{equation}
  w(M,\mathbf m-h)
  \le
  w_{\mathrm{full}}(M)
  \prod_{e\in M}n^{h_e}R_{m_e,d_e}(h_e).
  \label{eq:aggregate-deficit-comparison}
\end{equation}
\end{lemma}

\begin{proof}
Apply Lemma~\ref{lem:exact-local-ratio} in each selected cell, apply
\eqref{eq:one-global-falling-factorial} once, and use
$n^{h_\bullet}=\prod_{e\in M}n^{h_e}$.
\end{proof}

\subsection{Summing all positive deficits}
\label{app:deficit-sum}

For a selected cell $e$, enlarge the exact positive-deficit range to
\[
  A_e:=\{h\in\mathbb N:h\ge1,\ 2h<m_e\}.
\]
This range may include deficits that no longer give a large cell. Adding
their nonnegative weights only increases the upper bound.

For any nonnegative $\varphi_e(h)$, set $\varphi_e(0):=1$. Expanding the
finite product gives the optional-choice identity
\[
  \sum_{\omega\in\prod_{e\in M}(\{0\}\cup A_e)}
    \prod_{e\in M}\varphi_e(\omega_e)
  =
  \prod_{e\in M}
    \left(1+\sum_{h\in A_e}\varphi_e(h)\right),
\]
where each $\omega_e$ is either the zero-deficit choice, of weight one, or one
positive deficit $h\in A_e$.

For $2h<m$, integrality gives $2h+1\le m$. Hence
\[
  \frac{h+1}{2}
  \le
  \frac{m+1}{4}
  \le
  m-\left\lfloor\frac{3m-1}{4}\right\rfloor.
\]
Multiplying by $h$ yields
\begin{equation}
  h\left\lfloor\frac{3m-1}{4}\right\rfloor
  \le
  hm-\frac{h(h+1)}2.
  \label{eq:three-quarter-budget}
\end{equation}
Together with $\binom mh\le m^h$ and
\eqref{eq:local-deficit-ratio}, this gives
\begin{equation}
  n^hR_{m,d}(h)
  \le
  \left(
    \frac{nm}{2^{\lfloor(3m-1)/4\rfloor}}
  \right)^h.
  \label{eq:three-quarter-charge}
\end{equation}

With the types $u_i$, $\kappa_i$ and the quantities $\rho_{ij}$,
$\rho_\Sigma$ of Section~\ref{subsec:deficit-sum}, every local base in
\eqref{eq:three-quarter-charge} is at most $\rho_{\Sigma}$.

\begin{lemma}[Coarse local deficit sum]
\label{lem:coarse-local-deficit-sum}
Suppose $\rho_{\Sigma}\le1$. Then, for every selected cell $e$,
\[
  \sum_{h\in A_e}n^hR_{m_e,d_e}(h)
  \le
  (\alpha+1)\rho_{\Sigma}.
\]
\end{lemma}

\begin{proof}
If $e$ has endpoint type $(i,j)$, then
\eqref{eq:three-quarter-charge} bounds its term of deficit $h$ by
$\rho_{ij}^h$. Since $h\ge1$ and
$\rho_{ij}\le\rho_{\Sigma}\le1$, this is at most $\rho_{\Sigma}$. Moreover,
$|A_e|\le m_e\le\alpha+1$. Summation proves the bound.
\end{proof}

\begin{proof}[Proof of Proposition~\ref{prop:fixed-support-all-deficit}]
By \eqref{eq:half-deficit}, every admissible $h_e$ lies in $\{0\}\cup A_e$.
Use Lemma~\ref{lem:aggregate-deficit-comparison}, enlarge the deficit range
to $A_e$, apply the optional-choice identity above, and then use
Lemma~\ref{lem:coarse-local-deficit-sum}.
\end{proof}

\subsection{Proof of the endpoint-table product comparison}
\label{app:endpoint-comparison-proof}

\begin{proof}[Proof of Lemma~\ref{lem:endpoint-table-product-comparison}]
Fix the table and divide $W(L)^2$ by $(D(r)A_L)(D(c)C_L)$. The profile
multiplicities and table factorials cancel, leaving one ratio for each
endpoint pair and one ratio of falling factorials from the matching
probabilities. We compute the cell ratios first, then bound the remaining
matching ratio.

Fix one endpoint pair and abbreviate $m=m_{ij}$ and $d=d_{ij}$, so its
endpoint sizes are $m$ and $m+d$. Its local atom in
\eqref{eq:endpoint-reference-weight} is
\[
  \frac{(m)_m(m+d)_m}{m!}g(m)
  =
  \frac{(m+d)!}{d!}g(m).
\]
Consequently
\[
  \frac{
    \left[\frac{(m)_m(m+d)_m}{m!}g(m)\right]^2
  }{m!g(m)\,(m+d)!g(m+d)}
  =
  \frac{(m+d)_d}{(d!)^2}
  2^{-dm-\binom d2}
  =
  \frac{Q_{ij}^2}{(n+1)^d}.
\]
Under the hypothesis, $m\ge\alpha-6\ge3$, so in particular $m\ge2$. Hence we
may use $g(m+d)/g(m)=2^{dm+\binom d2}$; the identity is also valid for
$d=0$.

Moreover, $m(r),m(c)\ge J(L)$ and
\[
  m(r)+m(c)-2J(L)
  =
  \sum_{i,j}d_{ij}\ell_{ij}.
\]
Therefore, since $(n)_x/(n)_y=(n-y)_{x-y}\le n^{x-y}$ for integers
$0\le y\le x\le n$,
\[
  \frac{(n)_{m(r)}(n)_{m(c)}}{(n)_{J(L)}^2}
  \le
  (n+1)^{m(r)+m(c)-2J(L)}.
\]
Combining this inequality with the local ratio cell by cell gives
\[
\begin{aligned}
 \frac{W(L)^2}{(D(r)A_L)(D(c)C_L)}
 &\le
 (n+1)^{m(r)+m(c)-2J(L)}
 \prod_{i,j}
 \left(\frac{Q_{ij}^2}{(n+1)^{d_{ij}}}\right)^{\ell_{ij}}\\
 &=\prod_{i,j}Q_{ij}^{2\ell_{ij}}
  =(Q^L)^2,
\end{aligned}
\]
where the equality uses
\(m(r)+m(c)-2J(L)=\sum_{i,j}d_{ij}\ell_{ij}\). This proves
\eqref{eq:endpoint-table-product-comparison}.
\end{proof}

\section{Residual factors}
\label{app:residual-factors}

This appendix proves the exact decomposition of
Proposition~\ref{prop:canonical-exact-overlap-decomposition} and the
quadratic activity bound of Lemma~\ref{lem:q-quadratic}. We keep the
notation of Sections~\ref{exact-signed-second-moment-representation},
\ref{sec:canonical-high-cells} and~\ref{sec:residual-attachments}: the
shifted type sizes $u_i=\alpha-2-i$, the cap $U=\alpha-2\ge2$,
$R_0=\lfloor U/2\rfloor$, the uniform configuration matching $\Pi$, the
feasible pairs $(M,j)$ with $J=\sum_{e\in M}j_e$, the residual mass
$m_0=n-J$ and the event $\mathcal E(M,j)$. Lemma~\ref{lem:q-quadratic}
concerns residual cells outside $M$ when $m_0>0$ and $2^U\le m_0^3$, as in
\eqref{eq:intrinsic-residual-regime}.

\subsection{Proof of the exact overlap decomposition}
\label{app:decomposition-proof}

\begin{proof}[Proof of Proposition~\ref{prop:canonical-exact-overlap-decomposition}]
Let $\Pi$ be a full bipartite configuration matching and let $r_{ab}(\Pi)$ be
its cell counts. Define
\[
  M(\Pi):=\{(a,b):r_{ab}(\Pi)>R_0\},
  \qquad
  j_{ab}(\Pi):=r_{ab}(\Pi)\quad((a,b)\in M(\Pi)).
\]
Since $2(R_0+1)>U$, the degree cap makes $M(\Pi)$ a matching. Moreover,
$j_{ab}>\lfloor U/2\rfloor$ and $U\ge2$ imply $j_{ab}\ge2$, so every cell of
$M(\Pi)$ is an edge of the support graph from Section~6. Let $S(\Pi)$
be the set of all stub pairs of $\Pi$ lying in cells of $M(\Pi)$, and put
$\Pi_{\mathrm{res}}:=\Pi\setminus S(\Pi)$. Then
\[
  \Pi\longmapsto
  \bigl(M(\Pi),j(\Pi),S(\Pi),\Pi_{\mathrm{res}}\bigr)
\]
takes values in the feasible pairs $(M,j)$, their labeled realizations, and residual
matchings satisfying $\mathcal E(M,j)$. It is injective because
$\Pi=S(\Pi)\cup\Pi_{\mathrm{res}}$.

Conversely, take a feasible $(M,j)$, an $S\in\mathfrak F(M,j)$, and a perfect
matching of the unused stubs satisfying $\mathcal E(M,j)$. Their disjoint union
is a full matching. The condition $r'_{ab}=0$ on $M$ leaves multiplicity $j_e>R_0$ in
each $e\in M$, and the residual cap leaves every cell outside $M$ with
multiplicity at most $R_0$. Thus its canonical large-cell data are exactly
$(M,j)$, and this construction is the inverse of the displayed map.

By (6.6) and the configuration-model form of (6.2),
$\mathbb E Z^2/(\mathbb E Z)^2=\mathbb E[A_\zeta(r(\Pi))]$. Under the
bijection, $r(\Pi)=j_M+r'$, where $r'$ counts the pairs of
$\Pi_{\mathrm{res}}$ and vanishes on $M$. Hence
$E(H(r(\Pi)))=M\cup E(H_{\mathrm{res}})$, and (6.4) factors as
\[
 A_\zeta(r(\Pi))=\prod_{e\in M}g(j_e)\cdot
 \Bigl(\prod_{a,b}g(r'_{ab})\Bigr)2^{\beta(M\cup H_{\mathrm{res}})}.
\]
For fixed $(M,j,S)$, \eqref{eq:exposed-residual-factorization} turns the sum
over $\Pi_{\mathrm{res}}\in\mathcal E(M,j)$ into
$(n)_J^{-1}\prod_{e\in M}g(j_e)\,\mathcal A(M,j)$; summing over
$S\in\mathfrak F(M,j)$ gives $w_{\mathrm{hi}}(M,j)\mathcal A(M,j)$. The
bijection therefore yields
\eqref{eq:exact-attachment-decomposition} and counts every full matching
once.
\end{proof}

\subsection{Quadratic bounds for residual cell activities}
\label{app:quadratic-activities}

\begin{proof}[Proof of Lemma~\ref{lem:q-quadratic}]
Fix a cell and write $\theta=\theta_{ab}$. If $\theta=0$, then
$q_{ab}=0$, so the claim holds. Assume $\theta>0$. If $R_0<3$, the
higher-reward sum is empty, so $q_{ab}=\theta^2/2$ and the claim is immediate.
We may therefore assume $R_0\ge3$. To control the contribution from higher
multiplicities, use $0\le\Delta_x\le g(x)$ and bound
\[
  a_x:=g(x)\frac{\theta^x}{x!},
  \qquad x\ge3.
\]
For $x\ge3$,
\[
  \frac{a_{x+1}}{a_x}
  =
  \frac{2^x\theta}{x+1},
  \qquad
  \frac{a_{x+2}/a_{x+1}}{a_{x+1}/a_x}
  =
  2\frac{x+1}{x+2}>1.
\]
The successive ratios increase, so $(a_x)$ is log-convex. Its maximum on
$[3,R_0]$ is therefore attained at an endpoint, giving
\[
  \lambda_{ab}
  \le
  R_0(a_3+a_{R_0}).
\]
Now $a_3=(2/3)\theta^3$. Since
$R_0\theta=O(U^3 2^{-U/3})$, one has $R_0a_3\le\theta^2$ for all sufficiently
large $U$.

For the other endpoint, using
$R_0=\lfloor U/2\rfloor$ and
$\theta\le\mathrm e U^2 2^{-U/3}$ gives
\[
  \log_2\theta
  \le
  \log_2\mathrm e+2\log_2U-\frac U3.
\]
For sufficiently large $U$, one has $R_0\ge3$ and
$g(R_0)=2^{\binom{R_0}2-1}$. Hence
\[
\begin{aligned}
  \log_2\!\left(\frac{a_{R_0}}{\theta^2}\right)
  &=
  \binom{R_0}2-1+(R_0-2)\log_2\theta-\log_2(R_0!)\\
  &\le
  \left(\frac{U^2}{8}+O(U)\right)
  -\left(\frac{U^2}{6}+O(U)\right)
  +O(U\log U)\\
  &=
  -\frac{U^2}{24}+O(U\log U).
\end{aligned}
\]
The positive quadratic term comes from the growth of $g(R_0)$; the larger
negative term comes from the small intensity in $\theta^{R_0-2}$. The
factorial only improves the bound. This estimate is uniform for
$0<\theta\le\mathrm e U^2 2^{-U/3}$. Adding $\log_2R_0=O(\log U)$ does not
change the leading negative term, so $R_0a_{R_0}\le\theta^2$ for all
sufficiently large $U$. Combining the two endpoint bounds gives
$\lambda_{ab}\le2\theta^2$, and hence $q_{ab}\le(5/2)\theta^2$.

Choose an integer $U_0$ beyond which the preceding two
endpoint bounds hold. For $1\le U<U_0$, define
\[
 q_U(\theta):=\frac{\theta^2}{2}
   +\sum_{x=3}^{\lfloor U/2\rfloor}
      \Delta_x\frac{\theta^x}{x!},
 \qquad
 \Theta_U:=\mathrm e U^2 2^{-U/3},
\]
where an empty sum is zero. The quotient $q_U(\theta)/\theta^2$ extends
continuously to $\theta=0$ with value $1/2$. Therefore
\[
 C_{\mathrm{small}}
 :=\max_{1\le U<U_0}\ \max_{0\le\theta\le\Theta_U}
     \frac{q_U(\theta)}{\theta^2}<\infty.
\]
Taking $C_0=\max\{5/2,C_{\mathrm{small}}\}$ proves the lemma for every $U$.
\end{proof}

\section{Amplification and failure probabilities}
\label{app:amplification}

This appendix proves the simultaneous leftover coloring of
Lemma~\ref{lemma-10.1-simultaneous-leftover-coloring}, states and proves an
explicit tail form of the amplification step, and uses it to prove
Corollary~\ref{cor:phase-failure-probability}. The tail form is not needed
for the main theorem: it gives the summable failure bound, and the released
Lean proof uses it (Appendix~\ref{app:lean-reproducibility}). All thresholds
and constants are deterministic and independent of the phase.

\subsection{Proof of the simultaneous leftover coloring}
\label{app:leftover-coloring-proof}

\begin{proof}[Proof of Lemma~\ref{lemma-10.1-simultaneous-leftover-coloring}]
We work in the complement \(H\) of \(G_n\), where a clique is an
independent set in the original graph. We will find such a clique in
every sufficiently large vertex set. Put
\(u_0=\lceil n^{1/4}\rceil\). For a fixed \(u_0\)-set, (1.5) bounds the
probability that its \(H\)-edge density is below \(1/4\) by
\(\exp\{-\binom{u_0}{2}/16\}\). There are
\(\binom n{u_0}\) such sets, and for all sufficiently large \(n\),

\[
 \binom n{u_0}\exp\!\left(-\frac{\binom{u_0}{2}}{16}\right)
 \le\exp\!\left\{u_0\log(\mathrm{e} n/u_0)
                 -\frac{u_0(u_0-1)}{32}\right\}
 \le e^{-c\sqrt n}
\]
for an absolute \(c>0\): the positive term is \(O(n^{1/4}\log n)\),
whereas \(u_0(u_0-1)=(1+o(1))\sqrt n\).
Thus, outside an event of probability at most
\(e^{-c\sqrt n}\), every \(u_0\)-set has density at least \(1/4\).
For edge counts, the subscript specifies the graph and the argument
specifies the vertex set. If \(S\) has size \(s\ge u_0\), double-counting each edge of
\(H[S]\) over the \(u_0\)-subsets containing it gives
\[
 \frac{1}{\binom{s}{u_0}}
 \sum_{\substack{T\subseteq S\\|T|=u_0}}
 \frac{e_H(T)}{\binom{u_0}{2}}
 =\frac{e_H(S)}{\binom{s}{2}}.
\]
Therefore every larger set also has density at least \(1/4\).

Let $S_0$ be any vertex set of size at least \(n^{1/3}\). While
$|S_t|\ge u_0$, choose $v_t\in S_t$ of maximum degree in $H[S_t]$ and retain
its neighbors in the current set. Define
$S_{t+1}:=N_H(v_t)\cap S_t$, and put $s_t:=|S_t|$.
Density at least \(1/4\) implies
$|N_H(v_t)\cap S_t|\ge(s_t-1)/4$. Hence

\begin{equation}
 s_{t+1}\ge(s_t-1)/4,\qquad s_t\ge4^{-t}s_0-1/3.         \label{eq:greedy-size-recursion}
\end{equation}

The sets are nested and $S_{t+1}\subseteq N_H(v_t)$. Thus, whenever $j>t$,
\[
  v_j\in S_j\subseteq S_{t+1}\subseteq N_H(v_t).
\]
Every later choice is therefore adjacent in $H$ to every earlier one, so the
chosen vertices form a clique in $H$. Since \(s_0\ge\lceil n^{1/3}\rceil\), for every
\(0\le t\le\lfloor {\log n}/(13\log4)\rfloor\), \eqref{eq:greedy-size-recursion} gives
\[
 s_t\ge n^{1/3-1/13}-\frac13
      =n^{10/39}-\frac13
      \ge n^{1/4}
\]
for all sufficiently large \(n\). Since \(s_t\) is an integer, this gives
\(s_t\ge u_0\), so all these iterations are permitted. Thus the procedure
constructs at least \(c_{\mathrm{ind}}\log n\) vertices for an absolute
\(c_{\mathrm{ind}}>0\), and these vertices form an independent set in
\(G_n\).

On the same density event, take any \(S\) and repeatedly remove one of
these independent sets. Assign it a new color, and stop when fewer than
\(n^{1/3}\) vertices remain. Color those vertices singly. Each removed
set has at least \(c_{\mathrm{ind}}\log n\) vertices, so this uses at most

\[
 \frac{|S|}{c_{\mathrm{ind}}\log n}+n^{1/3}.
\]

After enlarging the absolute constant $C_0$, this is (10.3), simultaneously for
every \(S\).
\end{proof}

\subsection{An explicit tail bound for amplification}
\label{app:amplification-tail}

The mean bound in Section~10 gives the high-probability conclusion. Here we
retain a free tail parameter in the concentration bound, as needed for the
summable failure estimate.
\begin{lemma}[Amplification by vertex exposure]
\label{lemma-10.2-amplification-from-a-seed}

There are absolute constants $C,c>0$ and a deterministic sequence
$\varepsilon_n^{\mathrm{left}}\ge0$ such that
$\varepsilon_n^{\mathrm{left}}\le e^{-c\sqrt n}$ for all sufficiently large
$n$, and the following holds. Suppose deterministic
integers $k_n\ge0$ and deterministic reals $\Lambda_n\ge0$ satisfy, for all
sufficiently large $n$,

\begin{equation}
 \Prob{\zeta(G_n)\le k_n}\ge e^{-\Lambda_n}.             \label{eq:amplification-tail-assumption}
\end{equation}

For every deterministic choice $r=r(n)>0$, with the same $C$ and
$\varepsilon_n^{\mathrm{left}}$ independent of $k_n,\Lambda_n,r$, the bound

\begin{equation}
\begin{split}
 \mathbb{P}\!\Bigg(\zeta(G_n)>k_n+C\bigg(
 &\frac{\sqrt{n\Lambda_n}+\sqrt{nr}}{\log n}
   +n^{1/3}+1\bigg)\Bigg)\\
 &\le e^{-r}+\varepsilon_n^{\mathrm{left}}.               \end{split}
\label{eq:amplification-tail-bound}
\end{equation}
holds at every sufficiently large $n$ for which \eqref{eq:amplification-tail-assumption} holds.

\end{lemma}

\begin{proof}
Let $\mathcal D_n$ be the density event in the proof of
Lemma~\ref{lem:mean-seed-amplification}; on it, (10.3) holds for every
vertex set. Set $\varepsilon_n^{\mathrm{left}}:=\mathbb P(\mathcal D_n^c)
\le\varepsilon_n^{\mathrm{dens}}\le e^{-c\sqrt n}$ for all sufficiently
large $n$.

Let \(S_{k_n}\) be the maximum cocolorable-set size defined in
\eqref{eq:amplification-largest-set}. The vertex-exposure argument in the
proof of Lemma~\ref{lem:mean-seed-amplification} gives
\eqref{eq:amplification-uncovered-mean}, with \(k=k_n\) and
\(\Lambda=\Lambda_n\).

The lower tail with radius \(\sqrt{(n-1)r/2}\) gives, outside an event of
probability at most \(e^{-r}\),

\begin{equation}
 n-S_{k_n}\le
 \sqrt{(n-1)\Lambda_n/2}+\sqrt{(n-1)r/2}.                \label{eq:amplification-uncovered-tail}
\end{equation}

\Needspace{6\baselineskip}
Choose a maximizing set and its complement \(V_{\mathrm{left}}\) as in
\eqref{eq:amplification-leftover-coloring}.

The leftover set $V_{\mathrm{left}}$ depends on the graph. The event
$\mathcal D_n$ nevertheless controls its coloring, because it controls every
vertex set simultaneously. Intersect this event with the event in
\eqref{eq:amplification-uncovered-tail}, and apply (10.3) with
$|V_{\mathrm{left}}|=n-S_{k_n}$ in
\eqref{eq:amplification-leftover-coloring}. After enlarging the absolute
constant $C$, the additional number of parts has the form in
\eqref{eq:amplification-tail-bound}. A union bound gives exceptional
probability at most $e^{-r}+\varepsilon_n^{\mathrm{left}}$.
\end{proof}

\subsection{Proof of the summable failure bound}
\label{app:failure-probability-proof}

\begin{proof}[Proof of Corollary~\ref{cor:phase-failure-probability}]
Write \(\ell=\log\log n\). Apply
Lemma~\ref{lemma-10.2-amplification-from-a-seed} to (10.2) with \(r=2L\).
Since \(\Lambda_n=O(n^{3/4}L)\), its additional class allowance
\[
 t_n:=C\left(
 \frac{\sqrt{n\Lambda_n}+\sqrt{2nL}}L+n^{1/3}+1
 \right)
 =O(n^{7/8}L^{-1/2})=o(n\ell/L^4)
\]
satisfies
\[
 \mathbb P(\zeta(G_n)>k_{\mathrm{co}}+t_n)
 \le n^{-2}+e^{-c\sqrt n}.
\]
Here the second term is the simultaneous leftover-coloring failure
probability in Lemma~\ref{lemma-10.1-simultaneous-leftover-coloring}.

For the chromatic lower event, the uniform expansion (2.3) gives
\[
 \log\mu_{\alpha+2}
 =-L+\beta\ell+(\delta_n-1)(L-\ell)+O(1),
 \qquad \mu_{\alpha+2}\le C_{\mathrm{cap}}w_n
\]
for a phase-independent constant \(C_{\mathrm{cap}}\). In (4.5), the
estimates for \(a,b\) imply \(e^a=O(1)\) and
\(e^b-1=O(L^{5/2}/\sqrt n)\). Thus
\[
 \mathbb P(\chi(G_n)\le k_\chi^-)
 \le C\left(w_n^2+\frac{w_nL^{5/2}}{\sqrt n}\right)+e^{-c_1L^3}.
\]
Equations (5.11), (4.1), and (5.13), together with the bound on \(t_n\),
give, after one phase-independent threshold,
\[
 k_\chi^- - k_{\mathrm{co}}-t_n
 \ge\frac{q^2}{4}A_5(\delta_n)\frac n{L^3}
       -C_1\frac{n\ell}{L^4}.
\]
On the chromatic lower event and the cocoloring upper event, the gap is
larger than the left-hand side. A union bound, with both exponential terms
absorbed into \(n^{-2}\), proves the displayed failure bound.
Since \(0\le\delta_n<1\) and \(L/n<1\),
\[
 w_n\le\frac{L^\beta}{n},\qquad
 w_n^2+\frac{w_nL^{5/2}}{\sqrt n}+n^{-2}
 =O\!\left(n^{-3/2}L^{\beta+5/2}\right).
\]
This envelope is summable. Finally, for \(c<c_5\), the positive margin
\(c_5-c\le\frac{q^2}{4}A_5(\delta_n)-c\) absorbs \(C_1\ell/L\), proving the
assertion for the uniform bound.
\end{proof}

\section{The Lean theorem and reproducibility records}\label{app:lean-reproducibility}
A companion Lean~4 development~\citep{petkov-erdos625-lean} kernel-checks a
weaker full-sequence consequence of the uniform bound in the main
theorem, with formal coefficient
$((\log 2)^2/4)\log(200/153)$. The top-level declaration is
\texttt{Erdos625.\allowbreak erdos625}; its type is
\texttt{Erdos625.\allowbreak Erdos625Statement}. The following results are
proved in this manuscript but lie outside the scope of the released Lean
theorem: the stronger five-size coefficient, the phase-resolved refinement
involving $A_5(\delta_n)$, the quantified remainder in
Remark~\ref{rem:phase-remainder-rate}, and the failure-probability estimate in
Corollary~\ref{cor:phase-failure-probability} together with its almost-sure
consequence in Corollary~\ref{cor:almost-sure-gap}.
The Lean proof uses the tail bound of
Lemma~\ref{lemma-10.2-amplification-from-a-seed}. The mean-positive-excess
bound in Lemma~\ref{lem:mean-seed-amplification} is proved here but is not
asserted as a separately released Lean theorem.
The formal source revision \texttt{824e4b6} recorded in
\citet{petkov-erdos625-lean} fixes Lean and Mathlib at version~4.31.0 and
contains the modular sources together with their generated single-file
closure. A separate replay revision records the trust-boundary check and a
replay in a clean environment; the commands, checksums and transcripts needed
to reproduce that replay are recorded in the \href{https://github.com/SamPetkov/Erdos/blob/31bbe00c529a996bdb61b880120d71240172d18f/625/formalization/replay/aristotle-03e8124e-c359-4219-8b37-442060dae209/README.md}{replay record for the released theorem}.
A later revision of the same repository,
\href{https://github.com/SamPetkov/Erdos/tree/b654d1ca6eaddaf3025fe4d487628ec4d05a21b9/625/formalization}{\texttt{b654d1c}},
also kernel-checks the entropy bounds of
Lemmas~\ref{lem:four-size-entropy-certificate}, \ref{lem:entropy-gap}
and~\ref{lem:phase-gain-sharp}, including $A_5(\delta)>\gamma_5$ and
$A_5(\delta)>\log(169/100)$ for every $\delta\in[0,1]$; the last bound is
\texttt{Erdos625.\allowbreak FiveSizePhaseGainCertificate.\allowbreak
log\_169\_div\_100\_lt\_phaseGain5}. The released
theorem \texttt{Erdos625.\allowbreak erdos625} does not use these lemmas, and
its coefficient is unchanged.

\end{document}